\documentclass[11pt, leqno]{amsart}

\usepackage{graphicx,xcolor}
\usepackage{amssymb,amsmath,amsthm,amscd}
\usepackage{mathrsfs}
\usepackage{enumerate}
\usepackage{upgreek}
\usepackage{lscape}
\usepackage[all]{xy}
\usepackage{verbatim}
\usepackage{tikz}
\usepackage{tikz-cd}
\tikzset{dynkdot/.style={circle,draw,scale=.38}}
\usepackage[
bookmarks=false,
colorlinks=true,
debug=true,
naturalnames=true,
pdfnewwindow=true,
citecolor=blue,
linkcolor=blue]{hyperref}
\usepackage{comment}

\numberwithin{equation}{section}

\newcommand{\arxiv}[1]{\href{http://arxiv.org/abs/#1}{\texttt{arXiv:#1}}}

\newtheorem{theorem}{Theorem}[section]
\newtheorem{proposition}[theorem]{Proposition}
\newtheorem{corollary}[theorem]{Corollary}
\newtheorem{conjecture}[theorem]{Conjecture}
\newtheorem{lemma}[theorem]{Lemma}

\theoremstyle{definition}
\newtheorem{definition}[theorem]{Definition}

\newtheorem{example}[theorem]{Example}
\newtheorem{remark}[theorem]{Remark}
\newtheorem*{convention}{Convention}

\newcommand{\lf}{[\hspace{-0.3ex}[}
\newcommand{\rf}{]\hspace{-0.3ex}]}
\newcommand{\pprime}{\prime\prime}

\newcommand{\Z}{\mathbb{Z}}
\newcommand{\Q}{\mathbb{Q}}
\newcommand{\C}{\mathbb{C}}
\newcommand{\kk}{\mathbf{k}}
\newcommand{\de}{\operatorname{\mathfrak{d}}}

\newcommand{\Aut}{\mathop{\mathrm{Aut}}\nolimits}

\newcommand{\zero}{\mathop{\mathrm{zero}}\nolimits}

\newcommand{\g}{\mathfrak{g}}

\newcommand{\sfg}{\mathsf{g}}
\newcommand{\sfW}{\mathsf{W}}
\newcommand{\sfP}{\mathsf{P}}
\newcommand{\sfQ}{\mathsf{Q}}
\newcommand{\calQ}{\mathcal{Q}}
\newcommand{\hsfQ}{\widehat{\sfQ}}
\newcommand{\sfR}{\mathsf{R}}
\newcommand{\hsfR}{\widehat{\sfR}}
\newcommand{\al}{\alpha}
\newcommand{\be}{\beta}
\newcommand{\uxi}{\upxi}

\newcommand{\bfi}{\boldsymbol{i}}
\newcommand{\res}{\mathop{\mathrm{res}}\nolimits}

\newcommand{\hD}{\widehat{\Delta}}
\newcommand{\Rep}{\mathop{\mathrm{Rep}}}
\newcommand{\calD}{\mathcal{D}}
\newcommand{\Ind}{\mathop{\mathrm{Ind}}}

\newcommand{\im}{\imath}
\newcommand{\jm}{\jmath}

\newcommand{\id}{\mathrm{id}}
\newcommand{\wt}{\mathrm{wt}}
\newcommand{\wtc}{\wt^\circ}
\newcommand{\tvee}{\widetilde{\vee}}

\newcommand{\Ds}{\Delta^{\!\sigma}}
\newcommand{\hDs}{\widehat{\Delta}^{\sigma}}
\newcommand{\btau}{\breve{\tau}}

\newcommand{\hI}{\widehat{I}}
\newcommand{\bphi}{\bar{\phi}}
\newcommand{\bGamma}{\bar{\Gamma}}

\newcommand{\diag}{\mathop{\mathrm{diag}}\nolimits}
\newcommand{\tC}{\widetilde{C}}
\newcommand{\tc}{\widetilde{c}}

\newcommand{\lr}[1]{\langle #1 \rangle }
\newcommand{\prt}[2]{ \left( \begin{matrix} #1 \\ #2 \end{matrix}\right) }

\newcommand{\calY}{\mathcal{Y}}
\newcommand{\hg}{\widehat{\g}}
\newcommand{\scrN}{\mathscr{N}}

\newcommand{\af}{\mathrm{af}}
\newcommand{\qs}{q_{s}}
\newcommand{\scrC}{\mathscr{C}}
\newcommand{\scrD}{\mathscr{D}}

\newcommand{\fkD}{\mathfrak{D}}
\newcommand{\Irr}{\mathop{\mathrm{Irr}}}
\newcommand{\calM}{\mathcal{M}}
\newcommand{\calK}{\mathcal{K}}

\newcommand{\rd}{{}^*\mspace{-3mu}}

\newcommand{\evt}{\mathrm{ev}_{t=1}}

\newcommand{\univ}{\mathrm{univ}}
\newcommand{\norm}{\mathrm{norm}}
\newcommand{\seteq}{\mathbin{:=}}
\newcommand{\Deg}{\mathop{\mathrm{Deg}}\nolimits}

\newcommand{\bfD}{\mathbb{D}}

\newenvironment{rouge}{\color{black}}{}

\newcommand{\ber}{\begin{rouge}}
\newcommand{\er}{\end{rouge}}

\newenvironment{blue}{\relax\color{black}}{\hspace*{.5ex}\relax}

\newcommand{\beb}{\begin{blue}}
\newcommand{\eb}{\end{blue}}

\title[Q-data and untwisted quantum affine algebras] {Q-data and representation theory of untwisted
quantum affine algebras}

\author[R.~Fujita]{Ryo Fujita}
\address[R.~Fujita]{Institut de Math\'{e}matiques de Jussieu-Paris Rive Gauche, IMJ-PRG, Universit\'{e} de Paris, B\^{a}timent Sophie Germain, F-75013, Paris, France}
\email{ryo.fujita@imj-prg.fr}

\author[S.-j.~Oh]{Se-jin Oh}
\address[S.-j.~Oh]{Ewha Womans University Seoul, 52 Ewhayeodae-gil, Daehyeon-dong, Seodaemun-gu, Seoul, South Korea}
\email{sejin092@gmail.com}
\urladdr{https://sites.google.com/site/mathsejinoh/}

\date{\today}

\thanks{R.\ F.\ was supported in part by
Grant-in-Aid for JSPS Research Fellow JP18J10669, by JSPS Grant-in-Aid for Scientific Research (B) JP19H01782, (A) JP17H01086 and also by JSPS Overseas Research Fellowships (during the revision).}
\thanks{S.-j.\ Oh was supported by
the Ministry of Education of the Republic of Korea and the National Research Foundation of Korea (NRF-2019R1A2C4069647).}

\begin{document}

\begin{abstract}
For a complex finite-dimensional simple Lie algebra $\g$, we
introduce the notion of Q-datum, which generalizes the notion of
a Dynkin quiver with a height function from the viewpoint of Weyl group
combinatorics. Using this notion, we develop a unified  theory describing
the twisted Auslander-Reiten quivers and the twisted adapted classes introduced in~\cite{OS19b}
with an appropriate notion of the generalized Coxeter elements. 
As a consequence, we obtain a combinatorial formula
expressing the inverse
of the quantum Cartan matrix of $\g$,
which generalizes the result of \cite{HL15} in the simply-laced case.
We also find several applications of our combinatorial theory of Q-data to the finite-dimensional representation theory of 
the untwisted quantum affine algebra of $\g$. 
In particular, 
in terms of Q-data and the inverse of the quantum Cartan matrix,
(i) we give an alternative description of the block decomposition results due to \cite{CM05, KKOP20b},
(ii) we present a unified (partially conjectural) formula of the denominators of the normalized  R-matrices between all the Kirillov-Reshetikhin modules, 
and (iii) we compute the invariants $\Lambda(V,W)$ and $\Lambda^\infty(V, W)$
introduced in~\cite{KKOP20}
for each pair of simple modules $V$ and $W$.

\end{abstract}

\maketitle

\tableofcontents

\section*{Introduction}

\subsection{}
For a complex finite-dimensional simple Lie algebra $\g$, 
let $U'_q(\hg)$ denote its untwisted quantum affine algebra,
where $q$ is a generic quantization parameter.
Originally, it was introduced in the middle of the 80s in the context of 
solvable lattice models in order to give a systematic way to construct R-matrices, i.e., (trigonometric) solutions for the quantum Yang-Baxter equation.
In this paper, we are interested in the rigid monoidal abelian category $\scrC$ formed by finite-dimensional modules over $U'_q(\hg)$.
The structure of the category $\scrC$ is quite rich.
In particular, it is neither semisimple as an abelian category nor braided as an monoidal category. 
Due to its rich structure,
it have been intensively studied 
in fruitful connections with many research
areas of mathematics and physics including quantum groups, statistical mechanics, cluster algebras, dynamical systems, geometry of quiver varieties etc (see
for example \cite{AK97,CP,FR99,HL15,HL16,IKT12, Nakajima01q}).
Our main motivation is to obtain a better understanding of the structure of the category $\scrC$. 

Here we remark that the \emph{quantum Cartan matrix} $C(z)$ of $\g$ (see Subsection~\ref{ssec: Cartan} for its definition), or rather its inverse $\tC(z) := C(z)^{-1}$,
 has been playing an important role 
 as a combinatorial key ingredient in the study of the category $\scrC$. For instance, it appeared in:
 \begin{itemize}
 \item the factorization formula of the universal R-matrix due to~\cite{KT94, Damiani98},
 \item the study of the $q$-characters (natural quantum affine analog of the usual characters) of modules in $\scrC$ and 
 the deformed $\mathcal{W}$-algebras due to e.g.~\cite{FR99, FM01},
 \item the construction of the quantum Grothendieck ring $K_t(\scrC)$, which is a $t$-deformation of the usual Grothendieck ring $K(\scrC)$, due to 
 \cite{Nakajima04, VV03, Hernandez04}.
 \end{itemize}   
 In this paper, we particularly focus on this combinatorial ingredient. Roughly speaking, 
 we establish a unified formula computing the inverse $\tC(z)$ of the quantum Cartan matrix
 and apply it to the study of the category $\scrC$ especially for non-simply-laced cases. 
 
 \subsection{}  \label{intro: HL15}
 When $\g$ is of simply-laced type (i.e., of types $\mathsf{ADE}$), 
 the category $\scrC$ has an intimate connection with the representation theory of a Dynkin quiver $Q$ 
of the same type as $\g$.
 Such a remarkable connection was initially established in the seminal paper \cite{HL15} by D.~Hernandez and B.~Leclerc. 
 In that paper, it was proved that
 the quantum Grothendieck ring $K_t(\scrC^0)$
of a ``skeleton'' rigid monoidal subcategory $\scrC^0 \subset \scrC$ is isomorphic to
the derived Hall algebra $\mathrm{DH}(Q)$ of the category $\Rep(Q)$ of representations of a Dynkin quiver $Q$.   
Under this isomorphism, the cohomological degree shift $[1]$ in the derived category $D^b(\Rep(Q))$ 
corresponds to the right dual functor $\scrD$ on $\scrC^0$.    
In addition, restricting the isomorphism $K_t(\scrC^0) \simeq \mathrm{DH}(Q)$ to the subalgebras corresponding to the heart $\Rep(Q) \subset D^b(\Rep(Q))$
of the standard $t$-structure yields an isomorphism 
\begin{equation} \label{eq: HL}
K_t(\scrC_Q) \simeq U_{t}^{+}(\g),
\end{equation} 
where $\scrC_Q \subset \scrC^0$ is a certain monoidal subcategory defined by using the Auslander-Reiten (AR) quiver $\Gamma_Q$ of the category 
$\Rep(Q)$, and $U_{t}^{+}(\g)$ is the positive half of the quantized enveloping algebra $U_t(\g)$ of the finite-dimensional Lie algebra $\g$. 
(Strictly speaking, the definition of the subcategory $\scrC_Q$ and hence the isomorphism \eqref{eq: HL} actually depend not only on the quiver $Q$
but also on the choice of a height function $\xi$ on it.)   

To establish these isomorphisms, Hernandez-Leclerc~\cite{HL15} gave a homological interpretation of the inverse $\tC(z)$ of the quantum Cartan matrix 
in terms of the derived category $D^b(\Rep(Q))$. 
This homological interpretation implies that $\tC(z)$ can be computed from the AR quiver $\Gamma_Q$.
Moreover,  it also yields a combinatorial formula expressing $\tC(z)$ in terms of the action of the Coxeter element 
$\tau_Q$ adapted to $Q$ on the set of roots of $\g$.
From these results, we can deduce some useful properties of the coefficients of $\tC(z)$ such as periodicity and positivity.

\subsection{} \label{intro: Q-datum}
We aim to extend the above results by~\cite{HL15} to general $\g$ including non-simply-laced cases. 
Now let $\sfg$ be another simple Lie algebra of simply-laced type whose Dynkin diagram $\Delta$ is related to that of $\g$ via the folding 
with respect to a graph automorphism $\sigma$ on $\Delta$. When $\g$ is of simply-laced type, 
we understand that $\sigma$ is the identity and $\g = \sfg$.

Recall that the AR quiver $\Gamma_Q$ has another combinatorial meaning (due to~\cite{Bedard99})
as the Hasse quiver of the convex partial ordering among the positive roots
arising from the commutation class of reduced words for the longest element of the Weyl group adapted to $Q$. 
Generalizing this combinatorial aspect of the AR quivers, U.~R.~Suh and the second named author~\cite{OS19b}
defined \emph{the twisted AR quivers} to be 
the Hesse quivers of the convex ordering among positive roots of $\sfg$ (not $\g$)
arising from another kind of commutation classes  
called \emph{$\sigma$-adapted classes} (or \emph{twisted adapted classes}).   
With this notion, they introduced a nice subcategory of $\scrC^0$ for $\g$ attached to each $\sigma$-adapted class, 
which serves as a generalization of Hernandez-Leclerc's subcategory $\scrC_Q$ in the simply-laced case.   

Having this construction, it is natural to expect that 
we can compute the inverse $\tC(z)$ of the quantum Cartan matrix of general $\g$ from the twisted AR quivers.  
In the former half of this paper, we realize this expectation. 

In order to make our formulation unified, 
we introduce a new combinatorial notion of \emph{a Q-datum}.
By definition, a Q-datum for $\g$ is a triple $\calQ = (\Delta, \sigma, \uxi)$,
where $(\Delta, \sigma)$ is as above and $\uxi \colon \Delta_0 \to \Z$ is
a function on the vertex set $\Delta_0$, which we call \emph{a height function on $(\Delta, \sigma)$}, subject to certain axioms.    
When $\g$ is of simply-laced type ($\sigma = \id$), a height function $\uxi$ on $(\Delta, \id)$ is the same as a usual height function on a unique Dynkin quiver $Q$.   
In this sense, the notion of a Q-datum is a combinatorial generalization of the notion of a Dynkin quiver with  a height function. 

For a Q-datum $\calQ$ for $\g$, we can define the notion of \emph{$\calQ$-adaptedness} for reduced expressions of elements in the Weyl group $\sfW$ of $\sfg$
as an analog of the $Q$-adaptedness to a usual quiver $Q$.
Then we prove that the set $[\calQ]$ of reduced words of the longest element of $\sfW$ which are adapted to $\calQ$ 
forms a $\sigma$-adapted class and 
that the assignment $\calQ \mapsto [\calQ]$ gives a bijection between the set of Q-data for $\g$ (up to adding constants to height functions)
and the set of $\sigma$-adapted classes.
Moreover, we introduce \emph{the generalized Coxeter element} $\tau_\calQ$ as a specific element of $\sfW \rtimes \langle \sigma \rangle$
and construct the corresponding twisted AR quiver $\Gamma_\calQ$ by using $\tau_\calQ$ only. 
When $\sigma = \id$, they respectively coincide with the usual Coxeter element $\tau_Q$ and the AR quiver $\Gamma_{Q}$ in~\ref{intro: HL15}.  
Thus the notion of Q-datum 
yields a unified description of the $\sigma$-adapted classes and their twisted AR quivers. 

By means of these gadgets, we obtain the desired combinatorial formula of $\tC(z)$ for general $\g$
in terms of the action of the generalized Coxeter element 
$\tau_\calQ$ on the set $\sfR$ of roots of $\sfg$ (Theorem~\ref{theorem: combinatorial formula}),
which generalizes the above formula by \cite{HL15} in \ref{intro: HL15}.  
From this formula, we deduce several properties  of the coefficients of $\tC(z)$ for general $\g$ such as periodicity and positivity.
These results finally imply that $\tC(z)$ can be computed from the twisted AR quiver $\Gamma_{\calQ}$
associated with an arbitrary $\calQ$.
We will carry out such a computation for all non-simply-laced $\g$ with a specific choice of $\calQ$ (Subsection~\ref{subsection: computation}).

\subsection{}  
In the latter half of this paper,  we pursue several applications of the above combinatorial theory of Q-data
to the study of the category $\scrC$.  

For each Q-datum $\calQ$ for $\g$, we define a certain monoidal subcategory $\scrC_{\calQ}$
of the category $\scrC^0$ (Subsection \ref{Subsection: CQ}). 
From the construction, it coincides with the subcategory introduced in~\cite{OS19a} (mentioned above in \ref{intro: Q-datum})
attached to the $\sigma$-adapted class $[\calQ]$. 
Then we observe that the duals of simple objects in $\scrC_{\calQ}$, i.e., the set
$\{ \scrD^{k}(V) \mid V \in \Irr \scrC_{\calQ}, k \in \Z \}$ monoidally generates the whole category $\scrC^{0}$.
In this sense, our subcategory $\scrC_{\calQ}$ plays a role of a ``heart'' of the category $\scrC^{0}$, just as 
the subcategory $\scrC_{Q}$ does in the simple-laced case.    

In addition, we can assign to each simple module $V$ in $\scrC^0$
an element $\wt_{\calQ}(V)$ of the root lattice $\sfQ = \Z \sfR$ of 
the simply-laced Lie algebra $\sfg$, which
we call \emph{the $\calQ$-weight of $V$} (Definition~\ref{def: new-weights}).
It turns out that the assignment $V \mapsto \wt_{\calQ}(V)$ is comparable with the elliptic character considered in~\cite{CM05} and therefore 
we obtain a block decomposition of the category $\scrC^0$ labeled by $\sfQ$.  
The resulting block decomposition is in turn the same as the one recently obtained in~\cite{KKOP20b}.
Thus the notion of $\calQ$-weight connects these two known block decomposition results.

\subsection{}
We can also apply our combinatorial theory of Q-data and the inverse of the quantum Cartan matrix
to the study of R-matrices and related invariants. 

For any pair $(V, W)$ of simple modules in $\scrC$, we have
the normalized R-matrix $R^{\norm}_{V, W}(z)$, which is a unique non-trivial intertwining operator
between the tensor products $V \otimes W$ and $W \otimes V$ depending rationally on
the spectral parameter $z$,
whose singularity strongly reflects the structure of these tensor product modules.
For example, assuming one of the simple modules is real (i.e., its tensor square is again simple),
we have $V \otimes W \cong V \otimes W$ if and only if both $R^{\norm}_{V, W}(z)$ and $R^{\norm}_{W,V}(z)$ have no poles at $z=1$.
Thus to compute the denominator $d_{V, W}(z)$ of $R^{\norm}_{V, W}(z)$ is of fundamental importance
in the study of the monoidal category $\scrC$. 

When $\g$ is of simply-laced type,
the present authors recently discovered in~\cite{Oh15E, Fujita19} that
the denominators between the fundamental modules are closely related to the AR quiver $\Gamma_Q$
for a Dynkin quiver $Q$. 
In particular, we find that 
these denominators are expressed in a unified way in terms of the inverse  $\tC(z)$ of the Cartan matrix as the first named author proved in \cite{Fujita19} by 
using the geometry of quiver varieties. 
In this paper, 
motivated by this result, we present a conjectural unified denominator formula between
all the Kirillov-Reshetikhin (KR) modules (Conjecture~\ref{conjecture: unified}). 
Recall that the KR modules form a family of simple modules in $\scrC$ whose $q$-characters 
satisfy the remarkable functional relations called $T$-systems~\cite{Nakajima03II, Her06} and hence play an important role in the theory of monoidal categorifications of
cluster algebras~\cite{HL10, HL16, KKOP20-3}.
Our conjectural KR denominator formula is also expressed in terms of $\tC(z)$.
We check that it holds at least for types $\mathsf{ABCDG}$
by comparing the explicit case-by-case computations of the denominators obtained in \cite{OhS19s} with the explicit computations of $\tC(z)$ obtained in this paper.  

Finally, we compute the $\Z$-valued invariants $\Lambda(V, W)$ and $\Lambda^\infty(V, W)$ 
in terms of $\tC(z)$ and the $\calQ$-weights (Subsection~\ref{subsection: degree}).
These invariants are recently introduced by M.~Kashiwara, M.~Kim, E.~Park and the second named author~\cite{KKOP20}
to establish a criterion for a monoidal subcategory of $\scrC^0$ to give a monoidal categorification of a cluster algebra.
They are defined from
the denominator $d_{V,W}(z)$ and also \emph{the universal coefficient} $a_{V, W}(z)$, i.e., the ratio of the universal R-matrix $R^{\mathrm{univ}}_{V, W}(z)$
and the normalized R-matrix $R^{\norm}_{V, W}(z)$.
Based on a formula of $a_{V, W}(z)$ due to Frenkel-Reshetikhin~\cite{FR99}, we observe that
\begin{align*}
\Lambda(V, W) & = \scrN(V, W) + \deg_{z=1} d_{V, W}(z), \\
\Lambda^\infty(V, W) &= -(\wt_{\calQ}(V), \wt_{\calQ}(W))
\end{align*}
hold for any simple modules $V, W$ in $\scrC^0$.
Here  
$(\cdot, \cdot)$ is the standard symmetric bilinear form on the root lattice $\sfQ$, and
$\scrN(V, W)$ is a skew-symmetric $\Z$-valued form defined as a certain signed sum of coefficients of $\tC(z)$,
which was
introduced by Hernandez~\cite{Hernandez04} 
 in his algebraic construction of the quantum Grothendieck ring $K_t(\scrC^0)$.
Thus the invariants $\Lambda(V, W)$ and $\Lambda^\infty(V, W)$  are also closely related to Q-data and $\tC(z)$.

\subsection{} 
We expect that our results in this paper have some nice applications in the study of the quantum Grothendieck ring $K_t(\scrC)$.
In particular, we hope to generalize the results of \cite{HL15} (especially the isomorphism \eqref{eq: HL}) to the case of general $\g$ 
and investigate an expected quantum cluster algebra structure in the quantum Grothendieck ring
using our combinatorial theory of Q-data.  We will come back to these problems in forthcoming papers.

\subsection{Organization} This paper is organized as follows.
In Section~\ref{section: Dynkin}, we give a quick review of the classical notion
of the usual Dynkin quiver with height function, associated AR quivers and
adapted commutation classes in Weyl groups.
In Section~\ref{section: Q-data}, to generalize the story in Section~\ref{section: Dynkin},
we introduce the notion of Q-datum
and develop a unified theory for twisted AR quivers and generalized Coxeter elements.
In Section~\ref{section: Cartan}, we study the inverse $\tC(z)$ of the quantum Cartan matrix
in relation with Q-data.
Up to this part, our exposition is a pure combinatorics.
In Section~\ref{section: QAffine}, after reviewing some basic facts
on the finite-dimensional representation theory of the untwisted quantum affine algebras,
we revisit the block decomposition results of the category $\scrC^0$ with the notion of $\calQ$-weights.
In the final Section~\ref{section: denominator}, we present a conjectural unified KR denominator formula and
compute the invariants $\Lambda(V, W)$ and $\Lambda^\infty(V, W)$ 
in terms of $\tC(z)$ and the $\calQ$-weights.
Appendix~\ref{section: denominator formula} is a list of case-by-case formulae of denominators of the normalized R-matrices between
KR modules known at this moment.
\subsection*{Acknowledgments}
We are deeply grateful to David Hernandez for helpful discussions and for pointing out
that some of the results in the initial version of this paper were already known in the literature.  
We also wish to thank Masaki Kashiwara, Bernhard Keller, Myungho Kim, Bernard Leclerc and Euiyong Park for stimulating discussions and comments.
Finally, we would like to thank the anonymous referees for many helpful suggestions about the exposition and the appropriate references.

\begin{convention} Throughout this paper, we keep the following conventions.
\begin{enumerate}
\item For a statement $\mathtt{P}$, we set $\delta(\mathtt{P})$ to be $1$ or $0$
according that $\mathtt{P}$ is true or not.
In particular, we set $\delta_{i,j} := \delta(i=j)$ $($Kronecker's delta$)$.
\item \ber
For an abelian category $\scrC$, 
we denote by $\Irr \scrC$ the set of simple objects
(up to isomorphisms)
and by $K(\scrC)$ the Grothendieck group of $\scrC$.
The class of an object $X \in \scrC$ is denoted by $[X] \in K(\scrC)$. 
If moreover the category $\scrC$ is monoidal with a bi-exact tensor product $\otimes$, 
$K(\scrC)$ is endowed with a ring structure by
$[X] \cdot [Y] = [X \otimes Y]$, which we call the Grothendieck ring of $\scrC$. 
\er
\end{enumerate}
\end{convention}

\section{Dynkin quivers, adapted classes and associated Auslander-Reiten quivers}
\label{section: Dynkin}

In this section, we briefly recall the classical
notion of Dynkin quivers of type $\mathsf{ADE}$
and their height functions,
which give a useful description
of the associated Auslander-Reiten quiver
and the adapted commuting classes in the Weyl group.

\subsection{Basic notation}
\label{subsection: notation1}

Let $\sfg$ be a finite-dimensional complex simple Lie algebra of type $\mathsf{ADE}$ of rank $n$.
We denote by $\Delta$ the Dynkin diagram of $\sfg$ (simply-laced) and
by $\Delta_{0}$ the set of vertices of $\Delta$.
For $i, j \in \Delta_{0}$, we write $i \sim j$ if
$i$ is adjacent to $j$ in the Dynkin diagram $\Delta$.

Let $\sfP=\bigoplus_{i \in \Delta_{0}} \Z \varpi_{i}$ denote
the weight lattice of $\sfg$,
where $\varpi_{i}$ is the $i$-th fundamental weight.
Let $\alpha_{i} = 2 \varpi_{j} - \sum_{j \sim i} \varpi_{j}$
be the $i$-th simple root and
$\sfQ = \bigoplus_{i \in \Delta_{0}} \Z \alpha_{i} \subset \sfP$
be the root lattice.
We set
$\sfP^{+} := \sum_{i \in \Delta_{0}} \Z_{\ge 0} \varpi_{i}$ and
$\sfQ^{+} := \sum_{i \in \Delta_{0}} \Z_{\ge 0} \alpha_{i}$.
Let $( \ , \ ) \colon \sfP \times \sfP \to \Q$
denote the symmetric bilinear form determined by
$(\varpi_{i}, \alpha_{j}) = \delta_{i, j}$ for $i, j \in \Delta_{0}$.
The Weyl group $\sfW$ of $\sfg$ is defined as a subgroup of $\Aut(\sfP)$
generated by the simple reflections $\{ s_{i} \}_{i \in \Delta_{0}}$ defined by
$s_{i}(\lambda) = \lambda - (\lambda, \alpha_{i})\alpha_{i}$ for $\lambda \in \sfP$.
The set of roots is defined by $\sfR := \sfW\cdot\{\alpha_{i}\}_{i \in \Delta_{0}}$.
We have the decomposition $\sfR = \sfR^{+} \sqcup \sfR^{-}$, where
$\sfR^{+} := \sfR \cap \sfQ^{+}$ is the set of positive roots
and $\sfR^{-} := -\sfR^{+}$ is the set of negative roots.

For an element $w \in \sfW$,
a sequence $(i_{1}, \ldots, i_{l})$ of elements of $\Delta_{0}$
is called a reduced word for $w$ if
it satisfies $w=s_{i_{1}}\cdots s_{i_{l}}$ and
$l$ is the smallest among all sequences with this property.
The length of a reduced word of $w$ is called the length of $w$, denoted  by $\ell(w)$.
It is well-known that there exists a unique element $w_{0} \in \sfW$ with the largest length
$\ell(w_{0}) = N := |\sfR^{+}|$.
We define an involution $i \mapsto i^{*}$
on the set $\Delta_{0}$ by the relation $w_{0}(\alpha_{i}) = -\alpha_{i^{*}}$.

\subsection{Dynkin quivers and height functions}

By definition, a Dynkin quiver $Q$ of type $\sfg$ is a quiver
whose underlying graph is equal to the Dynkin diagram $\Delta$.
A \emph{height function} on $Q$ is a function
$\xi \colon \Delta_{0} \to \Z$
satisfying the condition that
\begin{align}
\label{eq: def classical height function 1}
\xi_{i} = \xi_{j} + 1   \quad \text{if $Q$ has an arrow $i \to j$},
\end{align}
where we set $\xi_{i} := \xi(i)$ for simplicity.
Note that a height function on $Q$ is determined uniquely up to constant functions.

\begin{example} \label{example: claissical Dynkin quiver}
Here are two examples of Dynkin quivers with height functions
in type $\mathsf{A}_{5}$ and type $\mathsf{D}_{4}$ respectively.
We put the values of height functions near the vertices.
$$
Q^{(1)}=
\left(
\begin{xy}
\def\objectstyle{\scriptstyle}
\ar@{<-} (-10,0) *+!D{1} *\cir<2pt>{};
(0,0) *+!D{2} *\cir<2pt>{}="A",
\ar@{<-} "A"; (10,0) *+!D{3} *\cir<2pt>{}="B",
\ar@{->} "B"; (20,0) *+!D{2} *\cir<2pt>{}="C",
\ar@{<-} "C"; (30,0) *+!D{3} *\cir<2pt>{}="D"
\end{xy}
\hspace{6pt}
\right),
\qquad
Q^{(2)}
=\left(
\begin{xy}
\def\objectstyle{\scriptstyle}
\ar@{<-} (0,0) *+!D{1} *\cir<2pt>{};
(10,0) *+!D{2} *\cir<2pt>{}="A",
\ar@{<-} "A"; (19,4) *+!L{3} *\cir<2pt>{}
\ar@{<-} "A"; (19,-4) *+!L{3} *\cir<2pt>{}
\end{xy}
\hspace{6pt}
\right).
$$
\end{example}

We fix a function $\epsilon \colon \Delta_{0} \to \{0,1\}$
such that $\epsilon_{i} \neq \epsilon_{j}$ when $i \sim j$.
We refer to such a function $\epsilon$ as a \emph{parity function} on $\Delta_{0}$.
Since the Dynkin diagram $\Delta$ is a connected tree,
there are only two choices of parity functions and their difference is not essential.
Adding a constant if necessary, we can make a given height function $\xi$
satisfy the condition
\begin{align}
\label{eq: def classical parity}
\xi_{i} - \epsilon_{i} \in 2\Z \quad \text{for each $i \in \Delta_{0}$}.
\end{align}
Thus, without loss of generality, we require that a height function $\xi$ always
satisfy the parity condition~(\ref{eq: def classical parity}) in the rest of this section.

We say that a vertex $i \in \Delta_{0}$ is a \emph{source} of $Q$
if every arrow in $Q$ incident to $i$ has $i$ as its tail.
In terms of the height function $\xi$,
a vertex $i$ is a source of $Q$ if and only if we have $\xi_{i} > \xi_{j}$
for any $j \sim i$.
Let $i \in \Delta_{0}$ be a source of $Q$.
We denote by $s_{i}Q$
the new Dynkin quiver of the same type $\sfg$
obtained from $Q$ by reversing all arrows incident to $i$.
Moreover, for a height function $\xi$ on $Q$, we can define
the height function $s_{i}\xi$ on $s_{i}Q$ by
\begin{align}
(s_{i}\xi)_{j} := \xi_{j} - 2\delta_{i,j} \quad \text{for each $j \in \Delta_{0}$}.
\end{align}
We say that a sequence $(i_{1}, \ldots, i_{l})$ of elements of $\Delta_{0}$ is \emph{adapted to $Q$} if
$$\text{$i_k$ is a source of $s_{i_{k-1}}s_{i_{k-2}}\cdots s_{i_{1}}Q$ for all $1 \le k \le l$.}$$

Recall that a Coxeter element $\tau$ of $\sfW$
is a product of the form $\tau=s_{i_1}\cdots s_{i_{n}}$ such that $\{ i_1,\ldots,i_{n}\} = \Delta_{0}$.
All the Coxeter elements are conjugate in $\sfW$ and the order of Coxeter elements
is called the \emph{$($dual$)$ Coxeter number} of $\sfg$ and denoted by $h^\vee$.
Note that we have the relation $nh^{\vee} = 2N$.

For each Dynkin quiver $Q$,
there exists a unique Coxeter element $\tau_{Q}$, all of whose reduced words are adapted to $Q$.
For example, a sequence $(i_{1}, \ldots, i_{n})$ satisfying $\ber \Delta_{0} \er = \{i_{1}, \ldots, i_{n} \}$ and
$\xi_{i_{1}} \ge \cdots \ge \xi_{i_{n}}$ for a height function $\xi$ on $Q$
gives a reduced word for $\tau_{Q}$.
Conversely, for each Coxeter element $\tau$,
there \ber exists \er a unique Dynkin quiver $Q$ such that all reduced words for $\tau$ are adapted to $Q$.

\subsection{Associated Auslander-Reiten quivers}
\label{subsection: AR quiver}

For a Dynkin quiver $Q$,
we denote by $\Rep(Q)$ the category of finite-dimensional representations of $Q$
over the field $\C$ of complex numbers.
In this subsection, we recall a combinatorial description of the Auslander-Reiten (AR) quiver
of the category $\Rep(Q)$ and that of the derived category
$\calD_{Q}:=D^{b}(\Rep(Q))$ by using a height function $\xi$ on $Q$.
By the definition, the vertex set of the AR quiver of $\Rep(Q)$ (resp.~$\calD_{Q}$)
is the set of isomorphism classes of indecomposable objects, denoted by $\Ind \Rep(Q)$ (resp.~$\Ind \calD_{Q}$).
By the fundamental result of Gabriel~\cite{Gabriel72}, we have the canonical bijection
$\sfR^{+} \cong \Ind \Rep(Q)$ which associates each positive root $\alpha \in \sfR^{+}$ with
the class of an indecomposable representation $M_{Q}(\alpha) \in \Rep(Q)$ whose dimension vector is $\alpha$.
Furthermore we have a canonical bijection $\hsfR^+ \seteq \sfR^+ \times \Z \cong \Ind \calD_{Q}$
which associates each element $(\alpha, k) \in \hsfR^+$ with
a stalk complex $M_{Q}(\alpha)[k] \in \calD_{Q}$. Here we naturally identify the abelian category $\Rep(Q)$
with the heart of the standard $t$-structure of $\calD_{Q}$ and denote by $[k]$ the cohomological degree shift by $k$.

Using the fixed parity function $\epsilon$ in~\eqref{eq: def classical parity},
we define the \emph{repetition quiver} associated with $\Delta$
to be the quiver $\hD$ whose vertex set $\hD_{0}$
and arrow set $\hD_{1}$ are given by
\begin{align*}
\hD_{0} &:= \{ (i,p) \in \Delta_{0} \times \Z \mid p - \epsilon_{i} \in 2\Z \}, \\
\hD_{1} &:= \{ (i,p) \to (j, p+1) \mid (i,p) \in \hD_{0}, \; j \sim i \}.
\end{align*}
\ber
\begin{example} Here are some examples of the repetition quiver $\hD$.
\begin{enumerate}
    \item When $\sfg$ is of type $\mathsf{A}_{5}$, 
the repetition quiver $\hD$ is depicted as:
$$
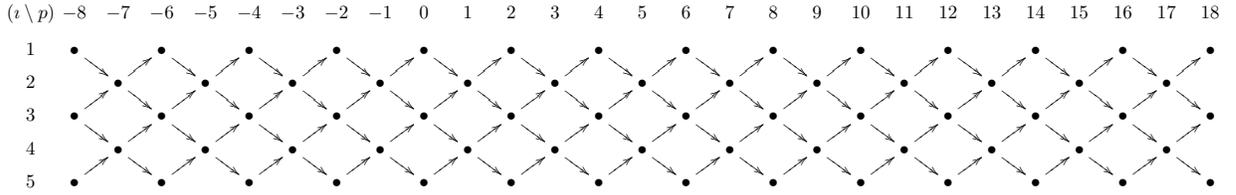
\raisebox{3mm}{
\scalebox{0.65}{\xymatrix@!C=0.5mm@R=2mm{
(\im\setminus p) & -8 & -7 & -6 &-5&-4 &-3& -2 &-1& 0 & 1& 2 & 3& 4&  5
& 6 & 7 & 8 & 9 & 10 & 11 & 12 & 13 & 14 & 15 & 16 & 17& 18 \\
1&\bullet \ar@{->}[dr]&& \bullet \ar@{->}[dr] &&\bullet\ar@{->}[dr]
&&\bullet \ar@{->}[dr] && \bullet \ar@{->}[dr] &&\bullet \ar@{->}[dr] &&  \bullet \ar@{->}[dr] 
&&\bullet \ar@{->}[dr] && \bullet \ar@{->}[dr] &&\bullet \ar@{->}[dr]  && \bullet\ar@{->}[dr] &&
\bullet\ar@{->}[dr] && \bullet\ar@{->}[dr]  && \bullet\\
2&&\bullet \ar@{->}[dr]\ar@{->}[ur]&& \bullet \ar@{->}[dr]\ar@{->}[ur] &&\bullet \ar@{->}[dr]\ar@{->}[ur]
&& \bullet \ar@{->}[dr]\ar@{->}[ur]&& \bullet\ar@{->}[dr] \ar@{->}[ur]&& \bullet \ar@{->}[dr]\ar@{->}[ur]&&\bullet \ar@{->}[dr]\ar@{->}[ur]&
&\bullet \ar@{->}[dr]\ar@{->}[ur]&&\bullet\ar@{->}[dr] \ar@{->}[ur]&& \bullet \ar@{->}[dr]\ar@{->}[ur]
&&\bullet \ar@{->}[dr]\ar@{->}[ur]&& \bullet \ar@{->}[dr] \ar@{->}[ur]&&\bullet \ar@{->}[dr]\ar@{->}[ur] & \\ 
3&\bullet \ar@{->}[dr] \ar@{->}[ur]&& \bullet \ar@{->}[dr] \ar@{->}[ur] &&\bullet\ar@{->}[dr] \ar@{->}[ur]
&&\bullet \ar@{->}[dr] \ar@{->}[ur] && \bullet \ar@{->}[dr]\ar@{->}[ur] &&\bullet \ar@{->}[dr] \ar@{->}[ur]&&  \bullet \ar@{->}[dr] \ar@{->}[ur]
&&\bullet \ar@{->}[dr] \ar@{->}[ur] && \bullet \ar@{->}[dr] \ar@{->}[ur]&&\bullet \ar@{->}[dr] \ar@{->}[ur] && \bullet\ar@{->}[dr] \ar@{->}[ur]&&
\bullet\ar@{->}[dr] \ar@{->}[ur] && \bullet\ar@{->}[dr] \ar@{->}[ur]  && \bullet\\
4&& \bullet \ar@{->}[ur]\ar@{->}[dr]&&\bullet \ar@{->}[ur]\ar@{->}[dr]&&\bullet \ar@{->}[ur]\ar@{->}[dr] &&\bullet \ar@{->}[ur]\ar@{->}[dr]&& \bullet \ar@{->}[ur]\ar@{->}[dr]
&&\bullet \ar@{->}[ur]\ar@{->}[dr]&& \bullet \ar@{->}[ur]\ar@{->}[dr] &&\bullet \ar@{->}[ur]\ar@{->}[dr]&&\bullet \ar@{->}[ur]\ar@{->}[dr]&&
\bullet \ar@{->}[ur]\ar@{->}[dr]&&\bullet \ar@{->}[ur]\ar@{->}[dr]&&\bullet \ar@{->}[ur]\ar@{->}[dr]&&\bullet\ar@{->}[ur]\ar@{->}[dr]\\
5&\bullet  \ar@{->}[ur]&& \bullet  \ar@{->}[ur] &&\bullet \ar@{->}[ur]
&&\bullet  \ar@{->}[ur] && \bullet \ar@{->}[ur] &&\bullet  \ar@{->}[ur]&&  \bullet  \ar@{->}[ur]
&&\bullet  \ar@{->}[ur] && \bullet  \ar@{->}[ur]&&\bullet  \ar@{->}[ur] && \bullet \ar@{->}[ur]&&
\bullet \ar@{->}[ur] && \bullet \ar@{->}[ur]  && \bullet }}}
$$
\item
When $\sfg$ is of type $\mathsf{D}_{4}$, 
the repetition quiver $\hD$ is depicted as:
$$
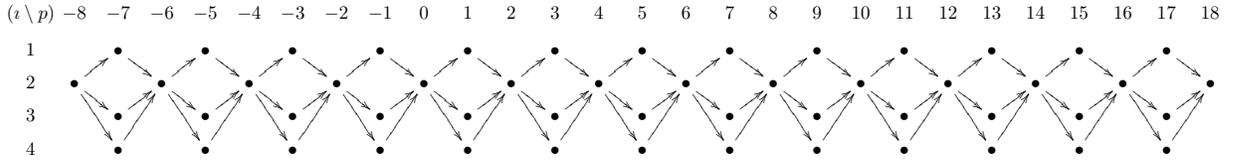
\raisebox{3mm}{
\scalebox{0.65}{\xymatrix@!C=0.5mm@R=2mm{
(\im\setminus p) & -8 & -7 & -6 &-5&-4 &-3& -2 &-1& 0 & 1& 2 & 3& 4&  5
& 6 & 7 & 8 & 9 & 10 & 11 & 12 & 13 & 14 & 15 & 16 & 17& 18 \\
1&&\bullet \ar@{->}[dr]&& \bullet \ar@{->}[dr]&&\bullet \ar@{->}[dr]
&& \bullet \ar@{->}[dr]&& \bullet\ar@{->}[dr] && \bullet \ar@{->}[dr]&&\bullet \ar@{->}[dr]&
&\bullet \ar@{->}[dr]&&\bullet\ar@{->}[dr]&& \bullet \ar@{->}[dr]
&&\bullet \ar@{->}[dr]&& \bullet \ar@{->}[dr]&&\bullet \ar@{->}[dr] & \\ 
2&\bullet \ar@{->}[dr] \ar@{->}[ur]&& \bullet \ar@{->}[dr] \ar@{->}[ur] &&\bullet\ar@{->}[dr] \ar@{->}[ur]
&&\bullet \ar@{->}[dr] \ar@{->}[ur] && \bullet \ar@{->}[dr]\ar@{->}[ur] &&\bullet \ar@{->}[dr] \ar@{->}[ur]&&  \bullet \ar@{->}[dr] \ar@{->}[ur]
&&\bullet \ar@{->}[dr] \ar@{->}[ur] && \bullet \ar@{->}[dr] \ar@{->}[ur]&&\bullet \ar@{->}[dr] \ar@{->}[ur] && \bullet\ar@{->}[dr] \ar@{->}[ur]&&
\bullet\ar@{->}[dr] \ar@{->}[ur] && \bullet\ar@{->}[dr] \ar@{->}[ur]  && \bullet\\
3&& \bullet \ar@{->}[ur]&&\bullet \ar@{->}[ur]&&\bullet \ar@{->}[ur] &&\bullet \ar@{->}[ur]&& \bullet \ar@{->}[ur]
&&\bullet \ar@{->}[ur]&& \bullet \ar@{->}[ur] &&\bullet \ar@{->}[ur]&&\bullet \ar@{->}[ur]&&
\bullet \ar@{->}[ur]&&\bullet \ar@{->}[ur]&&\bullet \ar@{->}[ur]&&\bullet\ar@{->}[ur]&\\
4&& \bullet \ar@{<-}[uul]\ar@{->}[uur]&&\bullet \ar@{<-}[uul]\ar@{->}[uur]&&\bullet \ar@{<-}[uul]\ar@{->}[uur] &&\bullet \ar@{<-}[uul]\ar@{->}[uur]&& \bullet \ar@{<-}[uul]\ar@{->}[uur]
&&\bullet \ar@{<-}[uul]\ar@{->}[uur]&& \bullet \ar@{<-}[uul]\ar@{->}[uur] &&\bullet \ar@{<-}[uul]\ar@{->}[uur]&&\bullet \ar@{<-}[uul]\ar@{->}[uur]&&
\bullet \ar@{<-}[uul]\ar@{->}[uur]&&\bullet \ar@{<-}[uul]\ar@{->}[uur]&&\bullet \ar@{<-}[uul]\ar@{->}[uur]&&\bullet\ar@{<-}[uul]\ar@{->}[uur]&
}}}
$$
\end{enumerate}
\end{example}
\er
It was shown by Happel~\cite{Happel87} that the AR quiver of $\calD_{Q}$ is isomorphic to the repetition quiver $\hD$.
An explicit underlying bijection $\phi_{Q} \colon \hD_{0} \to \Ind \calD_{Q}$ between the vertex sets is given by
$$
\phi_{Q}(i, p) := \uptau^{(\xi_{i}-p)/2}M_{Q}(\gamma^{Q}_{i}),
$$
where $\uptau$ denotes the AR translation of $\calD_{Q}$ and we set
$$
\gamma_{i}^{Q} := (1-\tau_{Q})\varpi_{i} 
$$
for each $i \in \Delta_{0}$.
Here $\tau_{Q} \in \sfW$ is the Coxeter element adapted to $Q$.
Note that the representation $M_{Q}(\gamma^{Q}_{i}) \in \Rep(Q)$
is an injective hull of the simple representation associated with $i \in \Delta_{0}$
\ber since $\gamma^{Q}_{i}$ is a sum of all the simple roots $\alpha_j$ such that there exists an oriented path from $j$ to $i$ in $Q$. \er

In what follows, we regard $\phi_{Q}$ as a bijection $\phi_{Q} \colon \hD_{0} \to \hsfR^{+}$
via the above canonical bijection $\Ind \calD_{Q} \cong \hsfR^+$.
Then it has the following recursive description (\cite[\S 2.1]{HL15}):
\begin{enumerate}
\item $\phi_Q(i,\xi_i)=(\gamma^{Q}_{i},0)$ for each $i \in \Delta_{0}$.
\item If $\phi_Q(i,p)=(\beta,k)$, we have
$$
\phi_Q(i,p\pm2) = \begin{cases}
(\tau_{Q}^{\mp1}(\beta),k) & \text{if $\tau_{Q}^{\mp1}(\beta) \in \sfR^+$}, \\
(-\tau_{Q}^{\mp1}(\beta), k\pm1) & \text{if $\tau_{Q}^{\mp1}(\beta) \in \sfR^-$}.
\end{cases}
$$
\end{enumerate}
In particular, we have $\tau_{Q}^{(\xi_{i}-p)/2}(\gamma^{Q}_{i}) = (-1)^{k}\beta$ if $\phi_{Q}(i,p) = (\beta, k)$.

The repetition quiver $\hD$ satisfies the \emph{additive property}:
For $i \in \Delta_{0}$ and $l \in \Z$, we have
\begin{align}
\label{eq: classical additive property}
\tau_{Q}^l(\gamma^{Q}_{i})+ \tau_{Q}^{l+1}(\gamma^{Q}_{i}) = \sum_{j \sim i} \tau_{Q}^{l+(\xi_j-\xi_i+1)/2}(\gamma^{Q}_{j}).
\end{align}
\ber This is because we have the following mesh in the AR quiver of $\calD_{Q}$:
$$
\xymatrix@!C=30mm@R=3mm{
& \uptau^{l+(\xi_j-\xi_i+1)/2}M_{Q}(\gamma^{Q}_{j}) 
\ar@{->}[dr]& \\
\uptau^{l+1} M_{Q}(\gamma^{Q}_{i}) 
\ar@{->}[ur] \ar@{->}[dr]
& \text{\raisebox{-1.5mm}{\rotatebox{90}{$\cdots$}}}& \uptau^{l}M_{Q}(\gamma^{Q}_{i}) \\
& \uptau^{l+(\xi_{j'}-\xi_i+1)/2}M_{Q}(\gamma^{Q}_{j'}) 
\ar@{->}[ur]& 
}
$$
where $\{j, \ldots, j'\} = \{j \in \Delta_0 \mid j \sim i \}$, which corresponds to an Auslander-Reiten triangle in $\calD_{Q}$. \er

Let $\Gamma_{Q}$ be the full subquiver of $\hD$ whose vertex set $(\Gamma_{Q})_{0}$ is given by
$\phi_{Q}^{-1}(\sfR^{+} \times \{ 0 \})$.
We define the bijection $\phi_{Q, 0} \colon (\Gamma_{Q})_{0} \to \sfR^{+}$
by $\phi_{Q}(i,p) = (\phi_{Q, 0}(i,p), 0)$ for $(i,p) \in (\Gamma_{Q})_{0}$.
For $\beta \in \sfR^{+}$, we call $(i,p) = \phi_{Q, 0}^{-1}(\beta)$ the \emph{coordinate of $\beta$ in $\Gamma_{Q}$}.

In what follows, we identify the vertex set $(\Gamma_{Q})_{0}$ with the set $\sfR^{+}$ of positive roots
via the bijection $\phi_{Q,0}$.
By construction, the quiver $\Gamma_{Q}$ is isomorphic to
the AR quiver of $\Rep(Q)$ under the canonical bijection $\sfR^{+} \cong \Ind \Rep(Q)$.
In addition, we have the following explicit characterization of $\Gamma_{Q}$ in $\hD$:
\begin{equation}
\label{eq: 2-segment}
\Gamma_{Q} =
\phi_{Q}^{-1}(\sfR^{+} \times \{ 0 \}) = \{ (i, \xi_{i}-2k) \in \hD_{0} \mid
0 \le 2k < h^{\vee} + \xi_{i}-\xi_{i^{*}} \}.
\end{equation}
See~\cite[\S 6.5]{Gabriel80}.
For $d \in \Z_{>0}$, we say that a subset $S \subset \Z$ is a \emph{$d$-segment}
if $S = \{ l+kd \mid k=0,\ldots,t\}$ for some $l \in \Z$ and $t \in \Z_{\ge 0}$.
The characterization~(\ref{eq: 2-segment}) shows the
\emph{$2$-segment property} of $\Gamma_Q$.

\begin{example} For the Dynkin quivers $Q^{(1)}$ and $Q^{(2)}$ in Example~\ref{example: claissical Dynkin quiver},
the corresponding AR quivers
$\Gamma_{Q^{(1)}}$ and $\Gamma_{Q^{(2)}}$
can be depicted as follows. Here
$[a,b] := \epsilon_{a} - \epsilon_{b+1}$ and $\lr{a,\pm b} := \epsilon_a \pm \epsilon_b$
denote the positive roots defined
as in Subsections~\ref{subsubsection: computation B} and~\ref{subsubsection: computation C} below.
$$
\Gamma_{Q^{(1)}}=
\raisebox{4em}{
\scalebox{0.7}{\xymatrix@!C=2ex@R=1ex{
(i\setminus p) & -3 & -2 & -1 & 0 & 1 & 2 & 3 \\
1&&& [4,5]\ar@{->}[dr] && [1,3]\ar@{->}[dr]  \\
2&&[4]\ar@{->}[dr]\ar@{->}[ur] && [1,5]\ar@{->}[dr]\ar@{->}[ur]  && [2,3] \ar@{->}[dr] \\
3&&& [1,4]\ar@{->}[dr]\ar@{->}[ur]  && [2,5]\ar@{->}[dr]\ar@{->}[ur] && [3]   \\
4&& [1,2] \ar@{->}[dr]\ar@{->}[ur] && [2,4]\ar@{->}[dr]\ar@{->}[ur] && [3,5] \ar@{->}[ur]\ar@{->}[dr] \\
5& [1] \ar@{->}[ur] && [2] \ar@{->}[ur]&& [3,4] \ar@{->}[ur]&& [5]}}},
\qquad
\Gamma_{Q^{(2)}}=
\raisebox{3.5em}{
\scalebox{0.7}{\xymatrix@!C=2ex@R=1.5ex{
(i\setminus p) & -3 & -2 & -1 & 0 & 1 & 2 & 3 \\
1& \lr{1,-2}\ar@{->}[dr] && \lr{2,-3}\ar@{->}[dr] && \lr{1,3} \ar@{->}[dr]\\
2& & \lr{1,-3}\ar@{->}[dr]\ar@{->}[ddr]\ar@{->}[ur] && \lr{1,2}\ar@{->}[dr]\ar@{->}[ddr]\ar@{->}[ur]  && \lr{2,3} \ar@{->}[ddr] \ar@{->}[dr] \\
3&&& \lr{1,-4} \ar@{->}[ur]  && \lr{2,4} \ar@{->}[ur] && \lr{3,-4}  \\
4&&&  \lr{1,4} \ar@{->}[uur] && \lr{2,-4}\ar@{->}[uur] && \lr{3,4} \\
}}}
$$
\end{example}

For a source $i \in \Delta_{0}$ of $Q$, the quiver $\Gamma_{s_iQ}$ can be obtained from $\Gamma_{Q}$ in the following way:
\begin{enumerate}
\item A positive root $\beta \in \sfR^{+} \setminus \{\alpha_{i}\}$ is located at the coordinate $(j,p)$ in $\Gamma_{s_iQ}$,
if $s_{i}\beta$ is located at the coordinate $(j,p)$ in $\Gamma_Q$,
\item The simple root $\alpha_i$ is located at the coordinate $(i^{*}, \xi_{i}-h^{\vee})$ in $\Gamma_{s_{i}Q}$,
while $\alpha_{i} = \gamma_{i}^{Q}$ was located at the coordinate $(i,\xi_{i})$ in $\Gamma_Q$.
\end{enumerate}

\subsection{Generalities on commutation classes}
\label{subsection: commutation classes}

Two sequences $\bfi$ and $\bfi^{\prime}$ of elements of $\Delta_{0}$
are said to be \emph{commutation equivalent}
if $\bfi^{\prime}$ is obtained from $\bfi$ by applying a sequence of operations
which transform some adjacent components $(i, j)$ such that $i \not \sim j$ into $(j,i)$.
This defines an equivalence relation.
We refer to the equivalent class containing $\bfi$ as the \emph{commutation class} of $\bfi$
and
denote it by $[\bfi]$.
Note that the set of all the reduced words of an element $w \in \sfW$ is divided into a disjoint union of
commutation classes.
In this paper, we mainly consider the commutation classes of
reduced words for the longest element $w_{0}$.

For a reduced word  $\bfi = (i_{1}, \ldots, i_{N})$ for the longest element $w_{0} \in \sfW$,
we have
$\sfR^{+}=\{  \beta^{\bfi}_{k} \mid 1 \le k \le N \}$, where
$\beta^{\bfi}_{k} := s_{i_{1}}\cdots s_{i_{k-1}}(\alpha_{i_{k}})$.
Thus the reduced word $\bfi$ defines a total order $<_{\bfi}$ on $\sfR^{+}$,
namely we write $\beta^{\bfi}_k <_{\bfi} \beta^{\bfi}_l$ if $k < l$.
Note that the order $<_{\bfi}$ is \emph{convex} in the sense that if
$\alpha,\beta,\alpha+\beta \in \sfR^+$,
we have either $\alpha <_{\bfi} \alpha+\beta <_{\bfi} \beta$
or $\beta <_{\bfi} \alpha+\beta <_{\bfi} \alpha$.
For a commutation class $[\bfi]$ of reduced words for $w_{0}$,
we define the convex partial order $\preceq_{[\bfi]}$ on $\sfR^{+}$ so that
$$\alpha \preceq_{[\bfi]} \beta \quad \text{ if and only if } \quad \alpha <_{\bfi^{\prime}} \beta \text{ for all } \ \bfi^{\prime} \in [\bfi].$$

For a commutation class $[\bfi]$ of reduced words of $w_{0}$ and a positive root $\alpha \in \sfR^{+}$,
we define the $[\bfi]$-\emph{residue} of $\alpha$, denoted by $\res^{[\bfi]}(\alpha)$,
to be $i_k \in \Delta_{0}$ if we have $\alpha = \beta^{\bfi}_{k}$ with
$\bfi = (i_{1}, \ldots, i_{N})$.
Note that this is well-defined, i.e.~if we have
$\alpha = \beta^{\bfi^{\prime}}_{l}$
for another $\bfi^{\prime} = (i_{1}^{\prime}, \ldots, i_{N}^{\prime}) \in [\bfi]$, then $i_{l}^{\prime} = i_{k}$.

In~\cite{OS19a}, U.~R.~Suh and the second named author
defined the \emph{combinatorial Auslander-Reiten quiver}
$\Upsilon_{\bfi}$ for each reduced word $\bfi = (i_{1}, \ldots, i_{N})$ for $w_{0}$.
By definition, it is a quiver whose vertex set is $\sfR^{+}$
and we have an arrow in $\Upsilon_{\bfi}$ from $\beta^{\bfi}_k$ to $\beta^{\bfi}_j$
if and only if $1\leq j < k \leq N$, $ i_{j} \sim i_{k}$
and there is no index $j<j'<k$ such that $i_{j'} \in \{  i_j, i_k\}$.
The quiver $\Upsilon_{\bfi}$ satisfies the following nice properties.
We say that a total ordering $\sfR^{+} = \{\beta_{1}, \beta_{2}, \ldots, \beta_{N} \}$
is a \emph{compatible reading of $\Upsilon_{\bfi}$} if we have $k<l$
whenever there is an arrow $\beta_{l} \to \beta_{k}$ in $\Upsilon_{\bfi}$.

\begin{theorem}[\cite{OS19a}]
\label{theorem: OS19a main}
For a commutation class $[\bfi]$ for $w_{0}$, we have the followings$\colon$
\begin{enumerate}
\item
If $\bfi^{\prime} \in [\bfi]$, then $\Upsilon_{\bfi} = \Upsilon_{\bfi^{\prime}}$. \ber Hence $\Upsilon_{[\bfi]}$ is well-defined. \er
\item For $\alpha, \beta \in \sfR^{+}$, we have
$\alpha \preceq_{[\bfi]} \beta$ if and only if there exists a path from $\beta$ to $\alpha$ in $\Upsilon_{[\bfi]}$.
In other words, the quiver $\Upsilon_{[\bfi]}$ is a Hasse quiver of the partial ordering $\preceq_{[\bfi]}$.
\item \label{OS19a pairing}
For $\alpha, \beta \in \sfR^{+}$, we have $(\alpha, \beta)=0$
if they are not comparable with respect to $\preceq_{[\bfi]}$\ber. \er
\item A sequence $\bfi^{\prime} = (i^{\prime}_{1}, \ldots, i_{N}^{\prime}) \in \Delta_{0}^{N}$
belongs to the commutation class $[\bfi]$ if and only if there is a compatible reading
$\sfR^{+} = \{\beta_{1}, \ldots, \beta_{N} \}$ of $\Upsilon_{[\bfi]}$
such that $i^{\prime}_{k} = \res^{[\bfi]}(\beta_{k})$ for all $1 \le k \le N$.
\end{enumerate}
\end{theorem}

For a reduced word $\bfi = (i_{1}, i_{2}, \ldots, i_{N})$ of $w_{0}$,
the sequence $\bfi^{\prime} = (i_{2}, \ldots, i_{N}, i^{*}_{1})$ is also a reduced word of $w_0$ and
$[\bfi] \neq [\bfi^{\prime}]$.
This operation is referred to as a \emph{$($combinatorial$)$ reflection functor} and we write $\bfi^{\prime} = r_{i_{1}} \bfi$.
We have the induced operation on commutation classes (i.e.~$r_{i_1} [\bfi] := [r_{i_{1}} \bfi]$ is well-defined).
The relations $[\bfi] \overset{r}{\sim} [r_{i}\bfi]$ for 
$i \in \Delta_{0}$
generate an equivalence relation, called the \emph{reflection equivalent relation} $\overset{r}{\sim}$,
on the set of commutation classes of reduced words for $w_0$.
For a given reduced word $\bfi$ of $w_0$, the family of commutation classes
$\lf\bfi\rf := \{ [\bfi^{\prime}] \mid [\bfi^{\prime}] \overset{r}{\sim} [\bfi] \}$
is called an \emph{$r$-cluster point}.

\subsection{Adapted commutation classes}

Let $Q$ be a Dynkin quiver of type $\sfg$.
It is well-known that the set of all reduced words of $w_{0}$ adapted to $Q$ forms a
single commutation class $[Q]$ of $w_0$, and
$[Q] =[Q']$ if and only if $Q = Q'$ (see~\cite{Bedard99}).
Furthermore, if $i$ is a source of $Q$, we have $r_{i} [Q] = [s_{i}Q]$.
Since any Dynkin quiver $Q^{\prime}$ of type $\sfg$ can be obtained from a given $Q$
by a sequence of source reflections as $Q^{\prime} = s_{i_1} \cdots s_{i_m}Q$,
we have $[Q^{\prime}] \overset{r}{\sim} [Q]$.
Therefore  the set $\{ [Q] \mid \text{$Q$ is a Dynkin quiver of type $\sfg$}\}$
forms a single reflection equivalent class $\lf\Delta\rf$ called the \emph{adapted $r$-cluster point}.
We call a commutation class in $\lf \Delta \rf$ an \emph{adapted class}.
By the above discussion, we have a canonical bijection
\begin{equation}
\label{eq: bijection1}
\{ \text{Dynkin quivers of type $\sfg$} \} \quad \overset{1:1}{\longleftrightarrow} \quad \lf \Delta \rf
\end{equation}
which associates a Dynkin quiver $Q$ with the adapted class $[Q]$.


\begin{theorem}[\cite{Bedard99}]
Let $Q$ be a Dynkin quiver of type $\sfg$.
\begin{enumerate}
\item If $\beta \in \sfR^{+}$ is located at the coordinate $(i,p)$ in $\Gamma_{Q}$,
we have $\res^{[Q]}(\beta) = i$.
\item We have $\Upsilon_{[Q]} = \Gamma_{Q}$.
\end{enumerate}
\end{theorem}

\subsection{Remark}
\label{subsection: remark on height}

So far we have defined a height function $\xi$ on a given Dynkin quiver $Q$
to be a function $\xi \colon \Delta_{0} \to \Z$ \ber satisfying \er the condition~(\ref{eq: def classical height function 1}).
Conversely, suppose that a function $\xi \colon \Delta_{0} \to \Z$
satisfies
\begin{align}
\label{eq: def classical height function 2}
|\xi_i - \xi_{j}| = 1   \qquad \text{if $i \sim j$}.
\end{align}
Then it defines a Dynkin quiver $Q$ of type $\sfg$
such that
there is an arrow $i \to j$ in $Q$ if and only if $i \sim j$ and $\xi_{i} > \xi_{j}$.
The function $\xi$ gives a height function on this Dynkin quiver $Q$.

Since a height function $\xi$ is determined uniquely up to constant from the Dynkin quiver $Q$,
we can rewrite the canonical bijection~(\ref{eq: bijection1}) as:
\begin{align}
\label{eq: bijection2}
\{ \text{functions $\xi \colon \Delta_{0} \to \Z$ satisfying (\ref{eq: def classical parity}) and
(\ref{eq: def classical height function 2})}\}/2\Z
\quad \overset{1:1}{\longleftrightarrow} \quad
\lf \Delta \rf.
\end{align}
In the next section, we generalize this bijection for
the twisted adapted classes,
which are other special kinds of commutation classes
for $w_{0} \in \sfW$.

\section{Q-data, twisted adapted classes and twisted Auslander-Reiten quivers}
\label{section: Q-data}

In this section, we first recall the twisted adapted class associated with a pair $(\Delta, \sigma)$
of Dynkin diagram $\Delta$ of type $\mathsf{ADE}$
and a diagram automorphism $\sigma$ on it.
Then we introduce the notion of a Q-datum to develop a unified theory of twisted AR quivers
and generalized twisted Coxeter elements.

\subsection{Basic notation}
\label{subsection: notation2}

\begin{table}[h]
\centering
 { \arraycolsep=1.6pt\def\arraystretch{1.5}
\begin{tabular}{|c|c|c|c|c|c|}
\hline
$\sfg$ & $\sigma$ & $r$ & $\g$ & $h^{\vee}$ & $N$ \\
\hline
\hline
$\mathsf{A}_{n}$ & $\id$ & $1$ & $\mathsf{A}_{n}$ & $n+1$ & $n(n+1)/2$ \\
$\mathsf{D}_{n}$ & $\id$ & $1$ & $\mathsf{D}_{n}$ & $2n-2$ & $n(n-1)$ \\
$\mathsf{E}_{6,7,8}$ & $\id$ & $1$ & $\mathsf{E}_{6,7,8}$ & $12, 18, 30$& $36, 63, 120$ \\
\hline
$\mathsf{A}_{2n-1}$ & $\vee$ & $2$ & $\mathsf{B}_{n}$ & $2n-1$ & $n(2n-1)$\\
$\mathsf{D}_{n+1}$ & $\vee$ & $2$ & $\mathsf{C}_{n}$ & $n+1$ & $n(n+1)$ \\
$\mathsf{E}_{6}$ & $\vee$ & $2$ & $\mathsf{F}_{4}$ & $9$ & $36$ \\
\hline
$\mathsf{D}_{4}$ & $\widetilde{\vee}, \widetilde{\vee}^{2}$ & $3$ & $\mathsf{G}_{2}$ & $4$ & $12$ \\
\hline
\end{tabular}
  }\\[1.5ex]
    \caption{}
    \protect\label{table: classification}
\end{table}

Let $\sfg$ be a finite-dimensional complex simple Lie algebra of type $\mathsf{ADE}$
and $\Delta$ be its Dynkin diagram (simply-laced).
We keep the notation in Subsection~\ref{subsection: notation1} except for $n$ and $h^{\vee}$,
which will denote respectively the rank and the dual Coxeter number of another simple Lie algebra $\g$
defined below. Also we mainly use the symbols $\im, \jm, \ldots$ for denoting the vertices in $\Delta_{0}$
in order to save the symbols $i, j, \ldots$ for denoting the Dynkin indices of $\g$.

Let $\sigma$ be an automorphism of $\Delta$
satisfying the condition
\begin{equation} \label{eq:automorphism}
\text{ there is no index $\im \in \Delta_{0}$ such that $\im \sim \sigma(\im)$}.
\end{equation}
Such a pair $(\Delta, \sigma)$ is classified in Table~\ref{table: classification},
where the automorphisms $\vee$ and $\tvee$ are defined as follows:

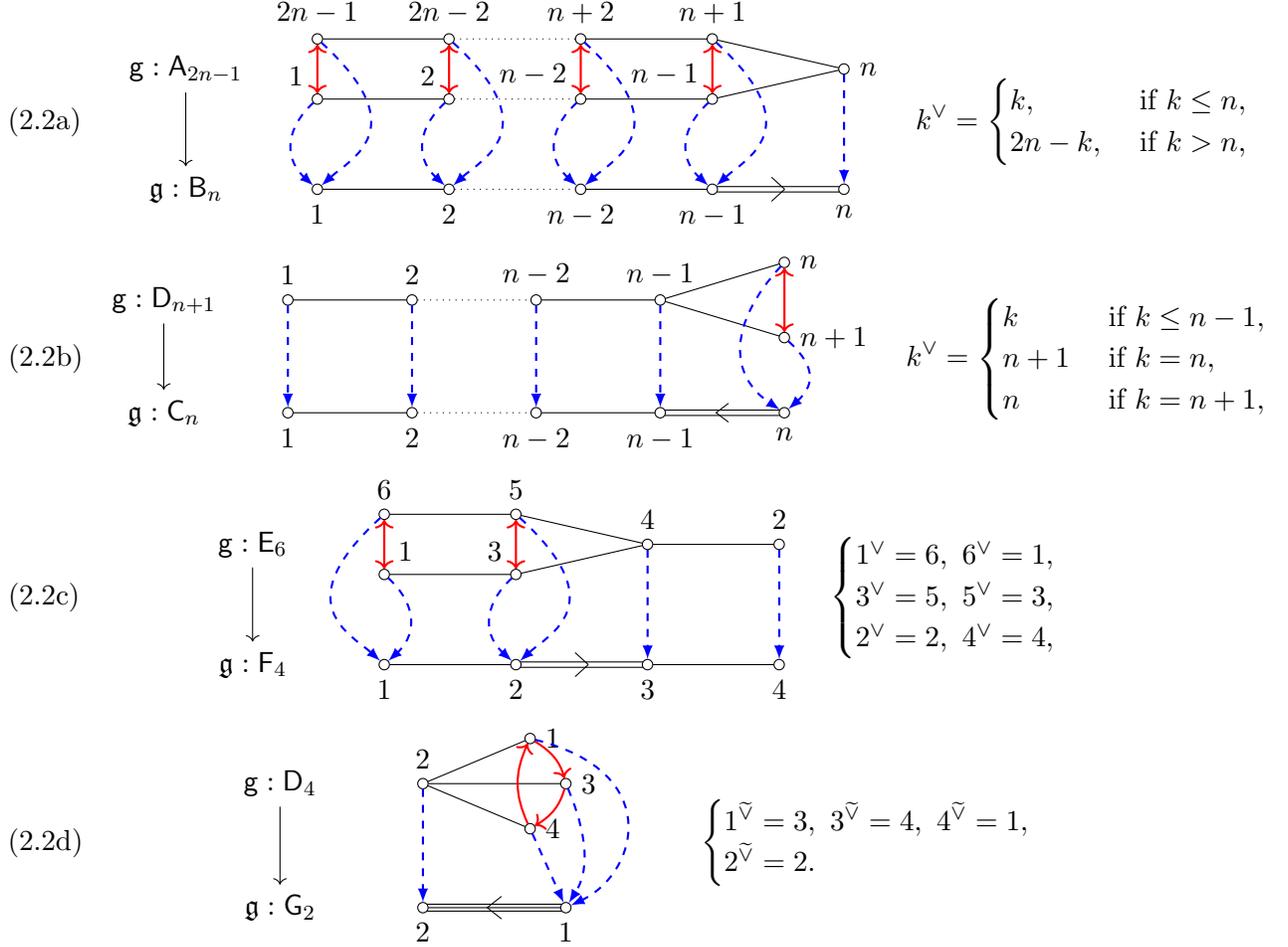
\begin{subequations}
\label{eq: diagram foldings}
\begin{gather}
\label{eq: B_n}
\begin{tikzpicture}[xscale=1.75,yscale=.8,baseline=0]
\node (A2n1) at (0,1) {$\sfg:\mathsf{A}_{2n-1}$};
\node[dynkdot,label={above:$n+1$}] (A6) at (4,1.5) {};
\node[dynkdot,label={above:$n+2$}] (A7) at (3,1.5) {};
\node[dynkdot,label={above:$2n-2$}] (A8) at (2,1.5) {};
\node[dynkdot,label={above:$2n-1$}] (A9) at (1,1.5) {};
\node[dynkdot,label={above left:$n-1$}] (A4) at (4,0.5) {};
\node[dynkdot,label={above left:$n-2$}] (A3) at (3,0.5) {};
\node[dynkdot,label={above left:$2$}] (A2) at (2,0.5) {};
\node[dynkdot,label={above left:$1$}] (A1) at (1,0.5) {};
\node[dynkdot,label={right:$n$}] (A5) at (5,1) {};
\path[-]
 (A1) edge (A2)
 (A3) edge (A4)
 (A4) edge (A5)
 (A5) edge (A6)
 (A6) edge (A7)
 (A8) edge (A9);
\path[-,dotted] (A2) edge (A3) (A7) edge (A8);
\path[<->,thick,red] (A1) edge (A9) (A2) edge (A8) (A3) edge (A7) (A4) edge (A6);
\def\Foffset{-1}
\node (Bn) at (0,\Foffset) {$\g:\mathsf{B}_n$};
\foreach \x in {1,2}
{\node[dynkdot,label={below:$\x$}] (B\x) at (\x,\Foffset) {};}
\node[dynkdot,label={below:$n-2$}] (B3) at (3,\Foffset) {};
\node[dynkdot,label={below:$n-1$}] (B4) at (4,\Foffset) {};
\node[dynkdot,label={below:$n$}] (B5) at (5,\Foffset) {};
\path[-] (B1) edge (B2)  (B3) edge (B4);
\draw[-,dotted] (B2) -- (B3);
\draw[-] (B4.30) -- (B5.150);
\draw[-] (B4.330) -- (B5.210);
\draw[-] (4.55,\Foffset) -- (4.45,\Foffset+.2);
\draw[-] (4.55,\Foffset) -- (4.45,\Foffset-.2);
 \draw[-latex,dashed,color=blue,thick]
 (A1) .. controls (0.75,\Foffset+1) and (0.75,\Foffset+.5) .. (B1);
 \draw[-latex,dashed,color=blue,thick]
 (A9) .. controls (1.75,\Foffset+1) and (1.25,\Foffset+.5) .. (B1);
 \draw[-latex,dashed,color=blue,thick]
 (A2) .. controls (1.75,\Foffset+1) and (1.75,\Foffset+.5) .. (B2);
 \draw[-latex,dashed,color=blue,thick]
 (A8) .. controls (2.75,\Foffset+1) and (2.25,\Foffset+.5) .. (B2);
 \draw[-latex,dashed,color=blue,thick]
 (A3) .. controls (2.75,\Foffset+1) and (2.75,\Foffset+.5) .. (B3);
 \draw[-latex,dashed,color=blue,thick]
 (A7) .. controls (3.75,\Foffset+1) and (3.25,\Foffset+.5) .. (B3);
 \draw[-latex,dashed,color=blue,thick]
 (A4) .. controls (3.75,\Foffset+1) and (3.75,\Foffset+.5) .. (B4);
 \draw[-latex,dashed,color=blue,thick]
 (A6) .. controls (4.75,\Foffset+1) and (4.25,\Foffset+.5) .. (B4);
 \draw[-latex,dashed,color=blue,thick] (A5) -- (B5);
\draw[->] (A2n1) -- (Bn);
\end{tikzpicture}
\quad k^{\vee} = \begin{cases}
k, & \text{ if } k \le n, \\
2n - k, & \text{ if } k > n,
\end{cases}
\allowdisplaybreaks \\
\label{eq: C_n}
\begin{tikzpicture}[xscale=1.65,yscale=1.25,baseline=-25]
\node (Dn1) at (0,0) {$\sfg : \mathsf{D}_{n+1}$};
\node[dynkdot,label={above:$1$}] (D1) at (1,0){};
\node[dynkdot,label={above:$2$}] (D2) at (2,0) {};
\node[dynkdot,label={above:$n-2$}] (D3) at (3,0) {};
\node[dynkdot,label={above:$n-1$}] (D4) at (4,0) {};
\node[dynkdot,label={right:$n$}] (D6) at (5,.4) {};
\node[dynkdot,label={right:$n+1$}] (D5) at (5,-.4) {};
\path[-] (D1) edge (D2)
  (D3) edge (D4)
  (D4) edge (D5)
  (D4) edge (D6);
\draw[-,dotted] (D2) -- (D3);
\path[<->,thick,red] (D6) edge (D5);
\def\Coffset{-1.2}
\node (Cn) at (0,\Coffset) {$\g: \mathsf{C}_n$};
\foreach \x in {1,2}
{\node[dynkdot,label={below:$\x$}] (C\x) at (\x,\Coffset) {};}
\node[dynkdot,label={below:$n-2$}] (C3) at (3,\Coffset) {};
\node[dynkdot,label={below:$n-1$}] (C4) at (4,\Coffset) {};
\node[dynkdot,label={below:$n$}] (C5) at (5,\Coffset) {};
\draw[-] (C1) -- (C2);
\draw[-,dotted] (C2) -- (C3);
\draw[-] (C3) -- (C4);
\draw[-] (C4.30) -- (C5.150);
\draw[-] (C4.330) -- (C5.210);
\draw[-] (4.55,\Coffset+.1) -- (4.45,\Coffset) -- (4.55,\Coffset-.1);
\path[-latex,dashed,color=blue,thick]
 (D1) edge (C1)
 (D2) edge (C2)
 (D3) edge (C3)
 (D4) edge (C4);
\draw[-latex,dashed,color=blue,thick]
 (D6) .. controls (4.55,-.25) and (4.55,-0.8) .. (C5);
\draw[-latex,dashed,color=blue,thick]
 (D5) .. controls (5.25,\Coffset+.5) and (5.25,\Coffset+.3) .. (C5);
 \draw[->] (Dn1) -- (Cn);
\end{tikzpicture}
\quad k^{\vee} = \begin{cases} k & \text{ if } k \le n-1, \\ n+1 & \text{ if } k = n, \\ n & \text{ if } k = n+1, \end{cases}
\allowdisplaybreaks \\
\label{eq: F_4}
\begin{tikzpicture}[xscale=1.75,yscale=.8,baseline=0]
\node (E6desc) at (0,1) {$\sfg : \mathsf{E}_6$};
\node[dynkdot,label={above:$2$}] (E2) at (4,1) {};
\node[dynkdot,label={above:$4$}] (E4) at (3,1) {};
\node[dynkdot,label={above:$5$}] (E5) at (2,1.5) {};
\node[dynkdot,label={above:$6$}] (E6) at (1,1.5) {};
\node[dynkdot,label={above left:$3$}] (E3) at (2,0.5) {};
\node[dynkdot,label={above right:$1$}] (E1) at (1,0.5) {};
\path[-]
 (E2) edge (E4)
 (E4) edge (E5)
 (E4) edge (E3)
 (E5) edge (E6)
 (E3) edge (E1);
\path[<->,thick,red] (E3) edge (E5) (E1) edge (E6);
\def\Foffset{-1}
\node (F4desc) at (0,\Foffset) {$\g : \mathsf{F}_4$};
\foreach \x in {1,2,3,4}
{\node[dynkdot,label={below:$\x$}] (F\x) at (\x,\Foffset) {};}
\draw[-] (F1.east) -- (F2.west);
\draw[-] (F3) -- (F4);
\draw[-] (F2.30) -- (F3.150);
\draw[-] (F2.330) -- (F3.210);
\draw[-] (2.55,\Foffset) -- (2.45,\Foffset+.2);
\draw[-] (2.55,\Foffset) -- (2.45,\Foffset-.2);
\path[-latex,dashed,color=blue,thick]
 (E2) edge (F4)
 (E4) edge (F3);
\draw[-latex,dashed,color=blue,thick]
 (E1) .. controls (1.25,\Foffset+1) and (1.25,\Foffset+.5) .. (F1);
\draw[-latex,dashed,color=blue,thick]
 (E3) .. controls (1.75,\Foffset+1) and (1.75,\Foffset+.5) .. (F2);
\draw[-latex,dashed,color=blue,thick]
 (E5) .. controls (2.75,\Foffset+1) and (2.25,\Foffset+.5) .. (F2);
\draw[-latex,dashed,color=blue,thick]
 (E6) .. controls (0.25,\Foffset+1) and (0.75,\Foffset+.5) .. (F1);
\draw[->] (E6desc) -- (F4desc);
\end{tikzpicture}
\quad \begin{cases} 1^{\vee}=6, \ 6^{\vee}=1, \\ 3^{\vee}=5, \ 5^{\vee}=3, \\ 2^{\vee}=2, \ 4^{\vee}=4, \end{cases}
\allowdisplaybreaks \\
\label{eq: G_2}
\begin{tikzpicture}[xscale=1.9,yscale=1.5,baseline=-25]
\node (D4desc) at (0,0) {$\sfg : \mathsf{D}_{4}$};
\node[dynkdot,label={right:$1$}] (D1) at (1.75,.4){};
\node[dynkdot,label={above:$2$}] (D2) at (1,0) {};
\node[dynkdot,label={right:$3$}] (D3) at (2,0) {};
\node[dynkdot,label={right:$4$}] (D4) at (1.75,-.4) {};
\draw[-] (D1) -- (D2);
\draw[-] (D3) -- (D2);
\draw[-] (D4) -- (D2);
\path[->,red,thick]
(D1) edge[bend left=20] (D3)
(D3) edge[bend left=20] (D4)
(D4) edge[bend left=20] (D1);
\def\Coffset{-1.1}
\node (G2desc) at (0,\Coffset) {$\g : \mathsf{G}_2$};
\node[dynkdot,label={below:$2$}] (G1) at (1,\Coffset){};
\node[dynkdot,label={below:$1$}] (G2) at (2,\Coffset) {};
\draw[-] (G1) -- (G2);
\draw[-] (G1.40) -- (G2.140);
\draw[-] (G1.320) -- (G2.220);
\draw[-] (1.55,\Coffset+.1) -- (1.45,\Coffset) -- (1.55,\Coffset-.1);
\path[-latex,dashed,color=blue,thick]
 (D2) edge (G1);
\draw[-latex,dashed,color=blue,thick]
 (D1) .. controls (2.7,0.1) and (2.5,-0.8) .. (G2);
\draw[-latex,dashed,color=blue,thick]
 (D3) .. controls (2.1,-0.3) and (2.2,-0.7) .. (G2);
\draw[-latex,dashed,color=blue,thick] (D4) -- (G2);
\draw[->] (D4desc) -- (G2desc);
\end{tikzpicture}
\quad \begin{cases} 1^{\widetilde{\vee}}=3, \ 3^{\widetilde{\vee}}=4, \
4^{\widetilde{\vee}}=1, \\ 2^{\widetilde{\vee}}=2. \end{cases}
\end{gather}
\end{subequations}

Hereafter we set $r := |\sigma| \in \{1,2,3\}$.
Given such a pair $(\Delta, \sigma)$,
we denote by $I$ the set of $\sigma$-orbits of $\Delta_{0}$.
We write the natural quotient map
$\Delta_{0} \to I$
by $\im \mapsto \bar{\im}$.
We set $n := |I|$.
For each $i \in I$, we define
$$d_{i} := |i| \quad \in \{ 1, r\}.$$
Note that we have $d_{\bar{\im}} = 1$ if and only if $\im \in \Delta_{0}$ is fixed by $\sigma$.
Also we define
$$r_{i} := r/d_{i} \quad \in \{ 1, r\}.$$

For $i, j \in I$,
we write $i \sim j$
if there are $\im, \jm \in \Delta_{0}$ satisfying $\bar{\im}= i, \bar{\jm} = j$ and $\im \sim \jm$.
We attach an integer $c_{ij}$ to each pair $(i, j) \in I$ by
$$
c_{ij} := \begin{cases}
2 & \text{if $i = j$;} \\
-\lceil d_{j}/d_{i} \rceil & \text{if $i \sim j$;} \\
0 & \text{otherwise}.
\end{cases}
$$
The resulting matrix $C= (c_{ij})_{i, j \in I}$
gives the Cartan matrix of a complex simple Lie algebra $\g$.
Conversely, every complex simple Lie algebra $\g$ is obtained in this way
from the unique pair $(\Delta, \sigma)$ determined as in Table~\ref{table: classification}.
Note that $\g$ is not simply-laced if and only if $\sigma$ is non-trivial.
If we set $D := \mathop{\mathrm{diag}}(d_{i} \mid i \in I)$,
the product $DC$ becomes symmetric:
\begin{equation}
\label{eq: DC symmetric}
d_{i}c_{ij} = d_{j}c_{ji} \quad \text{for any $i,j \in I$}.
\end{equation}
Write $$C^{-1} = (\tc_{ij})_{i,j \in I} \in GL_{I}(\Q).$$
Since $D C^{-1} = D \cdot (DC)^{-1} \cdot D$ is symmetric,
we also have
\begin{equation}
\label{eq: DtC symmetric}
d_{i}\tc_{ij} = d_{j}\tc_{ji} \quad \text{for any $i,j \in I$}.
\end{equation}

Let $h^{\vee}$ denote the dual Coxeter number of the simple Lie algebra $\g$.
For each pair $(\Delta, \sigma)$, we have the equality
\begin{equation} \label{eq:dcn}
n r h^{\vee} = 2 N,
\end{equation}
which can be checked easily from Table~\ref{table: classification}.

\begin{remark}
\label{remark: h^vee parity}
Let $(\Delta, \sigma)$ be as above.
Recall the involution $\im \mapsto \im^{*}$ on the set $\Delta_{0}$
from Subsection~\ref{subsection: notation1}.
Under the assumption $\sigma \neq \id$,
we have $* = \id$ if $h^{\vee}$ is even, and $* = \sigma$ if $h^{\vee}$ is odd.
Note also that $r=2$ whenever $h^{\vee}$ is odd.
In particular, the involution $*$ on $\Delta_{0}$ induces an involution
on the set $I$, for which we use the same notation $i \mapsto i^{*}$.
The latter involution is trivial if $\g$ is not simply-laced.
\end{remark}
\subsection{Twisted adapted classes}

Note that the diagram automorphism $\sigma$
naturally gives an element of $\Aut(\sfP)$ which we denote by the same symbol $\sigma$.
Namely we define $\sigma (\varpi_{\im}) := \varpi_{\sigma(\im)}$ for each $\im \in \Delta_{0}$.
Then it is immediate that $\sigma(\alpha_{\im}) = \alpha_{\sigma(\im)}$ and
$\sigma s_{\im} \sigma^{-1} = s_{\sigma(\im)}$ for each $\im \in \Delta_{0}$.

For a sequence $(\im_{1}, \ldots, \im_{n}) \in (\Delta_{0})^{n}$ such that
$\{ \bar{\im}_{1}, \ldots, \bar{\im}_{n}\} = I$,
the element $\tau = s_{\im_{1}} \cdots s_{\im_{n}} \sigma \in \sfW \sigma$
is called a \emph{$\sigma$-Coxeter element} \beb (or a \emph{twisted Coxeter element})  \cite{Springer74, CP96}. \eb
Note that if $\sigma = \id$, a $\sigma$-Coxeter element is
the same as a usual Coxeter element.

\begin{theorem}[\cite{OS19b}] \label{theorem: OSmain}
Assume $\sigma \neq \id$.
Note that we have $rh^{\vee} \in 2\Z$ in this case.
\begin{enumerate}
\item
For a $\sigma$-Coxeter element
$\tau = s_{\im_{1}} \cdots s_{\im_{n}} \sigma \in \sfW \sigma$, we have
$$
\tau^{r h^{\vee}/2} = w_{0}\sigma^{rh^{\vee}/2} =-1.
$$
In particular, it defines a reduced word
$$
\bfi(\tau) :=
(\im_{1}, \ldots, \im_{n}, \sigma(\im_{1}), \ldots, \sigma(\im_{n}), \ldots \ldots,
\sigma^{rh^{\vee}/2-1}(\im_{1}), \ldots, \sigma^{rh^{\vee}/2-1}(\im_{n}) )
$$
for the longest element $w_{0}$.
\item
The commutation class $[\bfi(\tau)]$ depends only on the element $\tau$.
\item
For any two $\sigma$-Coxeter elements $\tau_{1}, \tau_{2} \in \sfW\sigma$,
we have $[\bfi(\tau_{1})] = [\bfi(\tau_{2})]$ if and only if $\tau_{1} = \tau_{2}$.
On the other hand, we have $[\bfi(\tau_{1})] \overset{r}{\sim} [\bfi(\tau_{2})]$
in any cases.\footnote{\ber This latter assertion
goes back to a classical result due to 
Springer~\cite[Lemma 7.5]{Springer74}. \er}
\end{enumerate}
\end{theorem}
\begin{proof}
See~\cite[Section 3]{OS19b}.
Note that the fact $w_{0}\sigma^{rh^{\vee}/2} =-1$ follows from Remark~\ref{remark: h^vee parity} above.
\end{proof}

\begin{definition} \label{def: sigma cluster point}
Assume $\sigma \neq \id.$
Thanks to Theorem~\ref{theorem: OSmain},
there exists a unique $r$-cluster point containing
all the commutation classes $[\bfi(\tau)]$ arising from $\sigma$-Coxeter elements $\tau$.
We call it the \emph{$\sigma$-adapted $r$-cluster point}
(or the \emph{twisted adapted $r$-cluster point}) and denote it by $\lf \Delta^{\!\sigma} \rf$.
We refer to a commutation class belonging to $\lf \Delta^{\!\sigma} \rf$ as
a \emph{$\sigma$-adapted class} (or a \emph{twisted adapted class}).
\end{definition}

\begin{convention}
When $\sigma = \id$,
we understand $\lf \Ds \rf = \lf \Delta \rf$
so that an $\id$-adapted class is the same as a usual adapted class.
For a usual Coxeter element $\tau \in \sfW$, we define
$[\bfi(\tau)] := [Q] \in \lf \Delta \rf$, where $Q$ is the unique Dynkin quiver
to which $\tau$ is adapted, i.e.~$\tau_{Q}=\tau$.
\end{convention}

\begin{remark}
By definition, we have an inclusion
$\{ [\bfi(\tau)] \mid \text{$\tau$ is a $\sigma$-Coxeter element} \} \subset \lf \Delta^{\!\sigma} \rf$
for any $(\Delta, \sigma)$.
However the opposite inclusion is \emph{not} always true.
It fails precisely when $(\sfg, \sigma) = (\mathsf{A}_{2n-1}, \vee)$ with $n > 2$,
or $(\sfg, \sigma) = (\mathsf{E}_{6}, \vee)$.
See Corollary~\ref{corollary: twisted Coxeter height} below.
\end{remark}

\subsection{Q-datum}
\label{subsection: Q-datum}

In what follows,
we fix a pair $(\Delta, \sigma)$ of Dynkin diagram $\Delta$ of
a Lie algebra $\sfg$ of type $\mathsf{ADE}$
and an automorphism $\sigma$ on $\Delta$ satisfying the condition (\ref{eq:automorphism}),
or equivalently, a complex simple Lie algebra $\g$.
Recall that we set $I := \bar{\Delta}_{0}$.

\begin{definition}
\label{def: height}
A function $\uxi \colon \Delta_{0} \to \Z$
is called a \emph{height function on $(\Delta, \sigma)$} if
the following two conditions are satisfied.
\begin{itemize}
\item[($\mathtt{H1}$)] \label{def: height: cond1}
For any $\im, \jm \in \Delta_{0}$ such that $\im \sim \jm$ and $d_{\bar{\im}} = d_{\bar{\jm}}$,
we have $|\uxi_{\im} - \uxi_{\jm}| = d_{\bar{\im}} = d_{\bar{\jm}}$.
\item[($\mathtt{H2}$)] \label{def: height: cond2}
For any $i,j \in I$ with $i \sim j$ and $d_{i} < d_{j}$, there exists a unique element $j^{\circ} \in j$ such that
$|\uxi_{\im} - \uxi_{j^{\circ}}| = 1$ and $\uxi_{\sigma^{l}(j^{\circ})} = \uxi_{j^{\circ}} - 2l$
for any $0 \le l < r$, where $\im \in i$ is the unique element. (Note that we have
$d_{i} = 1$ and $d_{j} = r$ in this case.)
\end{itemize}
We refer to a triple $\calQ = (\Delta, \sigma, \uxi)$
as a \emph{${\rm Q}$-datum} for $\g$.
\end{definition}

The condition $(\mathtt{H2})$ can be seen as a local condition
at a ``branching point'' $i = \{ \im \}$,
while the condition $(\mathtt{H1})$ is a local condition at a ``non-branching point''.
\ber
\begin{example}
To illustrate the condition $(\mathtt{H2})$, 
we depict two examples of Q-data for $\g$ of type $\mathsf{G}_2$: 
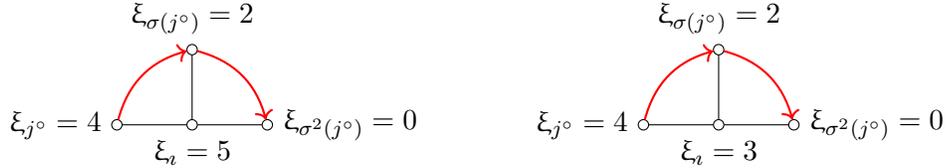
\begin{center}
    \begin{tikzpicture}[xscale=1,yscale=1,baseline=-25]
\node[dynkdot,label={below:$\uxi_\im = 5$}] (D1) at (0,0){};
\node[dynkdot,label={left:$\uxi_{j^\circ}=4$}] (D2) at (-1,0) {};
\node[dynkdot,label={right:$\uxi_{\sigma^2(j^\circ)}=0$}] (D3) at (1,0) {};
\node[dynkdot,label={above:$\uxi_{\sigma(j^\circ)}=2$}] (D4) at (0,1) {};
\draw[-] (D1) -- (D2);
\draw[-] (D1) -- (D3);
\draw[-] (D1) -- (D4);
\path[->,red,thick]
(D2) edge[bend left] (D4)
(D4) edge[bend left] (D3);
\node[dynkdot,label={below:$\uxi_\im=3$}] (D1) at (7,0){};
\node[dynkdot,label={left:$\uxi_{j^\circ}=4$}] (D2) at (6,0) {};
\node[dynkdot,label={right:$\uxi_{\sigma^2(j^\circ)}=0$}] (D3) at (8,0) {};
\node[dynkdot,label={above:$\uxi_{\sigma(j^\circ)}=2$}] (D4) at (7,1) {};
\draw[-] (D1) -- (D2);
\draw[-] (D1) -- (D3);
\draw[-] (D1) -- (D4);
\path[->,red,thick]
(D2) edge[bend left] (D4)
(D4) edge[bend left] (D3);
\end{tikzpicture}
\end{center}
Note that an arbitrary Q-datum of type $\mathsf{G}_2$
is identical to one of these two cases up to adding a constant and 
permuting the vertices.
\end{example}
\er

\begin{example} \label{example: QB3}
Here are two examples of Q-data for $\g$ of type $\mathsf{B}_{3}$.
Note that the corresponding pair is $(\sfg, \sigma) = (\mathsf{A}_{5}, \vee)$.
We put the values of height functions above the vertices.
$$
\calQ^{(1)}=
\left(
\begin{xy}
\def\objectstyle{\scriptstyle}
\ar@{->} (-10,0) *+!D{6} *\cir<2pt>{};
(0,0) *+!D{4} *\cir<2pt>{}="A",
\ar@{<-} "A"; (10,0) *+!D{7} *\cir<2pt>{}="B",
\ar@{->} "B"; (20,0) *+!D{6} *\cir<2pt>{}="C",
\ar@{<-} "C"; (30,0) *+!D{8} *\cir<2pt>{}="D"
\end{xy}
\hspace{6pt}
\right),
\qquad
\calQ^{(2)}=
\left(
\begin{xy}
\def\objectstyle{\scriptstyle}
\ar@{<-} (-10,0) *+!D{2} *\cir<2pt>{};
(0,0) *+!D{4} *\cir<2pt>{}="A",
\ar@{<-} "A"; (10,0) *+!D{5} *\cir<2pt>{}="B",
\ar@{<-} "B"; (20,0) *+!D{6} *\cir<2pt>{}="C",
\ar@{<-} "C"; (30,0) *+!D{8} *\cir<2pt>{}="D"
\end{xy}
\hspace{6pt}
\right).
\qquad
$$
Here $i \to j$ implies $\uxi_i > \uxi_j$.
\end{example}

\begin{remark}
When $\sigma = \id$, the condition~($\mathtt{H2}$) becomes empty.
Hence a height function $\uxi$ on $(\Delta, \id)$ is nothing but a function $\xi$
satisfying~(\ref{eq: def classical height function 2}).
Thus, in view of Subsection~\ref{subsection: remark on height},
the notion of Q-datum for a simply-laced $\g$ is equivalent
to the notion of Dynkin quiver $Q$ with a height function.
\end{remark}

We can deduce the following properties of the height function $\uxi$
easily from its definition and the classification of $(\Delta, \sigma)$ in Table~\ref{table: classification}.

\begin{lemma}
\label{lemma: parity}
Let $\calQ=(\Delta, \sigma, \uxi)$ be a ${\rm Q}$-datum and $i \in I$ be an index.
\begin{enumerate}
\item \label{parity_constant}
For any $\im, \im^{\prime} \in i$, we have $\uxi_{\im} \equiv \uxi_{\im^{\prime}} \pmod{2}$.
\item \label{parity_injection}
For any $\im \in i$ and $l \in \Z$, we have $\uxi_{\sigma^{l}(\im)} \equiv \uxi_{\im} - 2l \pmod{2d_{i}}$.
\item \label{parity_adjacent}
\ber For any $j \in I$ and $\im \in i, \jm \in j$ with $\im \sim \jm$, \er
we have $\uxi_{\im} \equiv \uxi_{\jm} + \min(d_{i}, d_{j}) \pmod{2\min(d_{i}, d_{j})}$.
\end{enumerate}
\end{lemma}

In what follows, we fix a $\Delta_{0}$-tuple $\epsilon = (\epsilon_{\im})_{\im \in \Delta_{0}}$
of integers with $0 \le \epsilon_{\im} < d_{\bar{\im}}$ for each $\im \in \Delta_{0}$ such that
the three conditions in Lemma~\ref{lemma: parity} are satisfied with $\uxi_{\im}$ therein replaced by $\epsilon_{\im}$.
We refer to such a tuple $\epsilon$ as a \emph{$\sigma$-parity function} on $\Delta$.
Note that there are only $2r$ choices of $\sigma$-parity functions
and their differences will not affect the results in this paper.
Adding a constant if necessary, we can make a given height function $\uxi$
satisfy the condition:
\begin{itemize}
\item[($\mathtt{H3}$)]
\label{eq: def parity}
We have $\uxi_{\im} - \epsilon_{\im} \in 2d_{\bar{\im}}\Z$ for each $\im \in \Delta_{0}$.
\end{itemize}
In what follows, we require a height function $\uxi \colon \Delta_{0} \to \Z$
always satisfies the condition~($\mathtt{H3}$) together with ($\mathtt{H1}$) and ($\mathtt{H2}$)
\ber without loss of generality. \er

\begin{definition}
Let $\calQ=(\Delta, \sigma, \uxi)$ be a Q-datum for $\g$.
A vertex $\im \in \Delta_{0}$ is called a \emph{source} of $\calQ$ if
we have $\uxi_{\im} > \uxi_{\jm}$ for any $\jm \in \Delta_{0}$ with $\im \sim \jm$.
\end{definition}

The following lemma is immediate from the definition.

\begin{lemma}
Let $\calQ = (\Delta, \sigma, \uxi)$ be a ${\rm Q}$-datum and
$\im \in \Delta_{0}$ be an index.
We define a function $s_{\im} \uxi \colon \Delta_{0} \to \Z$ by the rule
\begin{equation}
(s_{\im}\uxi)_{\jm} := \uxi_{\jm} - \delta_{\im, \jm} \times 2d_{\bar{\jm}}.
\end{equation}
Then $s_{\im} \uxi$ defines a height function on $(\Delta, \sigma)$
if and only if $\im$ is a source of $\calQ$.
\end{lemma}

Let $\calQ = (\Delta, \sigma, \uxi)$ be a Q-datum and
$\im \in \Delta_{0}$ be a source of $\calQ$.
Then we define a new Q-datum $s_{\im}\calQ$ to be the triple $(\Delta, \sigma, s_{\im}\uxi)$.

\begin{definition} \label{def: Qadaptedness}
Let $\calQ = (\Delta, \sigma, \uxi)$ be a Q-datum for $\g$.
We say that a sequence $(\im_{1}, \ldots, \im_{l})$ of elements of $\Delta_{0}$ is \emph{adapted to $\calQ$} if
$$\text{$\im_{k}$ is a source of $s_{\im_{k-1}} s_{\im_{k-2}} \cdots s_{\im_1}\calQ$ for all $1 \le k \le l$.}$$
\end{definition}

When $\sigma = \id$, this is equivalent to the usual notion of adaptedness for a Dynkin quiver.

\subsection{Repetition quiver and compatible readings}

\begin{definition} \label{def: r-quiver}
Recall that we have fixed a $\sigma$-parity function $\epsilon$ on $\Delta$.
The \emph{repetition quiver} associated with $(\Delta, \sigma)$ is the quiver $\hDs$ whose
vertex set $\hDs_{0}$ and arrow set $\hDs_{1}$ are given by
\begin{align*}
\hDs_{0} &:= \{ (\im, p) \in \Delta_{0} \times \Z \mid p-\epsilon_{\im} \in 2d_{\bar{\im}} \Z \}, \\
\hDs_{1} &:= \{ (\im,p) \to (\jm, s) \mid (\im, p), (\jm, s) \in \hDs_{0}, \jm \sim \im, s-p = \min(d_{\bar{\im}}, d_{\bar{\jm}}) \}.
\end{align*}
\end{definition}

\ber
\begin{example} Here are some examples of the repetition quiver.
\begin{enumerate}
    \item When $\g$ is of type $\mathsf{B}_{3}$, we choose the pair $(\mathsf{A}_{5}, \vee)$. The repetition quiver $\hDs$ is depicted as:
$$\raisebox{3.1em}{\scalebox{0.65}
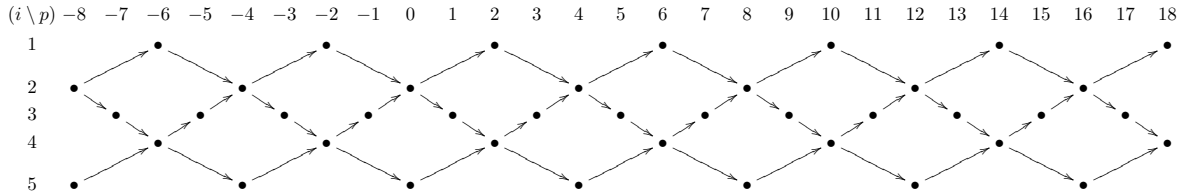
{\xymatrix@!C=0.1ex@R=0.5ex{
(i\setminus p) & -8 & -7 & -6 &-5&-4 &-3& -2 &-1& 0 & 1& 2 & 3& 4&  5
& 6 & 7 & 8 & 9 & 10 & 11 & 12 & 13 & 14 & 15 & 16 & 17& 18 \\
1&&&\bullet \ar@{->}[ddrr]&&&& \bullet \ar@{->}[ddrr]&&&&  \bullet \ar@{->}[ddrr] &&&& \bullet\ar@{->}[ddrr]
&&&& \bullet \ar@{->}[ddrr]&&&& \bullet \ar@{->}[ddrr] &&&& \bullet \\ \\
2&\bullet \ar@{->}[dr] \ar@{->}[uurr]&&&&\bullet\ar@{->}[dr] \ar@{->}[uurr]
&&&& \bullet \ar@{->}[dr]\ar@{->}[uurr] &&&&  \bullet \ar@{->}[dr] \ar@{->}[uurr]
&&&& \bullet \ar@{->}[dr]\ar@{->}[uurr]&&&& \bullet \ar@{->}[dr]\ar@{->}[uurr]&&&& \bullet \ar@{->}[dr]\ar@{->}[uurr] \\
3&& \bullet \ar@{->}[dr]&& \bullet \ar@{->}[ur] &&\bullet \ar@{->}[dr] && \bullet \ar@{->}[ur] && \bullet \ar@{->}[dr]
&& \bullet \ar@{->}[ur] && \bullet \ar@{->}[dr] &&
\bullet \ar@{->}[ur]&&\bullet \ar@{->}[dr] &&\bullet \ar@{->}[ur]&&\bullet \ar@{->}[dr] &&\bullet \ar@{->}[ur]
&&\bullet \ar@{->}[dr]\\
4&&&\bullet\ar@{->}[ddrr]\ar@{->}[ur]&&&&\bullet\ar@{->}[ddrr]\ar@{->}[ur] &&&&  \bullet \ar@{->}[ddrr]\ar@{->}[ur] &&&&
\bullet \ar@{->}[ddrr] \ar@{->}[ur]
&&&& \bullet \ar@{->}[ddrr]\ar@{->}[ur]&&&& \bullet\ar@{->}[ddrr]\ar@{->}[ur] &&&& \bullet \\ \\
5& \bullet  \ar@{->}[uurr]&&&&\bullet \ar@{->}[uurr] &&&& \bullet \ar@{->}[uurr]
&&&& \bullet   \ar@{->}[uurr] &&&& \bullet \ar@{->}[uurr]
&&&&\bullet \ar@{->}[uurr] &&&& \bullet \ar@{->}[uurr]}}}
$$
\item 
When $\g$ is of type $\mathsf{C}_{4}$, 
we choose the pair $(\mathsf{D}_5, \vee)$.
The repetition quiver $\hDs$ is depicted as:
$$
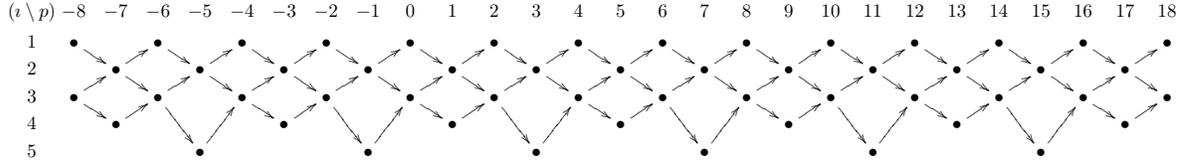
\raisebox{3mm}{
\scalebox{0.65}{\xymatrix@!C=0.1ex@R=0.5ex{
(\im\setminus p) & -8 & -7 & -6 &-5&-4 &-3& -2 &-1& 0 & 1& 2 & 3& 4&  5
& 6 & 7 & 8 & 9 & 10 & 11 & 12 & 13 & 14 & 15 & 16 & 17& 18 \\
1&\bullet \ar@{->}[dr]&& \bullet \ar@{->}[dr] &&\bullet\ar@{->}[dr]
&&\bullet \ar@{->}[dr] && \bullet \ar@{->}[dr] &&\bullet \ar@{->}[dr] &&  \bullet \ar@{->}[dr] 
&&\bullet \ar@{->}[dr] && \bullet \ar@{->}[dr] &&\bullet \ar@{->}[dr]  && \bullet\ar@{->}[dr] &&
\bullet\ar@{->}[dr] && \bullet\ar@{->}[dr]  && \bullet\\
2&&\bullet \ar@{->}[dr]\ar@{->}[ur]&& \bullet \ar@{->}[dr]\ar@{->}[ur] &&\bullet \ar@{->}[dr]\ar@{->}[ur]
&& \bullet \ar@{->}[dr]\ar@{->}[ur]&& \bullet\ar@{->}[dr] \ar@{->}[ur]&& \bullet \ar@{->}[dr]\ar@{->}[ur]&&\bullet \ar@{->}[dr]\ar@{->}[ur]&
&\bullet \ar@{->}[dr]\ar@{->}[ur]&&\bullet\ar@{->}[dr] \ar@{->}[ur]&& \bullet \ar@{->}[dr]\ar@{->}[ur]
&&\bullet \ar@{->}[dr]\ar@{->}[ur]&& \bullet \ar@{->}[dr] \ar@{->}[ur]&&\bullet \ar@{->}[dr]\ar@{->}[ur] & \\ 
3&\bullet \ar@{->}[dr] \ar@{->}[ur]&& \bullet \ar@{->}[ddr] \ar@{->}[ur] &&\bullet\ar@{->}[dr] \ar@{->}[ur]
&&\bullet \ar@{->}[ddr] \ar@{->}[ur] && \bullet \ar@{->}[dr]\ar@{->}[ur] &&\bullet \ar@{->}[ddr] \ar@{->}[ur]&&  \bullet \ar@{->}[dr] \ar@{->}[ur]
&&\bullet \ar@{->}[ddr] \ar@{->}[ur] && \bullet \ar@{->}[dr] \ar@{->}[ur]&&\bullet \ar@{->}[ddr] \ar@{->}[ur] && \bullet\ar@{->}[dr] \ar@{->}[ur]&&
\bullet\ar@{->}[ddr] \ar@{->}[ur] && \bullet\ar@{->}[dr] \ar@{->}[ur]  && \bullet\\
4&& \bullet \ar@{->}[ur]&&&&\bullet \ar@{->}[ur] &&&& \bullet \ar@{->}[ur]
&&&& \bullet \ar@{->}[ur] &&&&\bullet \ar@{->}[ur]&&&&\bullet \ar@{->}[ur]&&
&&\bullet \ar@{->}[ur]\\
5&&&&\bullet \ar@{->}[uur]&&&&\bullet \ar@{->}[uur]&&&&  \bullet \ar@{->}[uur]&&&&
\bullet \ar@{->}[uur]&&&& \bullet \ar@{->}[uur]&&&& \bullet \ar@{->}[uur]&& }}}
$$
\item
When $\g$ is of type $\mathsf{G}_{2}$, 
we choose the pair $(\mathsf{D}_4, \tvee)$. 
The repetition quiver $\hDs$ is depicted as:
$$
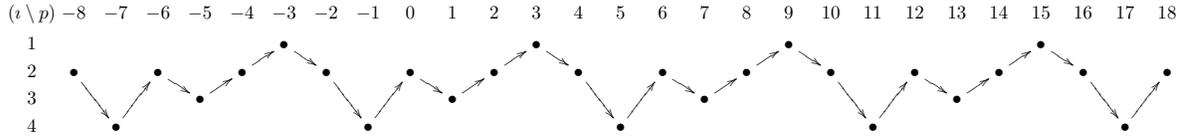
\raisebox{3mm}{
\scalebox{0.65}{\xymatrix@!C=0.1ex@R=0.5ex{
(\im\setminus p) & -8 & -7 & -6 &-5&-4 &-3& -2 &-1& 0 & 1& 2 & 3& 4&  5
& 6 & 7 & 8 & 9 & 10 & 11 & 12 & 13 & 14 & 15 & 16 & 17& 18 \\
1&&&&&&\bullet \ar@{->}[dr]&&&&&&\bullet \ar@{->}[dr]&&& 
&&&\bullet \ar@{->}[dr]&&&&&&\bullet \ar@{->}[dr]&&& \\
2&\bullet \ar@{->}[ddr]&&\bullet \ar@{->}[dr]&&\bullet\ar@{->}[ur] &&
\bullet \ar@{->}[ddr]&&\bullet \ar@{->}[dr]&&\bullet\ar@{->}[ur] &&
\bullet \ar@{->}[ddr]&&\bullet \ar@{->}[dr]&&\bullet\ar@{->}[ur] &&
\bullet \ar@{->}[ddr]&&\bullet \ar@{->}[dr]&&\bullet\ar@{->}[ur] &&
\bullet \ar@{->}[ddr]&&\bullet \\
3&&&&\bullet \ar@{->}[ur] &&
&&&&\bullet \ar@{->}[ur] &&
&&&&\bullet \ar@{->}[ur] &&
&&&&\bullet \ar@{->}[ur] &&
&&&\\
4&& \bullet \ar@{->}[uur]&&&&&& 
\bullet \ar@{->}[uur]&&&&&& 
\bullet \ar@{->}[uur]&&&&&& 
\bullet \ar@{->}[uur]&&&&&& 
\bullet \ar@{->}[uur]& \\
}}}
$$
\end{enumerate}
\end{example}
\er

When $\sigma = \id$, the repetition quiver $\hDs$ is the same as
the repetition quiver $\hD$ defined in Subsection~\ref{subsection: AR quiver}.
We denote by $\pi \colon \hDs_{0} \to \Delta_{0}$ the projection of the first components.

\beb
\begin{remark}
In \cite{HL16}, Hernandez-Leclerc introduced quivers $\Gamma$ and $G$ for untwisted affine types
whose set of vertices \emph{can be identified with} $\hDs_0$. 
We remark here that, for $\sigma=\id$, the quiver $\hD$ is isomorphic to the quivers $\Gamma$ and
$G$, when we remove the horizontal arrows $(\im,r) \to (\im,r+2)$ in $\Gamma$ and $G$ corresponding to the AR-translations. However, for $\sigma \neq \id$,
the quiver $\hDs$ is \emph{not} isomorphic to any of $\Gamma$ and $G$
even though we remove the horizontal arrows $(\im,r) \to (\im,r+2d_{\bar{\im}})$ in $\Gamma$ and $G$.
\end{remark}
\eb

\ber
\begin{definition}
Let $X \subset \hDs_{0}$ be a finite subset.
\begin{enumerate}
\item
We say that a total ordering $(x_{1}, x_{2}, \ldots, x_{l})$
of $X$
is a \emph{compatible reading of $X$} if 
we have $k<k'$
whenever there is an arrow $x_{k'} \to x_{k}$ in the quiver $\hDs$.
\item We define the element $w[X] \in \sfW$ by
\begin{equation*}
w[X] := s_{\pi(x_{1})} s_{\pi(x_{2})}\cdots s_{\pi(x_{l})},
\end{equation*}
where $(x_{1}, x_{2}, \ldots, x_{l} )$ is a compatible reading of $X$.
It does not depend on the choice of compatible reading of $X$
by Lemma~\ref{lemma: repq} below.
\end{enumerate}
\end{definition}

\begin{remark} \label{remark: repq}
Note that we have $\pi(x) \not \sim \pi(x')$ if there is no oriented path between $x$ and $x'$ in $\hDs$.
\end{remark}

\begin{lemma} \label{lemma: repq}
Let $(x_1, \ldots x_l)$ and $(x'_1, \ldots, x'_l)$ be two 
compatible readings of a finite subset $X \subset \hDs_0$.
Then the sequences $\bfi := (\pi(x_1), \ldots, \pi(x_l))$
and 
$\bfi^{\prime} := (\pi(x'_1), \ldots, \pi(x'_l))$
are commutation equivalent to each other.
\end{lemma}
\begin{proof}
Let $1 \le k \le l$ be the largest number such that
$x_{s} = x_{s}^{\prime}$ for $1 \le s \le k$.
We shall prove the assertion by downward induction on $k$. 
When $k=l$, we have $(x_1, \ldots x_l) = (x'_1, \ldots, x'_l)$
and nothing to prove.
Assume $k < l$.
Let $k + 1 < k^{\prime} \le l$ be the smallest number such that
$x_{k^{\prime}}^{\prime} = x_{k}$.
Then 
there is no oriented path between $x'_{k'}$ and $x'_t$ for all $k < t < k^{\prime}$
and hence $\pi(x_{k^{\prime}}^{\prime}) 
\not \sim \pi(x_{t}^{\prime})$ 
by Remark~\ref{remark: repq}.
Therefore the sequence 
$$
(x_{1}^{\pprime}, \ldots, x_{l}^{\pprime}) :=
(x_{1}^{\prime}, \ldots, x_{k}^{\prime}, x_{k^{\prime}}^{\prime}, x_{k+1}^{\prime}, x_{k+2}^{\prime}, \ldots,
x_{k^{\prime}-1}^{\prime}, x_{k^{\prime}+1}^{\prime}, \ldots, x_{l}^{\prime})
$$
gives a compatible reading of $X$ and the sequence
$\bfi'' := (\pi(x''_1), \ldots, \pi(x''_l) )$
is commutation equivalent to $\bfi'$.
Note that we have $x_s = x''_s$ for $1 \le s \le k+1$
and hence $\bfi^{\pprime}$ is commutation equivalent to $\bfi$
by the induction hypothesis.
Thus $\bfi^{\prime}$ is also commutation equivalent to $\bfi$.
\end{proof}
\er

\begin{lemma}
\label{lemma: compatible reading}
Let $(\im_{1}, \ldots, \im_{l})$ be a sequence of elements of $\Delta_{0}$.
We set $m_{\im} := |\{ k \mid \im_{k}=\im, 1 \le k \le l \}|$
for each $\im \in \Delta_{0}$.
Take a ${\rm Q}$-datum $\calQ=(\Delta, \sigma, \uxi)$ and assume that
\begin{equation}
\label{eq: assumption xi-m}
\text{the map
$\Delta_{0} \to \Z$ given by $\im \mapsto \uxi_{\im} -2m_{\im}d_{\bar{\im}}$
defines a height function on $(\Delta, \sigma)$.}
\end{equation}
Then the sequence $(\im_{1}, \ldots, \im_{l})$ is adapted to $\calQ$ if and only if
there exist a compatible reading $(x_{1}, \ldots, x_{l})$ of the subset
$$
X := \{ (\im, \uxi_{\im} - 2kd_{\bar{\im}} ) \in \hDs_{0} \mid 0 \le k < m_{\im}  \}
$$
of $\hDs_{0}$ such that we have $\im_{k} = \pi(x_{k})$ for all $1 \le k \le l$.
\end{lemma}
\begin{proof}
We prove the assertion by induction on $l$.
First let us assume that a given sequence $(\im_{1}, \ldots, \im_{l})$ is adapted to $\calQ$.
Then $\im_1$ is a source of $\calQ$ and the sequence $(\im_{2}, \ldots, \im_{l})$
is adapted to $s_{\im_{1}}\calQ$.
Defining
$$X^{\prime} := \{ (\im, (s_{\im_{1}}\uxi)_{\im} - 2kd_{\bar{\im}} ) \in \hDs_{0} \mid 0 \le k < m_{\im} -\delta_{\im, \im_{1}} \},$$
the induction hypothesis implies that
there exists a compatible reading $(x_{2}, \ldots, x_{l})$ of $X^{\prime}$ such that $\im_{k} = \pi(x_{k})$ for all $2 \le k \le l$.
Since $X=X^{\prime} \cup \{ (\im_{1}, \uxi_{\im_{1}})\}$, we obtain the desired compatible reading
$(x_{1}, x_{2}, \ldots, x_{l})$ by setting $x_{1} := (\im_{1}, \uxi_{\im_{1}})$.

Conversely assume that there is a compatible reading $(x_{1}, \ldots, x_{l})$ of $X$
satisfying $\im_{k} = \pi(x_{k})$ for all $1 \le k \le l$.
By the assumption~(\ref{eq: assumption xi-m}),
we see that $x_{1} = (\im_{1}, \uxi_{\im_{1}})$.
Let us prove that $\im := \im_{1}$ is a source of $\calQ$.
To deduce a contradiction, we assume the contrary that
there is a vertex $\jm \in \Delta_{0}$ such that $\im \sim \jm$ and $\uxi_{\im} < \uxi_{\jm}$.
Then we have $(\jm, \uxi_{\jm}) \not \in X$ and hence $m_{\jm} = 0$.
Therefore we have
$$(\uxi_{\jm} - 2m_{\jm}d_{\bar{\jm}}) - (\uxi_{\im} - 2m_{\im}d_{\bar{\im}}) = (\uxi_{\jm} - \uxi_{\im}) + 2m_{\im}d_{\bar{\im}} \ge
\min(d_{\bar{\im}}, d_{\bar{\jm}})+2d_{\bar{\im}},$$
which \ber contradicts \er the assumption~(\ref{eq: assumption xi-m}).
Thus $\im = \im_{1}$ is a source of $\calQ$ and hence $s_{\im_{1}}\calQ$ is well-defined.
Since $(x_{2}, \ldots, x_{l})$ is a compatible reading of the set $X \setminus \{ x_{1} \} = X^{\prime}$,
the sequence $(\im_{2}, \ldots, \im_{l})$ is adapted to $s_{\im_{1}}\calQ$ by induction hypothesis.
This completes the proof.
\end{proof}
\ber
\begin{corollary} \label{corollary: adapted seq}
Let $\calQ = (\Delta, \sigma, \uxi)$ be a {\rm Q}-datum.
Assume that two reduced words $\bfi = (\im_{1}, \ldots, \im_{l})$
and $\bfi^{\prime} = (\im_{1}^{\prime}, \ldots, \im_{l}^{\prime})$
of an element $w \in \sfW$ are both adapted to $\calQ$.
Then $\bfi^{\prime}$ is commutation equivalent to $\bfi$.
\end{corollary}
\begin{proof}
It follows from Lemma~\ref{lemma: repq} and Lemma~\ref{lemma: compatible reading}.
\end{proof}
\er
\begin{corollary} \label{cor: adpted Cartan}
Let $(\im_1,\cdots, \im_N)$ be a reduced word of $w_0$.
Then the following statements are equivalent$\colon$
\begin{enumerate}[{\rm (a)}]
\item For any $\im, \jm \in \Delta_0$ with $\im \sim \jm$
     and $k$ such that $1\le k \le k^+\le N$ and $\im=\im_k$,
     we have
$$
-c_{\bar{\jm}\bar{\im}} = \begin{cases}
|\{s\mid k<s<k^+,   \bar{\jm}=\bar{\im}_s \}| & \text{if $d_{\bar{\im}} < d_{\bar{\jm}}$}, \\[1ex]
|\{s\mid k<s<k^+,   \jm=\im_s \}| & \text{if $d_{\bar{\im}} \ge d_{\bar{\jm}}$}.
\end{cases}
$$
Here $k^+ \seteq \min\{p \mid k<p ,\; \im_k=\im_p \}$.
\item The reduced word $(\im_1,\cdots, \im_N)$ is adapted to some ${\rm Q}$-datum $\calQ$ for $\g$.
\end{enumerate}
\end{corollary}
\ber
\begin{proof}
Assume (a). For each $\im \in \Delta_0$, we set 
$k(\im) := \min\{ 1 \le k \le N \mid \im_k = \im \}$.
Since $(\im_1,\cdots, \im_N)$ is a reduced word of the longest element $w_0$,
$\{ k(\im) \mid \im \in \Delta_0\}$
is well-defined as a subset of $\{1, \ldots, N \}$ 
of cardinality $|\Delta_0|$.
Thanks to our assumption, there exists a height function 
$\uxi$ on $(\Delta, \sigma)$ uniquely up to adding an integer 
so that we have 
(i) $\uxi_\im > \uxi_\jm$ if $\im \sim \jm$ and $k(\im) < k(\jm)$,
and (ii) $k(j^\circ) < k(\sigma (j^\circ)) < \cdots < k(\sigma^{r-1} (j^\circ))$
if $d_i < d_j$ and $i \sim j$
(with the notation in Definition~\ref{def: height}).
The same assumption also enables us to prove that the sequence
$(\im_1,\cdots, \im_N)$ is adapted to the resulting Q-datum
$\calQ = (\Delta, \sigma, \uxi)$. 
This proves the implication (a) $\Rightarrow$ (b).

The other implication (b) $\Rightarrow$ (a) 
follows from Lemma~\ref{lemma: compatible reading} immediately.  
\end{proof}
\er

\subsection{Twisted Auslander-Reiten quivers}
\label{subsection: twisted AR quiver}
Mimicking the characterization~(\ref{eq: 2-segment}) of the usual AR quiver $\Gamma_{Q}$ for a Dynkin quiver $Q$,
we give the following definition.
\begin{definition}
Let $\calQ = (\Delta, \sigma, \uxi)$ be a Q-datum for $\g$. We define
the \emph{twisted Auslander-Reiten quiver $\Gamma_{\calQ}$} of
$\calQ$ as the full subquiver of $\hDs$ whose vertex set
$(\Gamma_{\calQ})_{0}$ is given by
\begin{equation}
(\Gamma_{\calQ})_{0} := \{(\im, \uxi_{\im}-2d_{\bar{\im}}k) \in \hDs_{0} \mid
0 \le 2d_{\bar{\im}}k < rh^{\vee} + \uxi_{\im}-\uxi_{\im^{*}}\}.
\end{equation}
\end{definition}

\begin{example}
\label{example: GammaB3}
The twisted AR quivers associated with the Q-data $\calQ^{(1)}$ and $\calQ^{(2)}$ for $\g$ of type $\mathsf{B}_{3}$
given in Example~\ref{example: QB3} above are depicted as follows:
$$ \ber \Gamma_{\calQ^{(1)}} =
\raisebox{3.3em}{\scalebox{0.63}
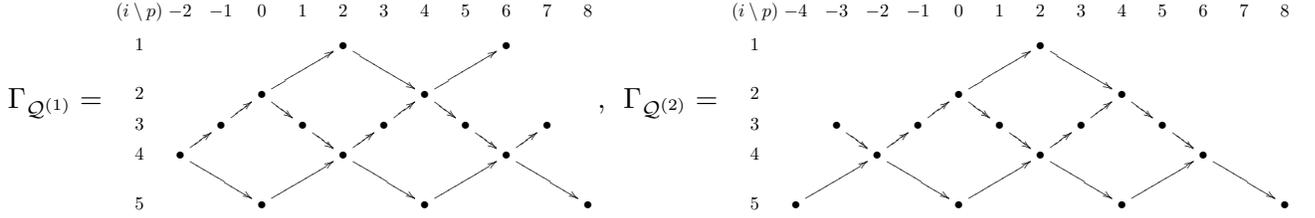
{\xymatrix @!C=0.1ex@R=1ex{
(i\setminus p) & -2 &-1&   0 &  1&   2 &  3&   4& 5&   6 &  7 & 8\\
1&&&&&  \bullet \ar@{->}[ddrr] &&&& \bullet \\ \\
2&&& \bullet \ar@{->}[dr]\ar@{->}[uurr] &&&&  \bullet \ar@{->}[dr] \ar@{->}[uurr] \\
3&& \bullet \ar@{->}[ur] && \bullet \ar@{->}[dr] && \bullet\ar@{->}[ur] && \bullet \ar@{->}[dr] && \bullet
\\
4&\bullet\ar@{->}[ddrr]\ar@{->}[ur] &&&&  \bullet \ar@{->}[ddrr]\ar@{->}[ur] &&&&  \bullet \ar@{->}[ur]   \ar@{->}[ddrr] \\ \\
5&&&  \bullet \ar@{->}[uurr] &&&& \bullet   \ar@{->}[uurr]   &&&& \bullet }}}, \er
 \ \
\Gamma_{\calQ^{(2)}} =
\raisebox{3.3em}{\scalebox{0.63}{\xymatrix @!C=0.1ex@R=1ex{
(i\setminus p) & -4 &-3& -2 &-1& 0 &1&2 & 3& 4& 5&  6 & 7 & 8 \\
1&&& &&&&  \bullet \ar@{->}[ddrr]  \\ \\
2&&&&& \bullet \ar@{->}[dr]\ar@{->}[uurr] &&&&  \bullet \ar@{->}[dr] \\
3&&\bullet \ar@{->}[dr] && \bullet \ar@{->}[ur] && \bullet \ar@{->}[dr] && \bullet \ar@{->}[ur] && \bullet \ar@{->}[dr] &&
\\
4&&&\bullet\ar@{->}[ddrr]\ar@{->}[ur] &&&&  \bullet \ar@{->}[ddrr]\ar@{->}[ur] &&&&  \bullet \ar@{->}[ddrr] \\ \\
5&\bullet \ar@{->}[uurr] &&&&  \bullet \ar@{->}[uurr] &&&& \bullet   \ar@{->}[uurr] &&&& \bullet }}}
$$
\end{example}

The following theorem gives a twisted analog of the canonical bijection (\ref{eq: bijection2}).

\begin{theorem}[cf.~\cite{OS19b}]
\label{theorem: OS19b main}
For each ${\rm Q}$-datum $\calQ=(\Delta, \sigma, \uxi)$ for $\g$,
we denote by $[\calQ]$ the set of
all the reduced words for the longest element $w_{0} \in \sfW$ adapted to $\calQ$.
Then $[\calQ]$ forms a single commutation class and there
is a unique isomorphism $\Gamma_{\calQ} \cong \Upsilon_{[\calQ]}$ of quivers
which intertwines $\pi$ and $\res^{[\calQ]}$.
Moreover, the assignment $\calQ \mapsto [\calQ]$ gives a bijection
$$
\{ \text{Q-datum for $\g$}\}/2r\Z \quad \overset{1:1}{\longleftrightarrow} \quad \lf \Delta^{\!\sigma} \rf,
$$
where $/2r\Z$ means that we ignore constant differences between height functions.
\end{theorem}
\begin{proof}
In the case $\sigma = \id$, the assertion goes back to the previous section.
Let us focus on the case $\sigma \neq \id$.
We pick a compatible reading $(x_{1}, \ldots, x_{N})$ of $(\Gamma_{\calQ})_{0}$
and set $\bfi := (\pi(x_{1}), \ldots, \pi(x_{N})) \in \Delta_{0}^{N}$.
Here we used the fact \ber$|(\Gamma_{\calQ})_{0}|=N=\ell(w_0)$\er.
By construction, the commutation class $[\bfi]$ does not depend on the choice of the compatible reading
and hence we denote it by $[\calQ]^{\prime}$.
Then the following results have been obtained in \cite[Section 4]{OS19b}:
\begin{itemize}
\item The sequence $\bfi$ obtained as above gives a reduced word for $w_{0}$,
\item The assignment $\calQ \mapsto [\calQ]^{\prime}$ yields a bijection
$
\{ \text{Q-datum for $\g$}\}/2r\Z \overset{1:1}{\longrightarrow} \lf \Delta^{\!\sigma} \rf,
$
\item We have a unique isomorphism $\Gamma_{\calQ} \cong \Upsilon_{[\calQ]^{\prime}}$ of quivers
which intertwines $\pi$ and $\res^{[\calQ]^{\prime}}$.
\end{itemize}
Thanks to Lemma~\ref{lemma: compatible reading} and \ber Corollary~\ref{corollary: adapted seq}, \er
we have $[\calQ] = [\calQ]^{\prime}$, which completes the proof.
\end{proof}

Let $\phi_{\calQ, 0} \colon (\Gamma_{\calQ})_{0} \to \sfR^{+}$ denote the underlying bijection
of the isomorphism $\Gamma_{\calQ} \cong \Upsilon_{[\calQ]}$ in Theorem~\ref{theorem: OS19b main}.
For $\beta \in \sfR^{+}$, we call $(\im,p) = \phi_{\calQ, 0}^{-1}(\beta)$
the \emph{coordinate of $\beta$ in $\Gamma_{\calQ}$}.
Theorem~\ref{theorem: OS19b main} implies that if $\beta \in \sfR^{+}$ is located at the coordinate $(\im, p)$ in $\Gamma_{\calQ}$,
we have $\res^{[\calQ]}(\beta) = \im$.

The following is a twisted analogue of the algorithm in the last paragraph of Subsection~\ref{subsection: AR quiver}.

\begin{proposition}[\cite{OS19b}] \label{proposition: reflection algorithm}
Let $\calQ = (\Delta, \sigma, \uxi)$ be a Q-datum.
For a source $\im \in \Delta_{0}$ of $\calQ$,
we have $r_{\im}[\calQ] = [s_{\im}\calQ]$.
The bijection $\phi_{s_{\im}\calQ, 0} \colon (\Gamma_{s_{\im}\calQ})_{0} \to \sfR^{+}$
can be obtained from the bijection $\phi_{\calQ, 0} \colon (\Gamma_{\calQ})_{0} \to \sfR^{+}$ in the following way:
\begin{enumerate}
\item A positive root $\beta \in \sfR^{+} \setminus \{\alpha_{\im}\}$ is located at the coordinate $(\jm, p)$ in $\Gamma_{s_{i}\calQ}$
if $s_{\im}\beta$ is located at the coordinate $(\jm,p)$ in $\Gamma_{\calQ}$,
\item The simple root $\alpha_{\im}$ is located at the coordinate $(\im^{*}, \uxi_{\im}-rh^{\vee})$ in $\Gamma_{s_{\im}\calQ}$
while $\alpha_{\im}$ was located at the coordinate $(\im,\uxi_{\im})$ in $\Gamma_\calQ$.
\end{enumerate}
\end{proposition}

\subsection{Generalizations of Coxeter element}
First we introduce a generalization of the $r$-th power $\tau^{r}$
of a $\sigma$-twisted Coxeter element $\tau$ for any Q-datum.
Note that we have $\sum_{\im \in \Delta_{0}} r_{\bar{\im}} = nr$.

\begin{proposition}
\label{proposition: quasi Coxeter}
For each ${\rm Q}$-datum $\calQ=(\Delta, \sigma, \uxi)$,
there exists a unique element $\btau_{\calQ} \in \sfW$
with $\ell(\btau_{\calQ}) = nr$
which has a reduced word $(\im_{1}, \ldots, \im_{nr})$
adapted to $\calQ$
such that
$$|\{ k \mid 1 \le k \le nr, \im_{k}=\im \}|=r_{\bar{\im}} \quad \text{ for each $\im \in \Delta_{0}$.}$$
\end{proposition}
\begin{proof}
Define the subset $X_{\calQ} \subset \hDs_{0}$ by
$$
X_{\calQ} := \{ (\im, \uxi_{\im} - 2kd_{\bar{\im}}) \in \hDs_{0} \mid 0 \le k < r_{\bar{\im}} \}.
$$
Note that the condition~(\ref{eq: assumption xi-m}) is satisfied when $m_{\im} = r_{\bar{\im}}$
because we have $r_{\bar{\im}}d_{\bar{\im}} = r$ for any $\im \in \Delta_{0}$.
Let $(x_{1}, \ldots, x_{nr})$ be a compatible reading of $X_{\calQ}$.
By construction, it can be extended to a compatible reading $(x_{1}, \ldots, x_{N})$ of $(\Gamma_{\calQ})_{0}$.
By Theorem~\ref{theorem: OS19b main}, the sequence
$(\im_{1}, \ldots, \im_{N}) := (\pi(x_{1}), \ldots, \pi(x_{N}))$
gives a reduced word for $w_{0}$.
In particular, the subsequence $(\im_{1}, \ldots, \im_{nr}) = (\pi(x_{1}), \ldots, \pi(x_{nr}))$ gives
a reduced word for the element $\btau_{\calQ} := w[X_{\calQ}] = s_{\im_{1}} \cdots s_{\im_{nr}}$,
which is adapted to $\calQ$ thanks to Lemma~\ref{lemma: compatible reading}.
The uniqueness also follows from Lemma~\ref{lemma: compatible reading}.
\end{proof}

\begin{definition} \label{def: quasi Coxeter element}
We refer to the unique element $\btau_{\calQ} \in \sfW$ in Proposition~\ref{proposition: quasi Coxeter}
as the \emph{quasi Coxeter element associated with $\calQ$}.
\end{definition}

\begin{corollary}
\label{corollary: quasi Coxeter}
Let $\calQ=(\Delta, \sigma, \uxi)$ be a ${\rm Q}$-datum and $\im \in \Delta_{0}$ be a source of $\calQ$.
Then we have $$s_{\im}\btau_{\calQ}s_{\im}=\btau_{s_{\im}\calQ}.$$
\end{corollary}
\begin{proof}
Let $(x_{1}, \ldots, x_{nr})$ be a compatible reading of $X_{\calQ}$ such that $x_{1} = (\im, \uxi_{\im})$.
By Proposition~\ref{proposition: quasi Coxeter}, we get a reduced word
$(\im_{2}, \ldots, \im_{nr}) := (\pi(x_{2}), \ldots, \pi(x_{nr}))$ for the element $s_{\im}\btau_{\calQ}$,
and hence $s_{\im}\btau_{\calQ} = w[X^{\prime}]$ with $X^{\prime} := X_{\calQ} \setminus \{(\im, \uxi_{\im})\}$.
On the other hand, the sequence $(\im_{2}, \ldots, \im_{nr}, \im)$
is adapted to $s_{\im}\calQ$ and hence there is
a compatible reading $(x^{\prime}_{1}, \ldots, x^{\prime}_{nr})$ of $X_{s_{\im}\calQ}$
such that $x^{\prime}_{nr} = (\im, \uxi_{\im} - 2r)$ by Lemma~\ref{lemma: compatible reading}.
Thus we have $\btau_{s_{\im}\calQ} s_{\im} = w[X^{\prime\prime}]$ where
$X^{\prime \prime} := X_{s_{\im}\calQ} \setminus \{ (\im, \uxi_{\im} - 2r)\}$.
Since $X^{\prime} = X^{\prime\prime}$, we obtain the conclusion.
\end{proof}

Next we consider a generalization of twisted Coxeter element for any Q-datum $\calQ$.
\begin{definition}
Let $\calQ=(\Delta, \sigma, \uxi)$ be a Q-datum for $\g$.
For each $i \in I$, we denote by $i^{\circ}$ an element in the $\sigma$-orbit $i$ satisfying the condition
$$
\uxi_{i^{\circ}} = \max\{\uxi_{\im} \mid \im \in i\}.
$$
By Lemma~\ref{lemma: parity}~(\ref{parity_injection}), $i^{\circ} \in i$ is uniquely determined.
The subset $I_{\calQ}^{\circ} := \{i^{\circ} \mid i \in I\} \subset \Delta_{0}$
gives a section of the natural quotient map $\bar{ \ } \colon \Delta_{0} \to I$.
We denote the corresponding $\sigma$-Coxeter element by
$\tau_{\calQ}^{\circ} := w[X_{\calQ}^{\circ}] \sigma$,
where $X_{\calQ}^{\circ} := \{ (\im, \uxi_{\im}) \mid \im \in I_{\calQ}^{\circ} \} \subset \hDs_{0}$.
\end{definition}

\begin{lemma}
\label{lemma: twisted Coxeter height}
Let $i, j \in I$ with $i \sim j$.
\ber With \er the above notation, we have:
\begin{enumerate}
\item \label{circ_adj} $|\uxi_{i^{\circ}} -\uxi_{j^{\circ}}| = \min(d_{i}, d_{j})$ if $i^{\circ} \sim j^{\circ}$,
\item \label{circ_nonadj} $\uxi_{i^{\circ}}=\uxi_{j^{\circ}}$ if $i^{\circ} \not \sim j^{\circ}$.
\end{enumerate}
\end{lemma}
\begin{proof}
(\ref{circ_adj}) is immediate from Definition~\ref{def: height}.
For (\ref{circ_nonadj}), we note that the conditions $i \sim j$ and $i^{\circ} \not \sim j^{\circ}$
are satisfied only if $d_{i}=d_{j}=r=2$.
In this case, we have $\uxi_{i^{\circ}} \ge \uxi_{\sigma(i^{\circ})} +2$
and $\uxi_{j^{\circ}} \ge \uxi_{\sigma(j^{\circ})} +2$.
Moreover we have $i^{\circ}\sim \sigma(j^{\circ})$ and $j^{\circ} \sim \sigma(i^{\circ})$.
If $\uxi_{i^{\circ}} > \uxi_{j^{\circ}}$, we have
$\uxi_{i^{\circ}} > \uxi_{\sigma(j^{\circ})} +2$, which \ber contradicts the condition ($\mathtt{H1}$): \er
$|\uxi_{i^{\circ}} - \uxi_{\sigma(j^{\circ})}|=2$.
Similarly we \ber cannot \er have $\uxi_{i^{\circ}} < \uxi_{j^{\circ}}$.
Thus we obtain $\uxi_{i^{\circ}} = \uxi_{j^{\circ}}$.
\end{proof}

\begin{proposition}
\label{proposition: twisted Coxeter height}
A twisted adapted class $[\calQ] \in \lf \Delta^{\!\sigma}\rf$
contains a reduced word arising from a $\sigma$-Coxeter element,
i.e.~$[\calQ] = [\bfi(\tau)]$ for some $\tau$
if and only if the corresponding {\rm Q}-datum $\calQ=(\Delta, \sigma, \uxi)$ satisfies the following condition:
\begin{equation}
\label{eq: condition twisted Coxeter height}
\text{For each $i \in I$ and $0 \le k< d_{i}$, we have
$\uxi_{\sigma^{k}(i^{\circ})} = \uxi_{i^{\circ}} - 2k$.}
\end{equation}
Moreover, if this condition~(\ref{eq: condition twisted Coxeter height}) is satisfied,
we have $[\calQ] = [\bfi(\tau_{\calQ}^{\circ})]$ and $(\tau_{\calQ}^{\circ})^{r} = \btau_{\calQ}$.
\end{proposition}
\begin{proof}
Let $\tau = s_{\im_{1}} \cdots s_{\im_{n}} \sigma \in \sfW \sigma$ be a $\sigma$-Coxeter element.
Let us choose a function $\uxi^{\tau} \colon \Delta_{0} \to \Z$ satisfying the following three conditions,
which is unique up to constant:
\begin{enumerate}
\item \label{xitau_adj}
$\uxi^{\tau}_{\im_{k}} = \uxi^{\tau}_{\im_{l}} + \min(d_{\bar{\im}_{k}}, d_{\bar{\im}_{l}})$
if $1 \le k < l \le n$ and $\im_{k} \sim \im_{l}$,
\item \label{xitau_nonadj}
$\uxi^{\tau}_{\im_{k}} = \uxi^{\tau}_{\im_{l}}$ if $\bar{\im}_{k} \sim \bar{\im}_{l}$ and $\im_{k} \not\sim \im_{l}$,
\item \label{xitau_sigma}
$\uxi^{\tau}_{\sigma^{l}(\im_{k})} = \uxi^{\tau}_{\im_{k}} - 2l$ for each $1 \le k \le n$ and $0 \le l <d_{\bar{\im}_{k}}$.
\end{enumerate}
Note that (\ref{xitau_nonadj}) occurs only when $d_{\bar{\im}_{k}} = d_{\bar{\im}_{l}} = r = 2$.
It is easy to see that $\uxi^{\tau}$ defines a height function on $(\Delta, \sigma)$.
Hence we obtain a Q-datum
$\calQ^{\tau} := (\Delta, \sigma, \uxi^{\tau})$ satisfying the condition~(\ref{eq: condition twisted Coxeter height}).
By construction, we have $I_{\calQ^{\tau}}^{\circ} = \{ \im_{k} \mid 1 \le k \le n\}$
and $\tau_{\calQ^{\tau}}^{\circ} = \tau$.
Conversely, every Q-datum $\calQ$ satisfying the condition~(\ref{eq: condition twisted Coxeter height})
is obtained in this way, i.e.~we can realize $\calQ=\calQ^{\tau}$
for a unique $\sigma$-Coxeter element $\tau$
by Lemma~\ref{lemma: twisted Coxeter height}.

It remains to show that $\tau^{r} = \btau_{\calQ^{\tau}}$ and $[\bfi(\tau)] = [\calQ^{\tau}]$.
In the case $\sigma = \id$, we have nothing to prove.
Assume that $\sigma \neq \id$.
We shall apply Lemma~\ref{lemma: compatible reading} to the set
$X_{\calQ^{\tau}}^{\circ} = \{(\im_{k}, \uxi^{\tau}_{\im_{k}}) \mid 1 \le k \le n \}$.
Note that the condition~(\ref{eq: assumption xi-m})
is satisfied for $m_{\im} = \delta(\im \in I_{\calQ^{\tau}}^{\circ})$
since we have
\begin{equation}
\label{eq: xi^prime}
\uxi_{\im}^{\prime} := \uxi^{\tau}_{\im} - 2d_{\bar{\im}} \times \delta(\im \in I_{\calQ^{\tau}}^{\circ})
= \uxi^{\tau}_{\sigma^{-1}(\im)} -2
\end{equation}
by (\ref{xitau_sigma}).
As a result, the sequence $(\im_{1}, \ldots, \im_{n})$ is adapted to $\calQ^{\tau}$
and $s_{\im_{n}} \cdots s_{\im_{1}}\uxi^{\tau} = \uxi^{\prime}$.
Then the equation~(\ref{eq: xi^prime})
also implies that the sequence
$(\sigma(\im_{1}), \ldots, \sigma(\im_{n}))$ is adapted to $(\Delta, \sigma, \uxi^{\prime})$.
Repeating this argument, we see that the sequence $\bfi(\tau)$ defined in Theorem~\ref{theorem: OSmain}
is adapted to $\calQ^{\tau}$.
Therefore, we obtain $\tau^{r} = \btau_{\calQ^{\tau}}$,
and also
$[\bfi(\tau)] = [\calQ^{\tau}]$ by Theorem~\ref{theorem: OS19b main}.
\end{proof}

\begin{corollary}
\label{corollary: twisted Coxeter height}
Except for $\g=\mathsf{B}_{n}$ or $\mathsf{F}_{4}$,
every twisted adapted class contains a reduced word arising from a $\sigma$-Coxeter element,
i.e.~$\lf \Delta^{\! \sigma} \rf = \{ [\bfi(\tau)] \mid \text{$\tau$ is a $\sigma$-Coxeter element.}\}$.
\end{corollary}
\begin{proof}
Unless $\g$ is of type $\mathsf{B}_{n}$ nor of type $\mathsf{F}_{4}$, every Q-datum for $\g$
satisfies the condition~(\ref{eq: condition twisted Coxeter height}).
Therefore Proposition~\ref{proposition: twisted Coxeter height} proves the assertion.
\end{proof}

Let $\calQ=(\Delta, \sigma, \uxi)$ be a Q-datum.
Associated with $\uxi$, we define another height function $\uxi^{\circ} \colon \Delta_{0} \to \Z$
by
$\uxi^{\circ}_{\sigma^{k}(\im)} =
\uxi_{\im} -2k$ for $\im \in I_{\calQ}^{\circ}$ and $0 \le k < d_{\bar{\im}}$.
Let $\calQ^{\circ} := (\Delta, \sigma, \uxi^{\circ})$ be the corresponding Q-datum.
Proposition~\ref{proposition: twisted Coxeter height} shows that we have $[\bfi(\tau_{\calQ}^{\circ})] = [\calQ^{\circ}]$.
Let
$$
X_{\calQ}^{\prime} := \{ (\sigma(i^{\circ}), p) \in \hDs_{0} \mid i \in I, \uxi_{\sigma(i^{\circ})} < p \le \uxi_{i^{\circ}} -2 \}.
$$
By definition, $X_{\calQ}^{\prime} = \varnothing$ if and only if
the condition~(\ref{eq: condition twisted Coxeter height}) is satisfied.
Observe that we have $\im \not \sim \jm$
if $\im \in I_{\calQ}^{\circ}$ and $\jm \in \pi(X_{\calQ}^{\prime})$.

Pick a compatible reading $(y_{1}, \ldots, y_{m})$ of $X_{\calQ}^{\prime}$,
where $m=|X_{\calQ}^{\prime}|$.
By Lemma~\ref{lemma: compatible reading}, the sequence $(\jm_{1}, \ldots, \jm_{m}) := (\pi(y_{1}), \ldots, \pi(y_{m}))$
is adapted to $\calQ^{\circ}$ and we have
\begin{equation}
\label{eq: translation}
\calQ = s_{\jm_{m}} \cdots s_{\jm_{1}}\calQ^{\circ}.
\end{equation}
Recall that we have defined $w[X_{\calQ}^{\prime}] := s_{\jm_{1}} \cdots s_{\jm_{m}}$.

\begin{definition} \label{def: generalized C-element}
Under the above notation, we define the \emph{generalized $\sigma$-Coxeter element}
(or the \emph{generalized twisted Coxeter element}) $\tau_{\calQ} \in \sfW \sigma$ associated with $\calQ$ by
$$
\tau_{\calQ} := w[X_{\calQ}^{\prime}]^{-1} \cdot \tau_{\calQ}^{\circ} \cdot w[X_{\calQ}^{\prime}].
$$
\end{definition}

Note that
the generalized twisted Coxeter element $\tau_{\calQ}$ is equal to the twisted Coxeter element $\tau_{\calQ}^{\circ}$
if and only if the condition~(\ref{eq: condition twisted Coxeter height}) is satisfied.
In particular, we have $\tau_{\calQ^{\circ}} = \tau^{\circ}_{\calQ}$.

\begin{proposition}
\label{proposition: conjugation}
Let $\calQ = (\Delta, \sigma, \uxi)$ be a {\rm Q}-datum.
\begin{enumerate}
\item \label{tauconj_conj} If $\im \in \Delta_{0}$ is a source of $\calQ$,
we have $s_{\im} \tau_{\calQ} s_{\im} = \tau_{s_{\im}\calQ}$.
\item \label{tauconj_ord} The order of $\tau_{\calQ}$ is $rh^{\vee}$.
\item \label{tauconj_qtau} We have $\tau_{\calQ}^{r} = \btau_{\calQ}$. Hence the order of $\btau_{\calQ}$ is $h^{\vee}$.
\item \label{tauconj_rh} If $\sigma \neq \id$, we have $\tau_{\calQ}^{rh^{\vee}/2} = -1$.
\end{enumerate}
\end{proposition}
\begin{proof}
The assertions (\ref{tauconj_ord}), (\ref{tauconj_qtau}) and (\ref{tauconj_rh})
follow easily from the assertion (\ref{tauconj_conj}), Theorem~\ref{theorem: OSmain},
Corollary~\ref{corollary: quasi Coxeter} and Proposition~\ref{proposition: twisted Coxeter height}.
Let us prove the remaining assertion (\ref{tauconj_conj}).
First we assume $\im \not\in I_{\calQ}^{\circ}$.
Then we have $X^{\circ}_{s_{\im}\calQ} = X^{\circ}_{\calQ}$ and
$X^{\prime}_{s_{\im}\calQ} = X^{\prime}_{\calQ} \cup \{ (\im, \uxi_{\im} - 2d_{\bar{\im}}) \}$,
which imply that $\tau^{\circ}_{s_{\im}\calQ} = \tau^{\circ}_{\calQ}$ and
$w[X^{\prime}_{s_{\im}\calQ}] = w[X^{\prime}_{\calQ}] s_{\im}$.
Thus we obtain $\tau_{s_{\im}\calQ} = s_{\im} \tau_{\calQ} s_{\im}$ as desired.
Now we assume $\im \in I_{\calQ}^{\circ}$.
Take a compatible reading $(x_{1}, \ldots, x_{n})$ (resp.~$(y_{1}, \ldots, y_{m})$) of $X_{\calQ}^{\circ}$ (resp.~$X_{\calQ}^{\prime}$)
and set $\im_{k} := \pi(x_{k})$, $\jm_{l} := \pi(y_{l})$.
Since $\im$ is a source of $\calQ$, we may assume that $\im_{1}=\im$.
By definition, we have
$$
\tau_{\calQ} = s_{\jm_{m}} \cdots s_{\jm_{1}} s_{\im} s_{\im_{2}} \cdots s_{\im_{n}} s_{\sigma(\jm_{1})} \cdots s_{\sigma(\jm_{m})} \sigma.
$$
Since $\im \not\sim \jm_{l}$ and
$\im_{k} \not \sim \jm_{l}$ for any $k, l$, we have
\begin{equation}
\label{eq: sts=t}
s_{\im}\tau_{\calQ}s_{\im} =
s_{\jm_{m}} \cdots s_{\jm_{1}} s_{\im_{2}} \cdots s_{\im_{n}} s_{\sigma(\im)} s_{\sigma(\jm_{1})} \cdots s_{\sigma(\jm_{m})} \sigma.
\end{equation}
Assume $\uxi_{\sigma(\im)} = \uxi_{\im} -2$.
Then we have
$X^{\circ}_{s_{\im}\calQ}= (X^{\circ}_{\calQ} \setminus \{(\im, \uxi_{\im})\})\cup \{(\sigma(\im), \uxi_{\sigma(\im)})\}$
and $X^{\prime}_{s_{\im}\calQ} = X^{\prime}_{\calQ}$,
which imply
$s_{\im_{2}} \cdots s_{\im_{n}} s_{\sigma(\im)} = \tau^{\circ}_{s_{\im} \calQ}$ and
$s_{\jm_{1}} \cdots s_{\jm_{m}} = w[X^{\prime}_{s_{\im} \calQ}]$ respectively.
Hence the RHS of (\ref{eq: sts=t}) is equal to $\tau_{s_{\im}\calQ}$ as desired.
Finally we assume $\uxi_{\sigma(\im)} < \uxi_{\im} -2$.
Note that it forces $r=d_{\bar{\im}} = 2$ in this case.
Since $\im$ is a source of $\calQ$,
we may assume that $y_{1} = (\sigma(\im), \uxi_{\im} -2)$ and hence $\jm_{1}=\sigma(\im)$.
Since $\sigma(\im) \not \sim \im_{k}$ for all $2 \le k \le n$, the equation~(\ref{eq: sts=t}) is rewritten as
\begin{align}
\label{eq: sts=t2}
s_{\im}\tau_{\calQ}s_{\im} &=
s_{\jm_{m}} \cdots s_{\jm_{2}} s_{\sigma(\im)}
s_{\im_{2}} \cdots s_{\im_{n}} s_{\sigma(\im)} s_{\im} s_{\sigma(\jm_{2})} \cdots s_{\sigma(\jm_{m})} \sigma  \nonumber \\
&= s_{\jm_{m}} \cdots s_{\jm_{2}}
s_{\im_{2}} \cdots s_{\im_{n}} s_{\im} s_{\sigma(\jm_{2})} \cdots s_{\sigma(\jm_{m})} \sigma.
\end{align}
On the other hand, we have
$X^{\circ}_{s_{\im}\calQ} = (X^{\circ}_{\calQ} \setminus \{(\im, \uxi_{\im})\})\cup \{(\im, \uxi_{\im} -4)\}$
and $X^{\prime}_{s_{\im}\calQ} = X^{\prime}_{\calQ} \setminus \{ (\sigma(\im), \uxi_{\im} -2)\}$,
which imply
$s_{\im_{2}} \cdots s_{\im_{n}} s_{\im} = \tau^{\circ}_{s_{\im} \calQ}$ and
$s_{\jm_{2}} \cdots s_{\jm_{m}} = w[X^{\prime}_{s_{\im} \calQ}]$ respectively.
Therefore the RHS of the equation~(\ref{eq: sts=t2}) is equal to $\tau_{s_{\im}\calQ}$ as desired.
\end{proof}

\subsection{The bijection $\phi_{\calQ}$ and $\g$-additive property} \label{subsec: bijection}
\beb In this subsection, we extend the bijection  
$\phi_Q: \hD_{0} \to \hsfR^+$ (see \cite[Section 2.2]{HL15}), induced from a Dynkin quiver $Q$ and its Coxeter element $\tau_Q$, to the bijection  $\phi_{\calQ} \colon \hDs_{0} \to \hsfR^+$ by using a Q-datum $\calQ= (\Delta, \sigma, \uxi)$ and its generalized twisted Coxeter element $\tau_\calQ$. Then we will prove $\g$-additive property, which is also well-known for $\calQ$
 being a Dynkin quiver $Q$.  
\eb

\smallskip

For each $\im \in \Delta_{0}$, we assign a positive root by
$$
\gamma_{\im}^{\calQ} := (1-\tau_{\calQ}^{d_{\bar{\im}}})\varpi_{\im}.
$$
Then we define the map $\phi_{\calQ} \colon \hDs_{0} \to \hsfR^+$ recursively by
the following rules:
\begin{enumerate}
\item $\phi_{\calQ}(\im, \uxi_{\im})=(\gamma^{\calQ}_{\im},0)$ for each $\im \in \Delta_{0}$.
\item If $\phi_\calQ(\im,p)=(\beta,k)$, we have
$$
\phi_{\calQ}(\im, p\pm2d_{\bar{\im}}) =
\begin{cases}
(\tau_{\calQ}^{\mp d_{\bar{\im}}}(\beta),k) & \text{if $\tau_{\calQ}^{\mp d_{\bar{\im}}}(\beta) \in \sfR^+$}, \\
(-\tau_{\calQ}^{\mp d_{\bar{\im}}}(\beta), k\pm1) & \text{if $\tau_{\calQ}^{\mp d_{\bar{\im}}}(\beta) \in \sfR^-$}.
\end{cases}
$$
\end{enumerate}
By definition, we have $\tau_{\calQ}^{(\uxi_{\im} -p)/2}(\gamma^{\calQ}_{\im}) = (-1)^{k}\beta$
if $\phi_{\calQ}(\im, p) = (\beta, k)$.

\begin{theorem}
\label{thm: phi_Q}
Let $\calQ= (\Delta, \sigma, \uxi)$ be a ${\rm Q}$-datum.
\begin{enumerate}
\item The map $\phi_{\calQ} \colon \hDs_{0} \to \hsfR^+$ is a bijection.
\item We have $\phi_{\calQ}^{-1}(\sfR^{+} \times \{ 0 \}) = (\Gamma_{\calQ})_{0}$
and $\phi_{\calQ}(\im, p) = (\phi_{\calQ, 0}(\im, p), 0)$ for each $(\im, p) \in (\Gamma_{\calQ})_{0}$.
Here $\phi_{\calQ, 0} \colon (\Gamma_{\calQ})_{0} \to \sfR^{+}$ is the underlying bijection of
the isomorphism $\Gamma_{\calQ} \cong \Upsilon_{[\calQ]}$ in Theorem~\ref{theorem: OS19b main}.
In particular, we have
$$
\phi_{\calQ,0}(\im, p) = \tau_{\calQ}^{(\uxi_{\im}-p)/2}(\gamma^{\calQ}_{\im})
$$
for each $(\im, p) \in (\Gamma_{\calQ})_{0}$.
\end{enumerate}
\end{theorem}

For a proof, we need some lemmas.

\begin{lemma}
\label{lemma: gamma}
Let $(\im_{1}, \ldots, \im_{nr})$ be a reduced word for the quasi Coxeter element $\btau_{\calQ}$
adapted to $\calQ$
such that $|\{ k \mid \im_{k} = \im \}| = r_{\bar{\im}}$ for each $\im \in \Delta_{0}$.
Then, for each $\im \in \Delta_0$, we have
$$
\gamma_{\im}^{\calQ} = s_{\im_{1}}s_{\im_{2}} \cdots s_{\im_{k-1}} \alpha_{\im},
$$
where $k$ is the smallest number such that $\im_{k} = \im$.
In particular, we have $\gamma^{\calQ}_{\im} = \alpha_{\im}$ if and only if $\im$ is a source of $\calQ$.
\end{lemma}
\begin{proof}
 \ber In view of Lemma~\ref{lemma: repq} and Lemma~\ref{lemma: compatible reading}, \er it suffices to prove the assertion for a special reduced word $(\im_{1}, \ldots, \im_{nr})$
for $\btau_{\calQ}$ satisfying the above conditions.
Let us take a compatible reading $(x_{1}, \ldots, x_{nr})$ of the set $X_{\calQ}$
such that the first $n$-letters $(x_{1}, \ldots, x_{n})$ gives a compatible reading of
the subset $X_{\calQ}^{\circ} \subset X_{\calQ}$.
Then we have $\btau_{\calQ} = s_{\im_{1}} \cdots s_{\im_{nr}}$
and $\tau_{\calQ}^{\circ} = s_{\im_{1}} \cdots s_{\im_{n}} \sigma$,
where $(\im_{1}, \ldots, \im_{nr}) = (\pi(x_{1}), \ldots, \pi(x_{nr}))$.
Let $k$ be the smallest number such that $\im_{k} = \im$.
If $d_{\bar{\im}} = r$, we have $\tau_{\calQ}^{d_{\bar{\im}}} = \btau_{\calQ}$
by Proposition~\ref{proposition: conjugation}~(3)
and $\im_{l} = \im$ only if $l = k$.
Therefore we have $\btau_{\calQ} \varpi_{\im} = \varpi_{\im} - s_{\im_{1}} \cdots s_{\im_{k-1}} \alpha_{\im}$
and hence $\gamma^{\calQ}_{\im} = s_{\im_{1}} \cdots s_{\im_{k-1}} \alpha_{\im}$.
If $d_{\bar{\im}} = 1$, we have
$\tau_{\calQ}^{d_{\bar{\im}}} = \tau_{\calQ} = w[X_{\calQ}^{\prime}]^{-1} \cdot \tau_{\calQ}^{\circ} \cdot w[X_{\calQ}^{\prime}]$
and $1 \le k \le n$.
Note that $w[X_{Q}^{\prime}] \cdot \varpi_{\im} = \varpi_{\im}$ and
$\ber s_{\im_l} \er w[X_{Q}^{\prime}] = w[X_{Q}^{\prime}] s_{\im_l}$ for any $1 \le l \le n$.
Moreover,  we have $\im_{l} = \im$ with $1 \le l \le n$
only if $l=k$.
Thus we have $\gamma^{\calQ}_{\im} = (1-\tau_{\calQ}^{\circ})\varpi_{\im} = s_{\im_{1}} \cdots s_{\im_{k-1}} \alpha_{\im}$.
\end{proof}

\begin{lemma}
\label{lemma: gamma reflection}
Let $\im \in \Delta_{0}$ be a source of $\calQ$.
For any $(\jm, p) \in \Delta_{0} \times \Z$ such that $p - \uxi_{\jm} \in 2\Z$, we have
$$s_{\im}\tau_{\calQ}^{(\uxi_{\jm}-p)/2}(\gamma_{\jm}^{\calQ}) = \tau_{s_{\im}\calQ}^{((s_{\im}\uxi)_{\jm}-p)/2}(\gamma_{\jm}^{s_{\im}\calQ}).$$
\end{lemma}
\begin{proof}
First we consider the case when $\jm \neq \im$.
By Proposition~\ref{proposition: conjugation}~(1) and $s_{\im} \varpi_{\jm} = \varpi_{\jm}$,
we have
$$
s_{\im}\tau_{\calQ}^{(\uxi_{\jm}-p)/2}(\gamma_{\jm}^{\calQ}) =
\tau_{s_{\im}\calQ}^{(\uxi_{\jm}-p)/2} (1-\tau_{s_{\im}\calQ}^{d_{\bar{\jm}}})\varpi_{\jm}
= \tau_{s_{\im}\calQ}^{((s_{\im}\uxi)_{\jm}-p)/2}(\gamma_{\jm}^{s_{\im}\calQ}),
$$
which proves the assertion.
Let us consider the other case when $\jm = \im$.
Since $\im$ is a source of $\calQ$, we have
$\gamma^{\calQ}_{\im} = \alpha_{\im} = (1-s_{\im})\varpi_{\im}$
by Lemma~\ref{lemma: gamma}.
Thus we obtain
\begin{align*}
s_{\im}\tau_{\calQ}^{(\uxi_{\im}-p)/2}(\gamma_{\im}^{\calQ})
&= \tau_{s_{\im}\calQ}^{(\uxi_{\im}-p)/2}(s_{\im}-1)\varpi_{\im} \\
&= \tau_{s_{\im}\calQ}^{(\uxi_{\im} -2d_{\bar{\im}}-p)/2}\tau_{s_{\im}\calQ}^{d_{\bar{\im}}}(s_{\im}-1)\varpi_{\im} \\
&= \tau_{s_{\im}\calQ}^{(\uxi_{\im} -2d_{\bar{\im}}-p)/2}(1-\tau_{s_{\im}\calQ}^{d_{\bar{\im}}})\varpi_{\im} \\
&= \tau_{s_{\im}\calQ}^{((s_{\im}\uxi)_{\im}-p)/2}(\gamma^{s_{\im}\calQ}_{\im})
\end{align*}
as desired.
Here for the third equality we used the fact $\tau_{s_{\im}\calQ}^{d_{\bar{\im}}}s_{\im} \varpi_{\im} = \varpi_{\im}$,
which holds because the element $\tau_{s_{\im}\calQ}^{d_{\bar{\im}}}s_{\im}$ has an expression without the simple reflection $s_{\im}$.
\end{proof}

\begin{lemma}
\label{lemma: gamma relation}
Let $\calQ=(\Delta, \sigma, \uxi)$ be a ${\rm Q}$-datum.
We fix an index $i \in I$.
For any $\im, \im^{\prime} \in i$ we have
$$\tau_{\calQ}^{(\uxi_{\im} - \uxi_{\im^{\prime}})/2}(\gamma_{\im}^{\calQ}) = \gamma^{\calQ}_{\im^{\prime}}.$$
\end{lemma}
\begin{proof}
First we consider the case when $[\calQ]$ contains a reduced word arising from a twisted Coxeter element
$\tau_{\calQ} = s_{\im_{1}} \cdots s_{\im_{n}} \sigma$,
or equivalently the condition~(\ref{eq: condition twisted Coxeter height}) is satisfied.
In this case,
what we have to show is that $\tau_{\calQ}(\gamma^{\calQ}_{\im}) = \gamma^{\calQ}_{\sigma(\im)}$
for all $\im \in I_{\calQ}^{\circ}$ with $d_{\bar{\im}} > 1$.
This can be deduced immediately from Lemma~\ref{lemma: gamma}.
A general case is reduced to this special case by Lemma~\ref{lemma: gamma reflection}
since any $\calQ$ can be obtained from $\calQ^{\circ}$ by a suitable sequence of source reflections
(cf.~(\ref{eq: translation})).
\end{proof}

\begin{corollary}
\label{corollary: opposite end}
Let $\calQ=(\Delta, \sigma, \uxi)$ be a Q-datum.
For each $\im \in \Delta_{0}$, we have
$$\tau_{\calQ}^{(rh^{\vee} + \uxi_{\im}-\uxi_{\im^{*}})/2}(\gamma^{\calQ}_{\im}) = -\gamma^{\calQ}_{\im^{*}}.$$
\end{corollary}
\begin{proof}
When $\sigma = \id$, the assertion is well-known.
When $\sigma \neq \id$,
we have $\im^{*} \in \bar{\im}$ for any $\im \in \Delta_{0}$
(see~Remark~\ref{remark: h^vee parity}).
Thus the assertion follows from
Proposition~\ref{proposition: conjugation}~(4) and
Lemma~\ref{lemma: gamma relation}.
\end{proof}

\begin{proof}[Proof of Theorem~\ref{thm: phi_Q}]
Let $\im \in \Delta_{0}$ be a source of $\calQ$.
We shall \ber prove the assertions \er for $s_{\im}\calQ$ assuming that
they are true for $\calQ$.
By Lemma~\ref{lemma: gamma reflection},
we have $\phi_{s_{\im}\calQ} = \widehat{s}_{\im} \circ \phi_{\calQ}$, where
$\widehat{s}_{\im} \colon \hsfR^+ \to \hsfR^+$
is the bijection defined by
\begin{align}\label{eq: sihat}
\widehat{s}_{\im} (\beta, k) := \begin{cases}
(s_{\im}\beta, k) & \text{if $\beta \in \sfR^{+} \setminus \{ \alpha_{\im} \}$}, \\
(\alpha_{\im} , k+1) & \text{if $\beta = \alpha_{\im}$}
\end{cases}
\end{align}
(see also \cite{KKOP20Br}). 
Therefore $\phi_{s_{\im}\calQ}$ is a bijection.
From the assumption $\phi_{\calQ}^{-1}(\sfR^{+} \times \{ 0 \}) = (\Gamma_{\calQ})_{0}$
and Corollary~\ref{corollary: opposite end},
we obtain $\phi_{\calQ}(\im^{*}, \uxi_{\im} - rh^{\vee}) = (\alpha_{\im}, -1)$.
Therefore we have
$$
\phi_{s_{\im}\calQ}^{-1}(\sfR^{+} \times \{ 0 \})
= \phi_{\calQ}^{-1}(\widehat{s}_{\im}^{-1} (\sfR^{+} \times \{ 0 \}))
= ((\Gamma_{\calQ})_{0} \setminus \{ (\im, \uxi_{\im}) \})\cup\{(\im^{*}, \uxi_{\im}-rh^{\vee})\}
= (\Gamma_{s_{\im}\calQ})_{0}.
$$
Furthermore, thanks to Proposition~\ref{proposition: reflection algorithm}, we obtain
$\phi_{s_{\im}\calQ}(\jm, p)=(\phi_{s_{\im}\calQ, 0}(\jm, p), 0)$ for any $(\jm,p)\in (\Gamma)_{s_{\im}\calQ}$.

Note that any Q-datum $\calQ$ can be obtained from $\calQ^{\circ}$
by a suitable sequence of source reflections (cf.~(\ref{eq: translation})).
Since the assertions for $\calQ^{\circ}$ follow immediately from
Theorem~\ref{theorem: OS19b main} and the definition of $\Upsilon_{[\calQ^{\circ}]}$,
we have obtained the proof for general $\calQ$.
\end{proof}

As a consequence of Theorem~\ref{thm: phi_Q} and Corollary~\ref{corollary: opposite end},
we observe the following.

\begin{corollary} \label{corollary: shift hDs}
Let $\calQ=(\Delta, \sigma, \uxi)$ be a ${\rm Q}$-datum
and $(\im, p) \in \hDs_{0}$.
If $\phi_{\calQ}(\im,p)= (\beta, k)$, we have
$$
\phi_{\calQ}(\im^{*}, p \mp rh^{\vee}) = (\beta, k\pm 1).
$$
\end{corollary}

The next statement is a generalization of the additive property~(\ref{eq: classical additive property}).

\begin{theorem}[$\g$-additive property]
\label{thm: twisted additivity}
Let $\calQ = (\Delta, \sigma, \uxi)$ be a ${\rm Q}$-datum for $\g$ and
$I^{\circ} \subset \Delta_{0}$ be
an arbitrary section of the natural quotient map $\bar{ \ } \colon \Delta_{0} \to I$.
For any
$\im \in \Delta_{0}$ and $l \in \Z$, we have
\begin{align}
\label{eq: twisted additivity}
\tau_{\calQ}^{l}(\gamma^{\calQ}_{\im}) + \tau_{\calQ}^{l + d_{\bar{\im}}}(\gamma^{\calQ}_{\im})
&= \sum_{\jm \in I^{\circ}, \ \bar{\jm} \sim \bar{\im}} \sum_{t=0}^{-c_{\bar{\jm} \bar{\im}}-1}
\tau_{\calQ}^{l+t+(\uxi_{\jm} -\uxi_{\im} + \min(d_{\bar{\im}}, d_{\bar{\jm}}))/2} (\gamma^{\calQ}_{\jm}) \\
&= \sum_{j \in I, \ j \sim \bar{\im}} \sum_{t=0}^{-c_{j \bar{\im}}-1}  \frac{1}{d_{j}} \sum_{\jm \in j}
\tau_{\calQ}^{l+t+(\uxi_{\jm} -\uxi_{\im} + \min(d_{\bar{\im}}, d_{j}))/2} (\gamma^{\calQ}_{\jm}). \nonumber
\end{align}
\end{theorem}
\begin{proof}
%
Note that each summand $\tau_{\calQ}^{l+t+(\uxi_{\jm} -\uxi_{\im} + \min(d_{\bar{\im}}, d_{j}))/2} (\gamma^{\calQ}_{\jm})$
does not depend on the choice of the representative $\jm \in j$ thanks to Lemma~\ref{lemma: gamma relation}.
This verifies the second equality and it suffices to prove the first equality~(\ref{eq: twisted additivity})
for a special choice of $I^{\circ}$.
Moreover,
Lemma~\ref{lemma: gamma reflection} reduces the situation to the case when
$\calQ$ satisfies the condition~(\ref{eq: condition twisted Coxeter height}) (i.e.~$\tau_{\calQ} = \tau_{\calQ}^{\circ}$)
and the vertex $\im$ is a source of $\calQ$.

The LHS of (\ref{eq: twisted additivity}) is computed as
\begin{align}
\label{eq: twisted additivity computation}
\tau_{\calQ}^{l}(\gamma^{\calQ}_{\im}) + \tau_{\calQ}^{l + d_{\bar{\im}}}(\gamma^{\calQ}_{\im})
&=
\tau_{\calQ}^{l} (1- \tau_{\calQ}^{d_{\bar{\im}}}) \varpi_{\im}
+ \tau_{\calQ}^{l + d_{\bar{\im}}} (1-\tau_{\calQ}^{d_{\bar{\im}}}) \varpi_{\im} \\
&=
\tau_{\calQ}^{l} (1-\tau_{\calQ}^{d_{\bar{\im}}}) (1+\tau_{\calQ}^{d_{\bar{\im}}}) \varpi_{\im}  \nonumber \\
&= \sum_{\jm \in \Delta_{0}, \ \jm \sim \im} \tau_{\calQ}^{l} (1-\tau_{\calQ}^{d_{\bar{\im}}})\varpi_{\jm}. \nonumber
\end{align}
Here for the last equality we used the fact
$$
(1+\tau_{\calQ}^{d_{\bar{\im}}}) \varpi_{\im} = 2 \varpi_{\im} - \alpha_{\im} = \sum_{\jm \in \Delta_{0}, \ \jm \sim \im} \varpi_{\jm},
$$
which holds because $\im$ is assumed to be a source of $\calQ$.

Now we assume that $d_{\bar{\im}} = r$.
Noting that $1-\tau_{\calQ}^{d_{\bar{\im}}} = (1-\tau_{\calQ}^{d_{\bar{\jm}}}) \sum_{t=0}^{-c_{\bar{\jm} \bar{\im}}-1} \tau_{\calQ}^{t}$
in this case, we have
$(1-\tau_{\calQ}^{d_{\bar{\im}}}) \varpi_{\jm} = \sum_{t=0}^{-c_{\bar{\jm} \bar{\im}}-1} \tau_{\calQ}^{t} (\gamma^{\calQ}_{\jm})$
for any $\jm \sim \im$.
On the other hand, we see $\uxi_{\im} - \uxi_{\jm} = \min(d_{\bar{\im}}, d_{\bar{\jm}})$ for
any $\jm \in \Delta_{0}$ with $\jm \sim \im$ since $\im$ is a source of $\calQ$ and $d_{\bar{\im}} =r$.
This completes a proof of (\ref{eq: twisted additivity}) in this case.

Finally we consider the case $d_{\bar{\im}} =1$. Note that, for any $j \in I$, we have
$$(1-\tau_{\calQ})\sum_{\jm \in j}\varpi_{\jm} =
\gamma^{\calQ}_{j^{\circ}}
$$
by Lemma~\ref{lemma: gamma} and the assumption $\tau_{\calQ}=\tau_{\calQ}^{\circ}$.
Applying this to the equation~(\ref{eq: twisted additivity computation}), we have
$$
\tau_{\calQ}^{l}(\gamma^{\calQ}_{\im}) + \tau_{\calQ}^{l + d_{\bar{\im}}}(\gamma^{\calQ}_{\im})
= \sum_{\jm \in I_{\calQ}^{\circ}, \ \jm \sim \im} \tau_{\calQ}^{l}(\gamma^{\calQ}_{\jm}).
$$
On the other hand, we have $c_{\bar{\jm} \bar{\im}} = -1$ and
$\uxi_{\im} - \uxi_{\jm} = \min(d_{\bar{\im}}, d_{\bar{\jm}})$ for
any $\jm \in I_{\calQ}^{\circ}$ with $\jm \sim \im$ under our assumption.
This completes a proof of (\ref{eq: twisted additivity}).
\end{proof}
\subsection{Folding twisted AR quivers}
\label{subsection: folding}

Let $\epsilon \colon I \to \{ 0,1\}$
be the function given by $\epsilon_{i} \equiv \epsilon_{\im} \pmod{2}$ for any $\im \in i$.
Note that this is well-defined by Lemma~\ref{lemma: parity}.
We define the infinite set $\hI$ by
\begin{align}\label{eq: hI}
\hI := \{ (i, p) \in I \times \Z \mid p - \epsilon_{i} \in 2\Z \}.
\end{align}
Restricting the map $\Delta_{0} \times \Z \to I \times \Z$ given by
$(\im, p) \mapsto (\bar{\im}, p)$,
we obtain the \emph{folding map}
$f \colon \hDs_{0} \to \hI.$
By Lemma~\ref{lemma: parity},
the map $f$ is a bijection.

Let $\calQ=(\Delta, \sigma, \uxi)$ be a Q-datum.
By composing, we obtain a bijection
$$\bphi_{\calQ} := \phi_{\calQ}  \circ (f^{-1}) \colon \hI \to \hsfR^+,
$$ which satisfies
\begin{equation}
\label{eq: bphi}
\tau_{\calQ}^{(\uxi_{\im} -p)/2}(\gamma^{\calQ}_{\im}) = (-1)^{k} \alpha \qquad \text{if $\bphi_{\calQ}(i,p) = (\alpha, k)$}
\end{equation}
for any $(i,p) \in \hI$ and $\im \in i$.
Note that the LHS of~(\ref{eq: bphi}) does not depend on the choice of $\im \in i$
thanks to Lemma~\ref{lemma: gamma relation}.

Following \cite{OS19b}, we define the \emph{folded AR quiver of $\calQ$} to be
the quiver $\bGamma_{\calQ}$ whose vertex set is
$$\hI_{\calQ} := f((\Gamma_{\calQ})_{0}) = \{ (\bar{\im}, p) \in \hI \mid (\im, p) \in (\Gamma_{\calQ})_{0}\}$$
and such that the restriction $f_{\calQ} := f|_{(\Gamma_{\calQ})_{0}} \colon (\Gamma_{\calQ})_{0} \to \hI_{\calQ}$
induces an isomorphism $\Gamma_{\calQ} \simeq \bGamma_{\calQ}$ of quivers.
We have the bijection $\bphi_{\calQ, 0} := \phi_{\calQ, 0} \circ (f_{\calQ}^{-1}) \colon \hI_{\calQ} \to \sfR^{+}$,
where $\phi_{\calQ, 0} \colon (\Gamma_{\calQ})_{0} \to \sfR^{+}$ is the bijection
introduced in Subsection~\ref{subsection: twisted AR quiver}.
Thanks to Theorem~\ref{thm: phi_Q} and (\ref{eq: bphi}), we have $\bphi_{\calQ}(i,p)
=(\bphi_{\calQ,0}(i,p), 0)$ and
$
\bphi_{\calQ, 0} (i, p) = \tau_{\calQ}^{(\uxi_{\im} -p)/2}(\gamma^{\calQ}_{\im})
$
for any $(i,p) \in \hI_{\calQ}$ and $\im \in i$.

\begin{example}
\label{example: bGammaB3}
The folded AR quivers associated with the Q-data $\calQ^{(1)}$ and $\calQ^{(2)}$ for $\g$ of type $\mathsf{B}_{3}$
given in Example~\ref{example: QB3} above are depicted as follows. Compare with
the twisted AR quivers $\Gamma_{\calQ^{(1)}}$ and $\Gamma_{\calQ^{(2)}}$ depicted in Example~\ref{example: GammaB3}.
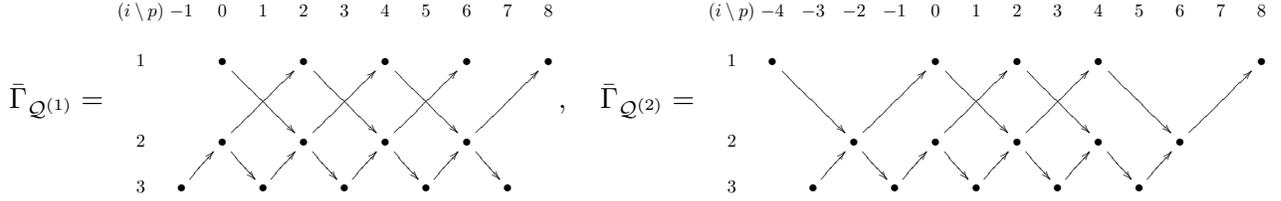
$$\ber \bGamma_{\calQ^{(1)}} =
\raisebox{3.3em}{\scalebox{0.63}{\xymatrix @!C=0.1ex@R=3ex{
(i\setminus p) &-1&   0 &  1&   2 &  3&   4& 5&   6 &  7 & 8 \\
1 &&\bullet \ar@{->}[ddrr]&&  \bullet \ar@{->}[ddrr] &&\bullet \ar@{->}[ddrr]&& \bullet && \bullet  \\ \\
2&& \bullet \ar@{->}[dr]\ar@{->}[uurr] &&\bullet \ar@{->}[dr]\ar@{->}[uurr]
&& \bullet \ar@{->}[dr] \ar@{->}[uurr] && \bullet \ar@{->}[dr]\ar@{->}[uurr] \\
3& \bullet \ar@{->}[ur] && \bullet\ar@{->}[ur]  && \bullet\ar@{->}[ur] && \bullet \ar@{->}[ur]&& \bullet }}}, \er
\quad
\bGamma_{\calQ^{(2)}} =
\raisebox{3.3em}{\scalebox{0.63}{\xymatrix @!C=0.1ex@R=3ex{
(i\setminus p) & -4 &-3& -2 &-1& 0 &1&2 & 3& 4& 5&  6 & 7 & 8 \\
1&\bullet \ar@{->}[ddrr]&& && \bullet \ar@{->}[ddrr]&&  \bullet \ar@{->}[ddrr]  && \bullet \ar@{->}[ddrr]&&&& \bullet \\ \\
2&&&\bullet \ar@{->}[uurr]  \ar@{->}[dr]&& \bullet \ar@{->}[dr]\ar@{->}[uurr]
&& \bullet \ar@{->}[dr]\ar@{->}[uurr] &&  \bullet \ar@{->}[dr] && \bullet\ar@{->}[uurr]  \\
3&&\bullet \ar@{->}[ur] && \bullet \ar@{->}[ur] && \bullet \ar@{->}[ur] && \bullet \ar@{->}[ur] && \bullet \ar@{->}[ur] &&
}}}
$$
\end{example}

\begin{remark}
\label{remark: 2seg}
Unlike the (twisted) AR quiver $\Gamma_{\calQ}$, the folded
AR quiver $\bGamma_{\calQ}$ may not satisfy the $2$-segment property.
See~$\bGamma_{\calQ^{(2)}}$ in Example~\ref{example: bGammaB3} above.
Proposition~\ref{proposition: twisted Coxeter height} tells us that
the folded AR quiver $\bGamma_{\calQ}$ satisfies the $2$-segment property
if and only if $[\calQ] = [\bfi(\tau)]$ for some twisted Coxeter element $\tau$,
or equivalently the condition~(\ref{eq: condition twisted Coxeter height}) is satisfied.
When $\sigma \neq \id$, we can see the $2$-segment property as
$$
\hI_\calQ = \{ (i, \uxi_{i^{\circ}}-2k) \in I\times \Z \mid k \in \Z, 0 \le k < rh^{\vee}/2 \}
$$
under the condition~(\ref{eq: condition twisted Coxeter height}).
\end{remark}

\begin{proposition} \label{proposition: shift hDs}
Let $\calQ=(\Delta, \sigma, \uxi)$ be a Q-datum
and $(i, p) \in \hI$.
If $\bphi_{\calQ}(i,p)= (\alpha, k)$, we have
$$
\bphi_{\calQ}(i^{*}, p \mp rh^{\vee}) = (\alpha, k\pm 1).
$$
\end{proposition}
\begin{proof}
This is just a re-expression of Corollary~\ref{corollary: shift hDs}.
\end{proof}

\section{Inverse of quantum Cartan matrices}
\label{section: Cartan}

In this section, we show that the inverse of the quantum Cartan matrix of $\g$ can be computed
by using the generalized twisted Coxeter element $\tau_{\calQ}$ associated with a Q-datum $\calQ$ for $\g$.
Let us keep the notation in Section~\ref{section: Q-data}. \beb Note that 
the results in this section can be understood as generalizations of the results for symmetric affine types in  \cite[\S 2.5]{HL15} to all untwisted affine types.  \eb

\subsection{Quantum Cartan matrix} \label{ssec: Cartan}

Let $\g$ be a finite-dimensional complex simple Lie algebra
and $C=(c_{ij})_{i, j \in I}$ denote its Cartan matrix as in Subsection~\ref{subsection: notation2}.

\begin{definition}
Let $z$ be an indeterminate.
The \emph{quantum Cartan matrix}
of $\g$ is the $\Z[z^{\pm1}]$-valued $(I\times I)$-matrix $C(z) = (C_{ij}(z))_{i,j \in I}$ defined by
\begin{align}\label{eq: q Cartan}
C_{ij}(z) = \begin{cases}
z^{d_{i}}+z^{-d_{i}} & \text{if $i=j$,} \\
[c_{ij}]_{z} & \text{if $i \ne j$.}
\end{cases}
\end{align}
where $[k]_{z}$ denotes the quantum integer
$
[k]_{z} := \frac{z^{k}-z^{-k}}{z-z^{-1}} \in \Z[z^{\pm 1}]
$
for each $k \in \Z$.
\ber Note that, for $i \sim j$, we have
$[c_{ij}]_z = -[r]_z$ (resp.~$-1$) if $d_{j} > d_{i}$ (resp.~$d_{j} \le d_i$). \er 
\end{definition}

We set $D(z) := \diag\left([d_{i}]_{z} \mid i \in I\right)$.
The following property is easy to see from the definitions.

\begin{lemma}
\label{lemma: qDCsymm}
We have $D(z) C(z) = ([d_{i}c_{ij}]_{z})_{i,j \in I}$.
In particular, it is symmetric (cf.~(\ref{eq: DC symmetric})).
\end{lemma}

Note that we have $C(z)|_{z=1} = C \in GL_{I}(\Q)$.
We regard $C(z)$ as an element of $GL_{I}\left( \Q(z) \right)$ and denote its inverse
by $\tC(z) = ( \tC_{ij}(z) )_{i,j \in I}$.
Let
$$\tC_{ij}(z)  = \sum_{u \in \Z}  \tc_{ij}(u)z^u$$
denote the formal Laurent expansion of the $(i,j)$-entry $\tC_{ij}(z)$ at $z = 0$.

\begin{lemma} \label{lemma: small vanish}
 For any $i, j \in I$ and $u \in \Z$, we have
$\tc_{ ij } (u) \in \Z$.
Moreover, we have
\begin{enumerate}
\item \label{tc1_small vanish}
$\tc_{ ij } (u) = 0$ if $u < d_{i}$,
\item \label{tc1_unip}
$\tc_{ij}(d_{i}) = \delta_{i,j}$.
\end{enumerate}
\end{lemma}
\begin{proof}
Set $E(z) = ( E_{ij}(z))_{i,j \in I} := C(z)\cdot z^{D}$,
where $z^{D} \seteq \diag(z^{d_{i}} \mid i \in I)$.
Since $\tC(z) = z^{D} \cdot E(z)^{-1}$, it suffices to show that
$E_{ij}(z) \in \delta_{i, j} + z\Z[z]$ for each $i,j \in I$.
If $i = j$, we have $E_{ii}(z) = 1 + z^{2 d_{i}}$.
If $i \neq j$, we have
$-E_{ij}(z) = z^{d_{j}}\sum_{u=0}^{-c_{i j} - 1} z^{c_{i j} +1 + 2u}.$
Therefore it is enough to show $d_{j} + c_{i j} + 1 > 0$, or equivalently
$d_{j} \ge - c_{i j} = \lceil d_{j}/d_{i} \rceil$ for $i \sim j$.
The last condition is now obvious.
\end{proof}

\begin{lemma}
\label{lemma: tc1}
The integers $\{ \tc_{ij}(u) \mid i,j \in I, u \in \Z \}$ enjoy the following properties$\colon$
\begin{enumerate}
\item \label{tc1_auto} We have $\tc_{ij}(u) = \tc_{a(i) a(j)}(u)$
for any automorphism $a$ of the Dynkin diagram of $\g$.
\item \label{tc1_swap} For any $i,j \in I$ and $u \in \Z$, we have
$$
\tc_{ij}(u) = \begin{cases}
\tc_{ji}(u) & \text{if $d_{i}=d_{j}$}, \\
\tc_{ji}(u+1) + \tc_{ji}(u-1) & \text{if $(d_{i}, d_{j}) = (1,2)$}, \\
\tc_{ji}(u+2) + \tc_{ji}(u) + \tc_{ji}(u-2) & \text{if $(d_{i}, d_{j}) = (1,3)$}.
\end{cases}
$$
\item \label{tc1_symm} For any $i,j \in I$ and $u \in \Z$, we have
$$
\tc_{ij}(u+d_{i}) - \tc_{ij}(u-d_{i}) = \tc_{ji}(u+d_{j}) - \tc_{ji}(u-d_{j}).
$$
\end{enumerate}
\end{lemma}
\begin{proof}
The assertion (\ref{tc1_auto}) is immediate from the definition.
Let us prove (\ref{tc1_swap}) and (\ref{tc1_symm}).
Since the product $D(z) C(z)$ is symmetric by Lemma~\ref{lemma: qDCsymm},
$D(z) C(z)^{-1} = D(z) \cdot (D(z) C(z))^{-1} \cdot D(z)$ is symmetric as well.
In other words, for each $i, j \in I$, we have
\begin{equation}
\label{eq: qDCsymm}
\frac{z^{d_{i}} - z^{-d_{i}}}{z-z^{-1}} \sum_{u \in \Z}\tc_{ij}(u)z^{u} = \frac{z^{d_{j}}-z^{-d_{j}}}{z-z^{-1}} \sum_{u\in \Z} \tc_{ji}(u)z^{u}.
\end{equation}
Comparing the coefficients of $z^{u}$ in both sides of (\ref{eq: qDCsymm}),
we obtain the assertion (\ref{tc1_swap}) for the cases $(d_{i}, d_{j}) = (1,1), (1,2), (1,3)$.
When $(d_{i}, d_{j})=(2,2)$, we obtain
$$\tc_{ij}(u+1) + \tc_{ij}(u-1) = \tc_{ji}(u+1) + \tc_{ji}(u-1).$$
Then the assertion $\tc_{ij}(u) = \tc_{ji}(u)$ can be proved by induction on $u$
since we know $\tc_{ij}(u)=\tc_{ji}(u)=0$ for $u \le 0$ thanks to Lemma~\ref{lemma: small vanish}.
The case $(d_{i}, d_{j})=(3,3)$ is proved in the same way.
Similarly, after multiplying $z-z^{-1}$ to the both sides of (\ref{eq: qDCsymm}), we obtain the assertion (\ref{tc1_symm})
by comparing the coefficients of $z^{u}$.
\end{proof}

\subsection{\ber A combinatorial formula\er}
Recall the symmetric bilinear form $(\ , \ ) \colon \sfP \times \sfP \to \Q$ determined by
$(\varpi_{\im}, \alpha_{\jm}) = \delta_{\im, \jm}$ for $\im, \jm \in \Delta_{0}$.
This is invariant under the action of $\sfW \rtimes \langle \sigma \rangle$.

\begin{definition}
Let $\calQ=(\Delta, \sigma, \uxi)$ be a Q-datum for $\g$.
For each $\im, \jm \in \Delta_{0}$,
we define a function $\eta^{\calQ}_{\im\jm} \colon \Z \to \Z$ by
$$
\eta^{\calQ}_{\im \jm}(u) := \begin{cases}
(\varpi_{\im}, \tau_{\calQ}^{(u+\uxi_{\jm} - \uxi_{\im} - d_{\bar{\im}})/2}(\gamma^{\calQ}_{\jm}))
& \text{if $u+\uxi_{\jm} - \uxi_{\im} - d_{\bar{\im}} \in 2\Z$,} \\
0 & \text{otherwise}.
\end{cases}
$$
\end{definition}

\begin{lemma} \label{lemma: eta welldefined}
Let $\calQ^{\prime}$ be another Q-datum for $\g$ and
$\im^{\prime} \in \bar{\im}, \jm^{\prime} \in \bar{\jm}$.
Then we have $\eta^{\calQ}_{\im \jm} = \eta^{\calQ^{\prime}}_{\im^{\prime} \jm^{\prime}}.$
\end{lemma}
\begin{proof}
We only have to consider the case when $u+\uxi_{\jm} - \uxi_{\im} - d_{\bar{\im}} \in 2\Z$.
The equality $\eta^{\calQ}_{\im \jm} = \eta^{\calQ}_{\im \jm^{\prime}}$ for $\jm^{\prime} \in \bar{\jm}$
follows from Lemma~\ref{lemma: gamma relation}.
To prove the equality $\eta^{\calQ}_{\im \jm} = \eta^{\calQ^{\prime}}_{\im \jm}$ for any $\calQ^{\prime}$,
it suffices to verify the equality $\eta^{\calQ}_{\im \jm} = \eta^{s_{k}\calQ}_{\im \jm}$,
where $k \in \Delta_{0}$ is a source of $\calQ$.
When $k \neq \im$, we have $s_{k}\varpi_{\im}=\varpi_{\im}$ and $(s_{k}\uxi)_{\im} = \uxi_{\im}$.
Therefore using Lemma~\ref{lemma: gamma reflection}, we obtain
$$
\eta^{\calQ}_{\im \jm}(u) = (\varpi_{\im}, s_{k}\tau_{\calQ}^{(u+\uxi_{\jm} - \uxi_{\im} - d_{\bar{\im}})/2}(\gamma^{\calQ}_{\jm}))
= (\varpi_{\im}, \tau_{s_{k}\calQ}^{(u+(s_{k}\uxi)_{\jm} - (s_{k}\uxi)_{\im} - d_{\bar{\im}})/2}(\gamma^{s_{k}\calQ}_{\jm}))
= \eta^{s_{k}\calQ}_{\im \jm}(u).
$$
When $k=\im$, we have $s_{\im}\tau_{\calQ}^{d_{\bar{\im}}} \varpi_{\im}=\varpi_{\im}$
since the element $s_{\im}\tau_{\calQ}^{d_{\bar{\im}}}$ has an expression without the simple reflection $s_{\im}$.
Using Lemma~\ref{lemma: gamma reflection} again, we obtain
$$
\eta^{\calQ}_{\im \jm}(u) = (s_{\im}\tau_{\calQ}^{d_{\bar{\im}}}\varpi_{\im},
s_{\im}\tau_{\calQ}^{(u+\uxi_{\jm} - \uxi_{\im} + d_{\bar{\im}})/2}(\gamma^{\calQ}_{\jm}))
=(\varpi_{\im}, \tau_{s_{\im}\calQ}^{(u+(s_{\im}\uxi)_{\jm} - (s_{\im}\uxi)_{\im} - d_{\bar{\im}})/2}(\gamma^{s_{\im}\calQ}_{\jm}))
= \eta^{s_{\im}\calQ}_{\im\jm}(u).
$$
Finally, let us verify the equality $\eta^{\calQ}_{\im\jm} = \eta^{\calQ}_{\im^{\prime}\jm}$ for $\im^{\prime} \in \bar{\im}$.
By the independence of the choice of $\calQ$,
we may assume that $\tau_{\calQ}$ is a twisted Coxeter element, or equivalently $\calQ$ satisfies
the condition~(\ref{eq: condition twisted Coxeter height}).
Note that we have $\tau_{\calQ}^{l}\varpi_{\im} = \varpi_{\sigma^{l}(\im)}$ for any $\im \in I_{\calQ}^{\circ}$
and $0 \le l < d_{\bar{\im}}$. Therefore we obtain
$$
\eta^{\calQ}_{\im \jm}(u)=
(\tau_{\calQ}^{l}\varpi_{\im}, \tau_{\calQ}^{(2l+u+\uxi_{\jm} - \uxi_{\im} + d_{\bar{\im}})/2}(\gamma^{\calQ}_{\jm}))
= (\varpi_{\sigma^{l}(\im)}, \tau_{\calQ}^{(u+\uxi_{\jm} - \uxi_{\sigma^{l}(\im)} + d_{\bar{\im}})/2}(\gamma^{\calQ}_{\jm}))
= \eta^{\calQ}_{\sigma^{l}(\im)\jm}(u)
$$
as desired.
\end{proof}

By  Lemma~\ref{lemma: eta welldefined}, the following notation is well-defined.
\begin{definition}
For each $i,j \in I$, we define $$\eta_{ij} \seteq \eta^{\calQ}_{\im \jm},$$
where $\calQ$ is a Q-datum for $\g$ and $\im \in i, \jm \in j$.
\end{definition}

The following statement is the main theorem of this subsection,
which gives a generalization of~\cite[Proposition 2.1]{HL15}.

\begin{theorem}
\label{theorem: combinatorial formula}
For each $i,j \in I$ and $u \in \Z_{\ge 0}$, we have
$\tc_{ij}(u) = \eta_{ij}(u)$.
In other words, we have
$$
\tc_{ij}(u) = \begin{cases}
(\varpi_{\im}, \tau_{\calQ}^{(u+\uxi_{\jm} - \uxi_{\im} - d_{i})/2}(\gamma^{\calQ}_{\jm}))
& \text{if $u+\uxi_{\jm} - \uxi_{\im} - d_{i} \in 2\Z$,} \\
0 & \text{otherwise}
\end{cases}
$$
for any Q-datum $\calQ=(\Delta, \sigma, \uxi)$ for $\g$ and $u \in \Z_{\ge 0}, \im \in i, \jm \in j$.
\end{theorem}

For a proof of Theorem~\ref{theorem: combinatorial formula},
we need a lemma.

\begin{lemma}
\label{lemma: eta property}
The functions $\{ \eta_{ij} \colon \Z \to \Z \mid i, j \in I \}$ enjoy the following properties:
\begin{enumerate}
\item \label{eta_+rh} $\eta_{ij}(u + rh^{\vee}) = -\eta_{ij^{*}}(u)$ for any $u \in \Z$.
\item \label{eta_additivity}
For any $u \in \Z$, we have
$$
\eta_{ij}(u-d_{j}) + \eta_{ij}(u+d_{j}) = \sum_{k \sim j}\sum_{l=0}^{-c_{kj}-1}\eta_{ik}(u+c_{kj}+1+2l).
$$
\item \label{eta_ARduality} For any $u \in \Z$, we have
$$
-\eta_{ij}(-u) = \begin{cases}
\eta_{ji}(u) & \text{if $d_{i}=d_{j}$}, \\
\eta_{ji}(u+1) + \eta_{ji}(u-1) & \text{if $(d_{i}, d_{j})=(1,2)$}, \\
\eta_{ji}(u+2) + \eta_{ji}(u) + \eta_{ji}(u-2) & \text{if $(d_{i}, d_{j})=(1,3)$}.
\end{cases}
$$
\item \label{eta_small vanish} $\eta_{ij}(u)=0$ if $|u| \le d_{i}-\delta_{ij}$.
\item \label{eta_unip} $\eta_{ii}(\pm d_{i})= \pm1$.
\end{enumerate}
\end{lemma}
\begin{proof}
When $\sigma \neq \id$, the property~(\ref{eta_+rh}) is a consequence of Proposition~\ref{proposition: conjugation}~(\ref{tauconj_rh}).
When $\sigma = \id$, it is proved in~\cite[\ber Lemma 3.7\er]{Fujita19}.
The property~(\ref{eta_additivity}) follows from the $\g$-additive property
(Theorem~\ref{thm: twisted additivity}) if we note that
$d_{j}+c_{kj}+1 = d_{j} - \lceil d_{j}/d_{k} \rceil +1= \min(d_{j}, d_{k})$ for any $j,k \in I$.

Let us prove the property~(\ref{eta_ARduality}).
We choose a Q-datum $\calQ=(\Delta, \sigma, \uxi)$ for $\g$ and $\im \in i, \jm \in j$.
Let $u \in \Z$ satisfy $u +\uxi_{\jm} -\uxi_{\im}-d_{i} \in 2\Z$.
We compute
\begin{align*}
-\eta_{ij}(-u) & = (\varpi_{\im}, \tau_{\calQ}^{(-u + \uxi_{\jm} - \uxi_{\im} - d_{i})/2} (\tau_{\calQ}^{d_{j}}-1) \varpi_{\jm}) \\
&= (\tau_{\calQ}^{(u + \uxi_{\im}-\uxi_{\jm} + d_{i})/2}(\tau_{\calQ}^{-d_{j}}-1)\varpi_{\im}, \varpi_{\jm}) \\
&= (\varpi_{\jm}, \tau_{\calQ}^{((u+d_{i} -d_{j}) + \uxi_{\im}-\uxi_{\jm} - d_{j})/2}(1-\tau_{\calQ}^{d_{j}})\varpi_{\im}).
\end{align*}
Under the assumption $d_{i} \le d_{j}$, we have
$1-\tau_{\calQ}^{d_{j}} = (\sum_{l=0}^{d_{j}-d_{i}} \tau_{\calQ}^{l})(1-\tau_{\calQ}^{d_{i}})$.
Combining with the above computation, we obtain (\ref{eta_ARduality}).

Finally we shall prove the property~(\ref{eta_small vanish}) and  (\ref{eta_unip}).
Fix $i,j \in I$ and choose a Q-datum $\calQ=(\Delta, \sigma, \uxi)$ for $\g$
satisfying the condition~(\ref{eq: condition twisted Coxeter height})
and such that (a) $j^{\circ}$ is a source of $\calQ$,
(b) $\uxi_{ij} := \uxi_{j^{\circ}} - \uxi_{i^{\circ}} \in \{ 0, 1\}$.
The condition (a) implies $\gamma^{\calQ}_{j^{\circ}}=\alpha_{j^{\circ}}$ by Lemma~\ref{lemma: gamma}.
For any $u \in \Z$ such that $u + \uxi_{ij} - d_{i} \in 2\Z$,
we have
$
\eta_{ij}(u) = (\tau_{\calQ}^{(d_{i} - \uxi_{ij} - u)/2}(\varpi_{i^{\circ}}), \alpha_{j^{\circ}}).
$
Combining with the fact that $\tau_{\calQ}^{l} (\varpi_{i^{\circ}}) = \varpi_{\sigma^{l}(i^{\circ})}$ for $0 \le l < d_{i}$,
we obtain the desired equalities~(\ref{eta_small vanish}) and  (\ref{eta_unip})
under the assumption $0 \le u \le d_{i}$.
The other case when $-d_{i} \le u \le 0$ follows from this case and the property~(\ref{eta_ARduality}).
\end{proof}

\begin{proof}[Proof of Theorem~\ref{theorem: combinatorial formula}]
Setting $H_{ij}(z) := \sum_{u \ge 0} \eta_{ij}(u)z^{u} \in \Z\lf z\rf$ for each $i, j \in I$,
we have to show that
\begin{equation}
\label{eq: HC to prove}
\sum_{k \in I} H_{ik}(z) C_{kj}(z) = \delta_{i,j}.
\end{equation}
We denote by $x_{ij}(u)$ the coefficient of $z^{u}$ in the LHS of (\ref{eq: HC to prove})
for each $u \in \Z$.
Then the equality~(\ref{eq: HC to prove}) is equivalent to
$x_{ij}(u) = \delta_{i,j} \delta_{u,0}$.
By Lemma~\ref{lemma: eta property}~(\ref{eta_small vanish}), we can write
$H_{ij}(z) = \sum_{u > -d_{i}} \eta_{ij}(u)z^{u}$ for any $i,j \in I$.
Therefore we have
$$
x_{ij}(u+d_{j}) = \eta_{ij}(u) + \eta_{ij}(u+2d_{j}) - \sum_{k \sim j}\sum_{l=0}^{-c_{kj}-1}\eta_{ik}(u+d_{j}+c_{kj}+1+2l) = 0
$$
for any $u > -d_{i}$ thanks to Lemma~\ref{lemma: eta property}~(\ref{eta_additivity}).
On the other hand, Lemma~\ref{lemma: eta property}~(\ref{eta_small vanish}) also tells us that
$$
H_{ik}(z)C_{kj}(z) \in \begin{cases}
z^{d_{i}+\delta(i \neq k)+c_{kj}+1} \Z \lf z \rf & \text{if $k\neq j$},\\
z^{d_{i} +\delta(i \neq j) - d_{j}}\Z \lf z \rf & \text{if $k=j$},
\end{cases}
$$
which implies $x_{ij}(u)=0$ if $u < d_{i}-d_{j} + \delta(i \neq j)$.
Therefore it is enough to show
$x_{ij}(u) = \delta_{i,j}\delta_{u,0}$
only for $d_{i} - d_{j} + \delta(i \neq j) \le u \le d_{j} - d_{i}$.
When $d_{i} \ge d_{j}$, we only have to \ber consider \er the case $i=j$ and $u=0$.
In this case, we have
$$
x_{ii}(0) = \eta_{ii}(d_{i}) - \sum_{k \sim i}\sum_{l=0}^{-c_{ki}-1}\eta_{ik}(u+c_{ki}+1+2l) = -\eta_{ii}(-d_{i}) = 1
$$
as desired thanks to Lemma~\ref{lemma: eta property}~(\ref{eta_additivity}) and (\ref{eta_unip}).

In the remaining case $(d_{i}, d_{j}) = (1,r)$ with $r > 1$,
we have to verify $x_{ij}(u)=0$ for $2-r \le u < r-1$.
First we assume that $r = 2$. We observe
$x_{ij}(0) = \eta_{ij}(2) - \eta_{ii}(1) \times \delta(i \sim j)$
using Lemma~\ref{lemma: eta property}~(\ref{eta_small vanish}).
Let us choose a Q-datum $\calQ=(\Delta, \sigma, \uxi)$ satisfying (\ref{eq: condition twisted Coxeter height})
and such that (a) $j^{\circ}$ is a source of $\calQ$, (b) $\uxi_{j^{\circ}} - \uxi_{i^{\circ}} \in \{0,1\}$
as before.
Note that if $\uxi_{j^{\circ}} - \uxi_{i^{\circ}} =0$, we have $i \not \sim j$ and $\eta_{ij}(2)=0$
by the parity reasons.
Thus we may assume $\uxi_{j^{\circ}} - \uxi_{i^{\circ}} =1$.
In this case, we can further require $\uxi_{i^{\circ}} = \min \{\uxi_{k^{\circ}} \mid k \in I \}$.
Then the corresponding twisted Coxeter element $\tau_{\calQ}$ can be written as
$\tau_{\calQ} = w s_{i^{\circ}} \sigma$, where $w$ is a product of $s_{k^{\circ}}$ with $k \in I \setminus \{i\}$.
Then we have
$$
\eta_{ij}(2) = (\varpi_{i^{\circ}}, ws_{i^{\circ}} \sigma(\alpha_{j^{\circ}}))
= (\varpi_{i^{\circ}}, s_{i^{\circ}} \alpha_{\sigma(j^{\circ})})
= \delta(i \sim j).
$$
On the other hand, we know $\eta_{ii}(1)=1$ by Lemma~\ref{lemma: eta property}~(\ref{eta_unip}).
Therefore we have $x_{ij}(0)=0$ as desired.

As for the remaining case $(d_{i}, d_{j}) = (1,3)$, we verify
$x_{ij}(u)=0$ for all $u \in \Z$ by a direct computation.
See Subsection~\ref{subsubsection: computation G} below.
\end{proof}

\ber We obtain the following corollaries, 
which are already proved for types $\mathsf{ADE}$ in \cite{HL15, Fujita19}.\er
\begin{corollary}
\label{corollary: tc2}
The integers $\{ \widetilde{c}_{ij}(u) \mid i,j \in I, u \in \Z \}$ enjoy the following properties:
\begin{enumerate}
\item \label{tc2_+rh} $\widetilde{c}_{ij}(u+rh^{\vee}) = - \widetilde{c}_{ij^{*}}(u)$ for $u \ge 0$.
\item \label{tc2_+2rh} $\widetilde{c}_{ij}(u + 2rh^{\vee}) = \widetilde{c}_{ij}(u)$ for $u \ge 0$.
\item \label{tc2_rh-} $\widetilde{c}_{ij}(rh^{\vee}-u) = \widetilde{c}_{ij^{*}}(u)$ for $0 \le u \le rh^{\vee}$.
\item \label{tc2_2rh-} $\widetilde{c}_{ij}(2rh^{\vee} -u) = -\widetilde{c}_{ij}(u)$ for $0 \le u \le 2rh^{\vee}$.
\item \label{tc2_zeros} $\widetilde{c}_{ij}(u) = 0$ if $|u- k r h^{\vee}| \le d_{i} -\delta_{ij}$ for some $k \in \Z_{\ge 0}$.
\item \label{tc2_pos} $\widetilde{c}_{ij}(u) \ge 0$ for $0 \le u \le rh^{\vee}$.
\item \label{tc2_neg} $\widetilde{c}_{ij}(u) \le 0$ for $rh^{\vee} \le u \le 2rh^{\vee}$.
\end{enumerate}
\end{corollary}
\begin{proof}
The property~(\ref{tc2_+rh}) follows from Theorem~\ref{theorem: combinatorial formula}
and Lemma~\ref{lemma: eta property}~(\ref{eta_+rh}).
The property~(\ref{tc2_+2rh}) is a consequence of the property~(\ref{tc2_+rh}).
Let us verify the property~(\ref{tc2_rh-}). For the case $d_{i} \le d_{j}$,
we compute:
\begin{align*}
\tc_{ij}(rh^{\vee} - u)
&= - \eta_{ij^{*}}(-u) &&
\text{(by Theorem~\ref{theorem: combinatorial formula} and Lemma~\ref{lemma: eta property}~(\ref{eta_+rh}))}\\
&= \sum_{l=0}^{d_{j}-d_{i}} \eta_{j^{*}i}(u +d_{i}-d_{j}+2l) &&\text{(by Lemma~\ref{lemma: eta property}~(\ref{eta_ARduality}))}\\
&= \sum_{l=0}^{d_{j}-d_{i}} \tc_{j^{*}i}(u +d_{i}-d_{j}+2l) &&
\text{(by Lemma~\ref{lemma: eta property}~(\ref{eta_small vanish}) and Theorem~\ref{theorem: combinatorial formula})}\\
&= \tc_{ij^{*}}(u).  && \text{(by Lemma~\ref{lemma: tc1}~(\ref{tc1_swap}))}
\end{align*}
The other case $d_{i} > d_{j}$ can be proved by induction on $u$ using Lemma~\ref{lemma: tc1}~(\ref{tc1_swap})
and Lemma~\ref{lemma: eta property}~(\ref{eta_ARduality}).
The property~(\ref{tc2_2rh-}) is immediate \ber from \er the properties~(\ref{tc2_+rh}) and~(\ref{tc2_rh-}).
The property~(\ref{tc2_zeros}) is a consequence of Theorem~\ref{theorem: combinatorial formula} and
Lemma~\ref{lemma: eta property}~(\ref{eta_+rh}), (\ref{eta_small vanish}).
To verify the property (\ref{tc2_pos}), it is enough to show $\tc_{ij}(u) \ge 0$ for $d_{i}\le u \le rh^{\vee}-d_{i}$
thanks to (\ref{tc2_zeros}).
For the case $\sigma = \id$, see~\cite[Lemma 3.7]{Fujita19} for instance.
Here we focus on the case $\sigma \neq \id$
although our discussion here is quite similar to that for the case $\sigma = \id$.
We take a Q-datum $\calQ=(\Delta, \sigma, \uxi)$ for $\g$
satisfying the condition~(\ref{eq: condition twisted Coxeter height}) and
such that $\uxi_{ij} := \uxi_{j^{\circ}} - \uxi_{i^{\circ}} \in \{ 0,1\}$.
By the parity reason, we may assume $u+\uxi_{ij}-d_{i} \in 2\Z$.
By Remark~\ref{remark: 2seg},
we see that
$\tau_{\calQ}^{(u+\uxi_{ij}-d_{i})/2}(\gamma^{\calQ}_{j^{\circ}}) \in \sfR^{+}$
whenever $0 \le u + \uxi_{ij} - d_{i} \le rh^{\vee}$.
This last condition is always satisfied under our assumption
$d_{i}\le u \le rh^{\vee}-d_{i}$.
Since $(\varpi_{i^{\circ}}, \alpha) \ge 0$ holds if $\alpha \in \sfR^{+}$,
we obtain $\tc_{ij}(u) = (\varpi_{i^{\circ}}, \tau_{\calQ}^{(u+\uxi_{ij}-d_{i})/2}(\gamma^{\calQ}_{j^{\circ}})) \ge 0$
as desired.
The last property~(\ref{tc2_neg}) follows from (\ref{tc2_+rh}) and (\ref{tc2_pos}).
\end{proof}

\subsection{Explicit computations}
\label{subsection: computation}
Theorem~\ref{theorem: combinatorial formula} tells us that one can compute the explicit values of
the integers $\tc_{ij}(u)$ from any (folded) AR quiver $\bGamma_{\calQ}$.
In this subsection, we carry out such case-by-case computations for all non-simply-laced $\g$.
The results are used in Section~\ref{section: denominator} to deduce a unified denominator formula
for the normalized R-matrices.
\ber We note that a list of explicit values of $\tC_{ij}(z)$ for all $\g$ appeared also in \cite[Appendix A]{GTL17}
without a proof. \er

\begin{remark}
Assume $\g$ is not simply-laced, or equivalently $\sigma \neq \id$.
To compute all the values of $\{\tc_{ij}(u) \mid i,j \in I, u\in\Z\}$,
it is enough to compute
$$
(1+z^{rh^{\vee}})\tC_{ij}(z) = \sum_{u=0}^{rh^{\vee}-1} \tc_{ij}(u)z^{u}
$$
for each $i,j \in I$ thanks to Corollary~\ref{corollary: tc2}~(\ref{tc2_+rh}).
\end{remark}

\subsubsection{Type $\mathsf{B}_{n}$}
\label{subsubsection: computation B}

Let $\g$ be of type $\mathsf{B}_{n}$.
The corresponding pair in Table~\ref{table: classification} is $(\sfg, \sigma) = (\mathsf{A}_{2n-1}, \vee)$.
We use the labeling $\Delta_{0}=\{1, \ldots, 2n-1\}, I =\{1, \ldots, n\}$ as in~(\ref{eq: B_n}).
The involution $\vee$ on $\Delta$ is given by $i^{\vee} = 2n-i$ for $1 \le i \le 2n-1$,
and we set $\bar{i} =i$ for $1 \le i \le n$.
Note that we have $(d_{1},\ldots,d_{n}) = (2, \ldots, 2, 1)$ and $rh^{\vee}=4n-2$.

\begin{theorem} \label{theorem: closed B}
For $i, j \in I=\{1,2,\ldots,n\}$, the closed formula of $\tC_{ij}(z)$ is given as follows:
\begin{align}\label{eq: pre denom A}
(1+z^{4n-2})\tC_{ij}(z) =
\begin{cases}
\sum_{s=1}^{\min(i,j) } (z^{2(|i-j|+2s-1)} +  z^{2(2n-i-j-2+2s)}) & \text{ if $1\le i,j < n$}, \\
\sum_{s=1}^{i} z^{2n-2i-3+4s} & \text{ if $i < n, j=n$}, \\
\sum_{s=1}^{j}  (z^{2n-2j-4+4s}+ z^{2n-2j-2+4s})    & \text{ if $i = n, j<n$}, \\
\sum_{s=1}^{n}   z^{4s-3} & \text{ if $i =j = n$}.
\end{cases}
\end{align}
\end{theorem}

\ber
To prove Theorem~\ref{theorem: closed B},
let us choose the height function $\uxi \colon \Delta_{0} \to \Z$ on $(\Delta, \vee)$
given by
$$\uxi_{i}=\begin{cases}
-2i+2 & \text{if $1\le i < n$}, \\
-2n+1 & \text{if $i=n$}, \\
2i - 4n -4 & \text{if $n < i \le 2n-1$}.
\end{cases}
$$
\begin{center}
    \begin{tikzpicture}[xscale=1.5,yscale=1,baseline=0]
\node[dynkdot,label={above:\footnotesize$\uxi_1 = 0$}, label={below:\footnotesize$1$}] (V1) at (1,0){};
\node[dynkdot,label={above:\footnotesize$-2$}, label={below:\footnotesize$2$}] (V2) at (2,0) {};
\node[dynkdot,label={above:\footnotesize$-4$}, label={below:\footnotesize$3$}] (V3) at (3,0) {};
\node (V4) at (4,0) {$\cdots$};
\node[dynkdot, label={above:\footnotesize$-2n+2$}] (V5) at (5,0){};
\node[dynkdot,label={above:\footnotesize$-2n+1$}, label={below:\footnotesize$n$}] (V6) at (6,0) {};
\node[dynkdot,label={above:\footnotesize$-2n$}] (V7) at (7,0) {};
\node (V8) at (8,0) {$\cdots$};
\node[dynkdot,label={above:\footnotesize$-4$}] (V9) at (9,0){};
\node[dynkdot,label={above:\footnotesize$-2$}, label={below:\footnotesize$2n-1$}] (V10) at (10,0) {};
\draw[->] (V1) -- (V2);
\draw[->] (V2) -- (V3);
\draw[->] (V3) -- (V4);
\draw[->] (V4) -- (V5);
\draw[->] (V5) -- (V6);
\draw[<-] (V6) -- (V7);
\draw[<-] (V7) -- (V8);
\draw[<-] (V8) -- (V9);
\draw[<-] (V9) -- (V10);
\end{tikzpicture}
\end{center}
\er
Clearly, the corresponding Q-datum $\calQ:=(\Delta, \vee, \uxi)$ satisfies the condition~(\ref{eq: condition twisted Coxeter height})
and $I_{\calQ}^{\circ} = \{1,2, \ldots, n\}$.
Thus the corresponding twisted Coxeter element $\tau_{\calQ}$
is given by
$\tau_{\calQ} = s_{1}s_{2} \cdots s_{n}\vee$
and
$\btau_{\calQ} = \tau_{\calQ}^{2} = s_{1}s_{2} \cdots s_{n-1}s_{n}s_{2n-1}s_{2n-2} \cdots s_{n+1}s_{n}$.
Let us realize the root lattice $\sfQ$ inside the lattice $\bigoplus_{i=1}^{2n}\Z \epsilon_{i}$
in a standard way as $\alpha_{i} = \epsilon_{i}-\epsilon_{i+1}$ for $1\le i < 2n$.
Then we have $\sfR^{+} = \{ [k,l] \mid 1 \le k \le l < 2n\}$,
where $[k,l] := \epsilon_{k} - \epsilon_{l+1}$.
We simplify $[i] := [i, i] = \alpha_{i}$ for $1 \le i < 2n$.
Using Lemma~\ref{lemma: gamma}, we have
$
\gamma^{\calQ}_{i} =
[1,i]
$ for $1 \le i \le n$.

By a direct computation, we can check the following.

\begin{lemma}
\label{lemma: computation B}
\ber With \er the above notation, we \ber have\er:
\begin{align*}
\tau_{\calQ}^{2t-2}(\gamma^{\calQ}_{i})
&= \begin{cases}
[t, t+i-1] & \text{if $1 \le i <n$ and $1 \le t \le n-i$}, \\
[t, 3n-t-i] & \text{if $1 \le i <n$ and $n-i < t \le n$}, \\
[t,n] & \text{if $i=n$ and $1 \le t \le n$},
\end{cases}
\allowdisplaybreaks\\
\tau_{\calQ}^{2t-1}(\gamma^{\calQ}_{i})
&= \begin{cases}
[2n-t-i+1, 2n-t] & \text{if $1 \le i <n$ and $1 \le t \le n-i$}, \\
[-n+t+i+1, 3n-t-i-1] & \text{if $1 \le i <n$ and $n-i < t < n$},\\
[n+1, 2n-t] & \text{if $i=n$ and $1 \le t < n$}.
\end{cases}
\end{align*}
\end{lemma}
In view of (\ref{eq: bphi}), Lemma~\ref{lemma: computation B} enables us to depict the folded AR quiver $\bGamma_{\calQ}$.

\begin{example}
When $n=4$,
the folded AR quiver $\bGamma_{\calQ}$ can be depicted as follows:
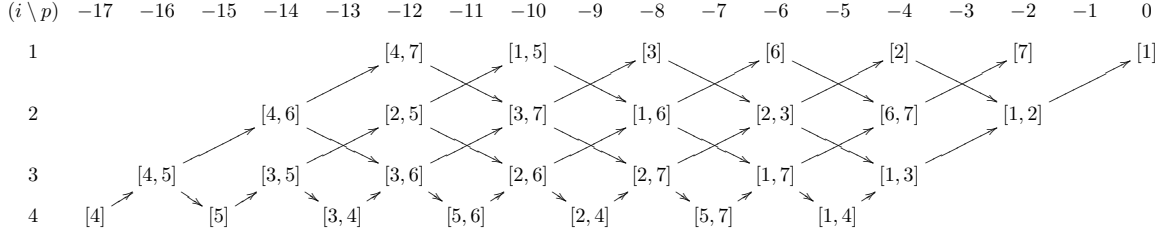
\ber
$$
\scalebox{0.7}{\xymatrix@!C=2ex@R=1ex{
(i\setminus p)  &-17&-16&-15&-14&-13&-12& -11 &-10& -9 &-8&  -7 &-6&  -5 &-4& -3  & -2 & -1 & 0 \\
1 &&&&&& [4,7]\ar[ddrr] &&[1,5]\ar[ddrr]&& [3]\ar[ddrr] &&[6]\ar[ddrr]&&  [2]\ar[ddrr] &&[7] && [1]  \\ \\
2 &&&& [4,6] \ar[ddrr]\ar[uurr]&&[2,5]\ar[uurr]\ar[ddrr]&&  [3,7]\ar[uurr]\ar[ddrr] &&[1,6]\ar[uurr]\ar[ddrr]&& [2,3]\ar[ddrr]\ar[uurr] &&[6,7]  \ar[uurr] && [1,2] \ar[uurr]  \\ \\
3 &&[4,5]\ar[uurr]\ar[dr] &&[3,5]\ar[dr]\ar[uurr]&& [3,6]\ar[dr]\ar[uurr]&&[2,6]\ar[uurr]\ar[dr]&& [2,7]\ar[dr]\ar[uurr]&&[1,7]\ar[uurr] \ar[dr]&& [1,3]\ar[uurr]  \\
4 &[4]\ar[ur] && [5]\ar[ur]&&[3,4]\ar[ur]&&[5,6]\ar[ur]&&[2,4]\ar[ur]&&[5,7]\ar[ur]&&[1,4]\ar[ur]
}}
$$
\er
\end{example}

\begin{lemma}
\label{lemma: support B}
For $1 \le i,j \le n$, the set
$S_{ij}(1) := \{ t \mid 0 \le t < 2n-1, (\varpi_{i}, \tau_{\calQ}^{t}(\gamma^{\calQ}_{j})) = 1 \}$
is given \ber by\er:
\begin{equation}
\label{eq: support B}
S_{ij}(1) =
\begin{cases}
\{ 2k \}_{k=0}^{i-1} \sqcup \{ 2n-2j-1+2k \}_{k=0}^{i-1} & \text{if $i \le j<n$}, \\
\{ 2i-2j+2k \}_{k=0}^{j-1} \sqcup \{ 2n-2j-1+2k \}_{k=0}^{j-1} & \text{if $j < i \le n$}, \\
\{2k\}_{k=0}^{i-1} & \text{if $i \le j=n$}.
\end{cases}
\end{equation}
\end{lemma}
\begin{proof}
First we observe that
\begin{equation}
\label{eq: support root B}
(\varpi_{i}, [k,l]) = \begin{cases}
1 & \text{if $k \le i \le l$}, \\
0 & \text{otherwise},
\end{cases}
\end{equation}
for any $1 \le i < 2n$ and $1 \le k \le l < 2n$.
Using this observation and Lemma~\ref{lemma: computation B},
we can easily check that the RHS of~(\ref{eq: support B}) is included in $S_{ij}(1)$.
To see that this inclusion is actually an equality, we can use the fact
$\sum_{\alpha \in \sfR^{+}} (\varpi_{i}, \alpha) = 2ni-i^{2}$ for any $1 \le i < 2n$,
which also follows from (\ref{eq: support root B}).
\end{proof}

\begin{proof}[Proof of Theorem~\ref{theorem: closed B}]
Let $1 \le i, j \le n$. For $t \in S_{ij}(1)$, let us
compute $s$ by the equation
$t = (s+\uxi_j-\uxi_i -d_{i})/2$ using Lemma~\ref{lemma: support B}.
Then one can easily check that $s$ coincides
with the exponents appearing in the RHS of in~\eqref{eq: pre denom A}.
Hence our assertion follows from Theorem~\ref{theorem: combinatorial formula}.
\end{proof}

\subsubsection{Type $\mathsf{C}_{n}$}
\label{subsubsection: computation C}

Let $\g$ be of type $\mathsf{C}_{n}$.
The corresponding pair in Table~\ref{table: classification} is $(\sfg, \sigma) = (\mathsf{D}_{n+1}, \vee)$.
We use the labeling $\Delta_{0}=\{1,\ldots, n+1\}, I=\{1, \ldots, n\}$ as in~(\ref{eq: C_n}).
The involution $\vee$ on $\Delta$ is given by
$$k^{\vee} = \begin{cases} k & \text{ if } k \le n-1, \\ n+1 & \text{ if } k = n, \\ n & \text{ if } k = n+1. \end{cases}$$
We set $\bar{k} =k$ for $1 \le k \le n$.
Note that $(d_{1},\ldots,d_{n}) = (1, \ldots, 1, 2)$ and $rh^{\vee} = 2n+2$.

\begin{theorem} \label{theorem: closed C}
For $i, j \in I=\{1,2,\ldots,n\}$, the closed formula of $\tC_{ij}(z)$ is given as follows:
\begin{align}\label{eq: pre denom D}
(1+z^{2n+2})\tC_{ij}(z) =
\begin{cases}
\sum_{s=1}^{\min(i,j)} (z^{|i-j|+2s-1} +  z^{2n-i-j+2s+1}) & \text{ if $i,j < n$}, \\
\sum_{s=1}^{i} (z^{n-1-i+2s}+z^{n+1-i+2s}) & \text{ if $i < n, j=n$}, \\
\sum_{s=1}^{j}  z^{n-j+2s}    & \text{ if $i = n, j<n$}, \\
\sum_{s=1}^{n}   z^{2s} & \text{ if $i =j = n$}.
\end{cases}
\end{align}
\end{theorem}

\ber
To prove Theorem~\ref{theorem: closed C},
we choose the height function $\uxi \colon \Delta_{0} \to \Z$ on $(\Delta, \vee)$
given by
$$\uxi_{i}=\begin{cases}
-i+1 & \text{if $1\le i \le n$}, \\
-n -1 & \text{if $i=n+1$}.
\end{cases}
$$
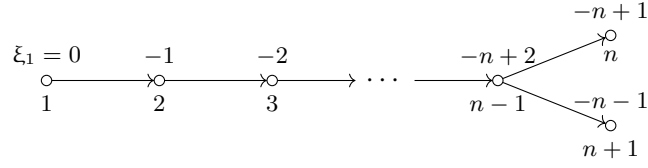
\begin{center}
    \begin{tikzpicture}[xscale=1.5,yscale=1,baseline=0]
\node[dynkdot,label={above:\footnotesize$\uxi_1 = 0$}, label={below:\footnotesize$1$}] (V1) at (1,0){};
\node[dynkdot,label={above:\footnotesize$-1$}, label={below:\footnotesize$2$}] (V2) at (2,0) {};
\node[dynkdot,label={above:\footnotesize$-2$}, label={below:\footnotesize$3$}] (V3) at (3,0) {};
\node (V4) at (4,0) {$\cdots$};
\node[dynkdot,label={above:\footnotesize$-n+2$}, label={below:\footnotesize$n-1$}] (V5) at (5,0) {};
\node[dynkdot,label={above:\footnotesize$-n+1$}, label={below:\footnotesize$n$}] (V6) at (6,0.6) {};
\node[dynkdot,label={above:\footnotesize$-n-1$}, label={below:\footnotesize$n+1$}] (V7) at (6,-0.6){};
\draw[->] (V1) -- (V2);
\draw[->] (V2) -- (V3);
\draw[->] (V3) -- (V4);
\draw[->] (V4) -- (V5);
\draw[->] (V5) -- (V6);
\draw[->] (V5) -- (V7);
\end{tikzpicture}
\end{center}
\er
Clearly, the corresponding Q-datum $\calQ:=(\Delta, \vee, \uxi)$ satisfies the condition~(\ref{eq: condition twisted Coxeter height})
and $I_{\calQ}^{\circ} = \{1,2, \ldots, n\}$.
Thus the corresponding twisted Coxeter element $\tau_{\calQ}$
is given by
$\tau_{\calQ} = s_{1}s_{2} \cdots s_{n}\vee$
and
$\btau_{\calQ} = \tau_{\calQ}^{2} = s_{1}s_{2} \cdots s_{n-1}s_{n}s_{1}s_{2} \cdots s_{n-1}s_{n+1}$.
Let us realize the root lattice $\sfQ$ inside the lattice $\bigoplus_{i=1}^{n+1}\Z \epsilon_{i}$
in a standard way as $\alpha_{i} = \epsilon_{i}-\epsilon_{i+1}$ for $1\le i \le n$
and $\alpha_{n+1} = \epsilon_{n} + \epsilon_{n+1}$.
Then we have $\sfR^{+} = \{ \langle k, \pm l\rangle \mid 1 \le k < l \le n+1\}$,
where $\langle k, \pm l \rangle := \epsilon_{k} \pm \epsilon_{l}$.
Using Lemma~\ref{lemma: gamma}, we have
$
\gamma^{\calQ}_{i} =
\langle 1, -(i+1)\rangle
$ for $1 \le i \le n$.

By a direct computation, we can check the following.

\begin{lemma}
\label{lemma: computation C}
\ber With \er the above notation, for each $1 \le i \le n$, we \ber have\er:
\begin{align*}
\tau_{\calQ}^{t-1}(\gamma^{\calQ}_{i})
&= \begin{cases}
\langle t, -(i+1) \rangle & \text{if $1 \le t \le n+1-i$}, \\
\langle t+i-n-1, t \rangle & \text{if $n+1-i<t\le n+1$}.
\end{cases}
\end{align*}
\end{lemma}
In view of (\ref{eq: bphi}), Lemma~\ref{lemma: computation C} enables us to depict the folded AR quiver $\bGamma_{\calQ}$.

\begin{example}
When $n=3$,
the folded AR quiver $\bGamma_{\calQ}$ can be depicted as follows:
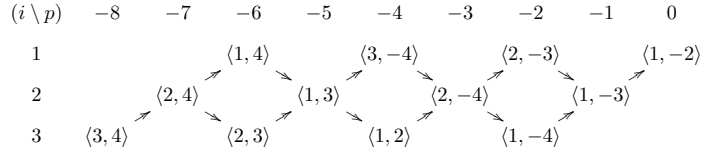
\ber
$$
\scalebox{0.7}{\xymatrix@!C=3ex@R=1ex{
(i\setminus p)  & -8 & -7 & -6 & -5 & -4 & -3 & -2 &-1 & 0\\
1&&& \lr{1,4} \ar@{->}[dr] && \lr{3,-4}   \ar@{->}[dr] && \lr{2,-3}   \ar@{->}[dr]&& \lr{1,-2} \\
2&& \lr{2,4}  \ar@{->}[dr]\ar@{->}[ur]&& \lr{1,3}  \ar@{->}[dr]\ar@{->}[ur] && \lr{2,-4} \ar@{->}[dr]\ar@{->}[ur] && \lr{1,-3}\ar@{->}[ur] \\
3& \lr{3,4}  \ar@{->}[ur] && \lr{2,3} \ar@{->}[ur]  && \lr{1,2}  \ar@{->}[ur] && \lr{1,-4}\ar@{->}[ur]  }}
$$
\er
\end{example}

\begin{lemma}
\label{lemma: support D}
For $1 \le i,j \le n$ and $k \in \Z_{\ge 0}$, we define
$S_{ij}(k) := \{ t \mid 0 \le t < 2n-1, (\varpi_{i}, \tau_{\calQ}^{t}(\gamma^{\calQ}_{j})) =k \}$
and $S_{ij}(\ge k) := \bigcup_{l \ge k}S_{ij}(l)$. Then the followings hold:
\begin{enumerate}
\item For $i \le j < n$, we have
\begin{align*}
S_{ij}(\ge 1)&= \{ k \}_{k=0}^{i-1} \cup \{ n-j+1+k\}_{k=0}^{i-1}, \\
S_{ij}(2)&= \{ k \}_{k=0}^{i-1} \cap \{ n-j+1+k\}_{k=0}^{i-1}.
\end{align*}
\item For $j < i < n$, we have
\begin{align*}
S_{ij}(\ge 1)&= \{ i-j+k \}_{k=0}^{j-1} \cup \{ n-j+1+k\}_{k=0}^{j-1}, \\
S_{ij}(2)&= \{ i-j+k \}_{k=0}^{j-1} \cap \{ n-j+1+k\}_{k=0}^{j-1}.
\end{align*}
\item For $i<j=n$, we have
$
S_{ij}(\ge 1)= \{ k \}_{k=0}^{i}$
and $
S_{ij}(2)= \{ k \}_{k=1}^{i-1}$.
\item For $j \le i=n$, we have
$
S_{ij}(\ge 1) = S_{ij}(1) = \{ n-j+k \}_{k=0}^{j-1}.
$
\end{enumerate}
\end{lemma}
\begin{proof}
First we observe that
\begin{align}
\label{eq: support root C}
(\varpi_{i}, \lr{k,-l}) &= \begin{cases}
1 & \text{if $k \le i < l$}, \\
0 & \text{otherwise},
\end{cases}
&
(\varpi_{i}, \lr{k,l}) &= \begin{cases}
2 & \text{if $l \le i \le n-1$} \\
1 & \text{if $k \le i < l$ or $i \in \{n, n+1\}$}, \\
0 & \text{otherwise},
\end{cases}
\end{align}
for any $1 \le i \le n+1$ and $1 \le k < l \le n+1$.
Using this observation and Lemma~\ref{lemma: computation C},
we can easily check that the LHS includes the RHS for each equation in (1)--(4).
To see that these inclusions are actually equalities, we can use the fact
$$\sum_{\alpha \in \sfR^{+}} (\varpi_{i}, \alpha) =
\begin{cases}
2ni-i^{2}+3i & \text{if $i < n$}, \\
(n^{2}+n)/2 & \text{if $i=n$},
\end{cases}
$$
which also follows from (\ref{eq: support root C}).
\end{proof}

\begin{proof}[Proof of \ber Theorem~\ref{theorem: closed C}\er]
Let $1 \le i, j \le n$. For $t \in S_{ij}(\ge 1)$, let us
compute $s$ by the equation
$t = (s+\uxi_j-\uxi_i -d_{i})/2$ using \ber Lemma~\ref{lemma: support D}\er.
Then one can easily check that $s$ coincides
with the exponents appearing in the RHS of \ber in~\eqref{eq: pre denom D}\er.
Hence our assertion follows from Theorem~\ref{theorem: combinatorial formula}.
\end{proof}

\subsubsection{Type $\mathsf{F}_{4}$}
\label{subsubsection: computation F}

Let $\g$ be of type $\mathsf{F}_{4}$.
The corresponding pair in Table~\ref{table: classification} is $(\sfg, \sigma) = (\mathsf{E}_{6}, \vee)$.
We use the labeling $\Delta_{0} = \{1,2,3,4,5,6\}, I=\{1,2,3,4\}$ as in~(\ref{eq: F_4}).
The involution $\vee$ on $\Delta$ is given by
$1^{\vee}=6$, $6^{\vee}=1$, $3^{\vee}=5$, $5^{\vee}=3$, $4^{\vee}=4$, $2^{\vee}=2$
and
we set $1=\bar{1} = \bar{6}$, $2= \bar{3}=\bar{5}$, $3=\bar{4}$, $4 = \bar{2}$.
Note that $(d_{1},d_{2},d_{3}, d_{4}) = (2,2,1,1)$ and $rh^{\vee} = 18$.

The following is an example of the folded AR quiver of a Q-datum for $\g$.
Here $\prt{a_{1} a_{2} a_{3}}{a_{4} a_{5} a_{6}}$ denotes the positive root $\sum_{i=1}^{6} a_{i} \alpha_{i} \in \sfR^{+}$ and \beb the Q-datum can be read from the folded AR quiver. \eb
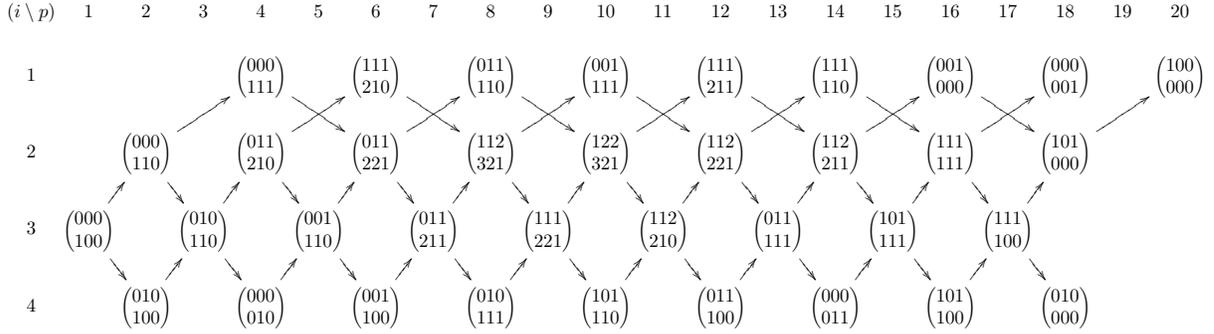
$$
\scalebox{0.65}{\xymatrix@!C=2ex@R=2.6ex{
(i \setminus p) & 1 & 2 & 3 & 4 & 5 & 6 & 7 & 8 & 9 & 10 & 11 & 12 & 13 & 14 & 15 & 16 & 17 & 18  & 19 & 20\\
1 &&&&{\scriptstyle\prt{000}{111}}  \ar[drr]
&&{\scriptstyle\prt{111}{210}}  \ar[drr]
&&{\scriptstyle\prt{011}{110}}  \ar[drr]
&& {\scriptstyle\prt{001}{111}}  \ar[drr]
&& {\scriptstyle\prt{111}{211}}  \ar[drr]
&& {\scriptstyle\prt{111}{110}}  \ar[drr]
&& {\scriptstyle\prt{001}{000}}   \ar[drr]
&&{\scriptstyle\prt{000}{001}}
&& {\scriptstyle\prt{100}{000}} \\
2 &\quad\quad&{\scriptstyle\prt{000}{110}}    \ar[urr]\ar[dr]
&&{\scriptstyle\prt{011}{210}} \ar[urr]\ar[dr]
&&{\scriptstyle\prt{011}{221}} \ar[urr]\ar[dr]
&& {\scriptstyle\prt{112}{321}} \ar[urr]\ar[dr]
&& {\scriptstyle\prt{122}{321}} \ar[urr]\ar[dr]
&& {\scriptstyle\prt{112}{221}} \ar[urr]\ar[dr]
&& {\scriptstyle\prt{112}{211}}  \ar[urr]\ar[dr]
&&{\scriptstyle\prt{111}{111}}  \ar[dr]\ar[urr]
&& {\scriptstyle\prt{101}{000}}  \ar[urr] \\
3 &{\scriptstyle\prt{000}{100}}   \ar[dr]\ar[ur]
&&{\scriptstyle\prt{010}{110}} \ar[dr]\ar[ur]
&&{\scriptstyle\prt{001}{110}} \ar[dr]\ar[ur]
&& {\scriptstyle\prt{011}{211}} \ar[dr]\ar[ur]
&& {\scriptstyle\prt{111}{221}} \ar[dr]\ar[ur]
&& {\scriptstyle\prt{112}{210}} \ar[dr]\ar[ur]
&& {\scriptstyle\prt{011}{111}}  \ar[dr]\ar[ur]
&&{\scriptstyle\prt{101}{111}}\ar[dr]\ar[ur]
&& {\scriptstyle\prt{111}{100}} \ar[dr]\ar[ur] \\
4 &&{\scriptstyle\prt{010}{100}}  \ar[ur]
&&{\scriptstyle\prt{000}{010}} \ar[ur]
&&{\scriptstyle\prt{001}{100}} \ar[ur]
&& {\scriptstyle\prt{010}{111}} \ar[ur]
&& {\scriptstyle\prt{101}{110}} \ar[ur]
&& {\scriptstyle\prt{011}{100}} \ar[ur]
&& {\scriptstyle\prt{000}{011}}  \ar[ur]
&&{\scriptstyle\prt{101}{100}}\ar[ur]
&& {\scriptstyle\prt{010}{000}}
}}
$$
Using Theorem~\ref{theorem: combinatorial formula},
we can compute the explicit values of $\tC_{ij}(z)$ as follows:
\begin{align*}
\tC_{11}(z) &= (z^{2} + z^{8} + z^{10} + z^{16} )/(1+z^{18}), \allowdisplaybreaks\\
\tC_{12}(z) &= \tC_{21}(z) = (z^{4} + z^{6} + z^{8} + z^{10} + z^{12} + z^{14})/(1+z^{18}), \allowdisplaybreaks\\
\tC_{13}(z) &=(z^{5} + z^{7} + z^{11} + z^{13})/(1+z^{18}), \allowdisplaybreaks\\
\tC_{14}(z) &=(z^{6} + z^{12})/(1+z^{18}), \allowdisplaybreaks\\
\tC_{22}(z) &= (z^{2} + z^{4} + 2z^{6} + 2z^{8} + 2z^{10} + 2z^{12} + z^{14} + z^{16})/(1+z^{18}), \allowdisplaybreaks\\
\tC_{23}(z) &= (z^{3} + z^{5} + z^{7} + 2z^{9} + z^{11} + z^{13} + z^{15})/(1+z^{18}), \allowdisplaybreaks\\
\tC_{24}(z) &= (z^{4} + z^{8} + z^{10} + z^{14})/(1+z^{18}), \allowdisplaybreaks\\
\tC_{31}(z) &= (z^{4} + 2z^{6} + z^{8} + z^{10} + 2z^{12} + z^{14})/(1+z^{18}), \allowdisplaybreaks\\
\tC_{32}(z) &= (z^{2} + 2z^{4} + 2z^{6} + 3z^{8} + 3z^{10} + 2z^{12} + 2z^{14} + z^{16})/(1+z^{18}), \allowdisplaybreaks\\
\tC_{33}(z) &= ( z + z^{3} + z^{5} + 2z^{7} + 2z^{9} + 2z^{11} + z^{13} + z^{15} + z^{17})/(1+z^{18}), \allowdisplaybreaks\\
\tC_{34}(z) &= \tC_{43}(z) = (z^{2} + z^{6} + z^{8} + z^{10} + z^{12} + z^{16})/(1+z^{18}), \allowdisplaybreaks\\
\tC_{41}(z) &=(z^{5} + z^{7} + z^{11} + z^{13})/(1+z^{18}), \allowdisplaybreaks\\
\tC_{42}(z) &=(z^{3} + z^{5} + z^{7} + 2z^{9} + z^{11} + z^{13} + z^{15})/(1+z^{18}), \allowdisplaybreaks\\
\tC_{44}(z) &=(z + z^{7} + z^{11} + z^{17})/(1+z^{18}).
\end{align*}

\subsubsection{Type $\mathsf{G}_{2}$}
\label{subsubsection: computation G}

Let $\g$ be of type $\mathsf{G}_{2}$.
The corresponding pair in Table~\ref{table: classification} is $(\sfg, \sigma) = (\mathsf{D}_{4}, \tvee)$ or
$(\mathsf{D}_{4}, \tvee{}^{2})$. Here we shall only consider the case $\sigma = \tvee$.
We use the labeling $\Delta_{0}=\{1,2,3,4\}$ and $I=\{1,2\}$
as in~(\ref{eq: G_2}).
The involution $\tvee$ on $\Delta$ is given by
$1^{\tvee}=3$, $3^{\tvee}=4$, $4^{\tvee}=1$, $2^{\tvee}=2$
and
we set $1 = \bar{1}= \bar{3}=\bar{4}$, $2=\bar{2}$.
Note that $(d_{1},d_{2}) = (3,1)$ and $rh^{\vee}=12$.

The following is an example of the folded AR quiver of a Q-datum for $\g$.
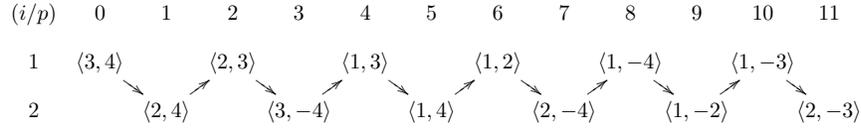
$$
\scalebox{0.75}{\xymatrix@!C=2ex@R=1.5ex{
(i/p) & 0  & 1 & 2 & 3 & 4 & 5 & 6 & 7 & 8 & 9 & 10 & 11 \\
1& \lr{3,4}\ar[dr]&& \lr{2,3}\ar[dr]&& \lr{1,3}\ar[dr]&& \lr{1,2}\ar[dr]&& \lr{1,-4}\ar[dr]&& \lr{1,-3}\ar[dr]\\
2&& \lr{2,4}\ar[ur]&& \lr{3,-4}\ar[ur]&& \lr{1,4}\ar[ur]&& \lr{2,-4}\ar[ur]&& \lr{1,-2}\ar[ur]&& \lr{2,-3}
}}
$$
On the other hand, the quantum Cartan matrix of type $\mathsf{G}_{2}$ is
$$C(z) = \prt{z^{3} + z^{-3} & -1}{-(z^{2}+1+z^{-2}) & z+z^{-1}}.$$
Therefore it is easy to compute its inverse $\tC(z)$ as:
\begin{align*}
\tC(z) &= \frac{1}{z^{4}-1+z^{-4}}\prt{z + z^{-1} & 1}{z^{2}+1+z^{-2} & z^{3}+z^{-3}} \\
&= \frac{1}{1+z^{12}} \prt{z^{3}+z^{5}+z^{7}+z^{9} & z^{4}+z^{8}}{z^{2}+z^{4}+2z^{6}+z^{8}+z^{10} & z+z^{5}+z^{7}+z^{11}}.
\end{align*}
By a comparison, 
\ber we can check that Theorem~\ref{theorem: combinatorial formula} also holds in this case.
Thus the proof of Theorem~\ref{theorem: combinatorial formula} is completed\er.

\section{Representations of untwisted quantum affine algebras}
\label{section: QAffine}

In this section, we recall some basic facts on the 
finite-dimensional representation theory
of untwisted quantum affine algebras.
We also introduce the notion of $\calQ$-weights and 
\ber observe its relation to the block decomposition of 
the finite-dimensional module category established by \cite{CM05}, \cite{KKOP20b}. \er
Let us keep the notation from the previous sections.

\subsection{Untwisted quantum affine algebras}
\label{ssec: qaffine}
Let $\g$ be a finite-dimensional complex simple Lie algebra as before.
We denote by $\hg$ the untwisted affine Lie algebra of $\g$.
The Cartan matrix $C=(c_{ij})_{i,j \in I}$ of $\g$
is extended to the generalized Cartan matrix $C_{\af} = (c_{ij})_{i,j \in I_{\af}}$
of $\hg$, where $I_{\af} := I \cup \{0\}$.
We set $d_{0} := r$ and $r_{0}:=r/d_{0}=1$.

Let $q$ be an indeterminate.
We denote by $\kk := \overline{\Q(q)}$ the algebraic closure of the field $\Q(q)$
inside the ambient field $\bigcup_{m \in \Z_{\ge 1}}\overline{\Q}(\!(q^{1/m})\!)$.
For each $i \in I_{\af}$, we set $q_{i} := q^{d_{i}/r} = q^{1/r_{i}} \in \kk^{\times}$.
We also use the notation $\qs := q^{1/r}$, where $s$ stands for ``short''.
For $a,b \in \Z_{\ge 0}$ with $b \le a$ and $i \in I_{\af}$, we set
$$
[a]_{i}! := \prod_{k=1}^{a} [k]_{q_{i}}, \quad
\left[ \begin{matrix} a \\ b \end{matrix} \right]_{i} := \frac{[a]_{i}!}{[a-b]_{i}! [b]_{i}!}.
$$
\begin{definition}
We define the \emph{quantum affine algebra} $U_{q}^{\prime}(\hg)$ (without the degree operator)
to be the $\kk$-algebra given by the set of generators
$\{ e_{i}, f_{i}, K^{\pm1}_{i} \mid i \in I_{\af} \}$ satisfying the following relations:
\begin{itemize}
\item $K_{i} K_{i}^{-1} = 1 = K_{i}^{-1} K_{i}, K_{i} K_{j} = K_{j}K_{i}$ for $i,j \in I_{\af}$,
\item $K_{i} e_{j} K_{i}^{-1} = q_{i}^{c_{ij}} e_{j}, K_{i} f_{j} K_{i}^{-1} = q_{i}^{-c_{ij}} f_{i}$ for $i,j \in I_{\af}$,
\item $\displaystyle e_{i}f_{j} - f_{j}e_{i} = \delta_{ij}\frac{K_{i}-K_{i}^{-1}}{q_{i}-q_{i}^{-1}}$ for $i,j \in I_{\af}$,
\item $\displaystyle \sum_{k=0}^{1-c_{ij}} (-1)^{k}\left[\begin{matrix} 1-c_{ij} \\ k \end{matrix}\right]_{i}e_{i}^{1-c_{ij}-k}e_{j}e_{i}^{k}
= \sum_{k=0}^{1-c_{ij}} (-1)^{k}\left[\begin{matrix} 1-c_{ij} \\ k \end{matrix}\right]_{i }f_{i}^{1-c_{ij}-k}f_{j}f_{i}^{k} = 0$ for $i \neq j$.
\end{itemize}
\end{definition}

We denote by $U_{q}^{+}(\hg)$ (resp.~$U_{q}^{-}(\hg)$) the $\kk$-subalgebra of $U_{q}^{\prime}(\hg)$
generated by $e_{i}$ (resp.~$f_{i}$) for $i \in I_{\af}$.
The algebra $U_{q}^{\prime}(\hg)$ is equipped with a structure of Hopf algebra over $\kk$
whose coproduct is given by
$$
e_{i} \mapsto e_{i} \otimes K_{i}^{-1} + 1 \otimes e_{i},
\qquad f_{i} \mapsto f_{i} \otimes 1 + K_{i} \otimes f_{i},
\qquad K_{i} \mapsto K_{i} \otimes K_{i},
$$
for $i \in I_{\af}$.
We denote by $\scrC$ the category of finite-dimensional $U_{q}^{\prime}(\hg)$-modules
of type $\bf{1}$, i.e.~finite-dimensional modules on which the element $K_{i}$
acts as a diagonalizable linear operator whose eigenvalues belong to the set $\{q_{i}^{k} \mid k \in \Z\}$
for each $i \in I_{\af}$. 
\ber It is well-known that the description of general finite-dimensional $U_{q}^{\prime}(\hg)$-modules 
are essentially reduced to that of finite-dimensional $U_{q}^{\prime}(\hg)$-modules of type $\bf{1}$. \er
The $\kk$-linear abelian category $\scrC$
becomes a rigid monoidal category by the above Hopf algebra structure of $U_{q}^{\prime}(\hg)$.
We denote the right dual (resp.~left dual) of an object $V \in \scrC$ by ${}^{*}V$ (resp.~$V^{*}$).
For $k \in \Z$ and a module $V$ in $\scrC$, we define
$$
\scrD^k(V) \seteq
\begin{cases}
(\cdots( V^* \underbrace{ )^* \cdots )^* }_{\text{$(-k)$-times}} & \text{ if } k <0, \\
 \underbrace{\rd ( \cdots ( }_{\text{$k$-times}} \rd V  )\cdots)  & \text{ if } k \ge 0.
\end{cases}
$$

The category $\scrC$ is neither semisimple as an abelian category, nor braided as a monoidal category.
We say that two objects $V, W \in \scrC$ \emph{mutually commute} if we have
$V \otimes W \cong W \otimes V$ as $U_{q}^{\prime}(\hg)$-modules.
We also say that a simple object $V \in \scrC$ is \emph{real}
if the tensor square $V \otimes V$ remains simple.

Recall the central element $q^{c} := \prod_{i \in I_{\af}}K_{i}^{a_{i}} \in U_{q}^{\prime}(\hg)$,
where the positive integers $(a_{i})_{i \in I_{\af}}$ are defined as in~\cite[Chapter 4, Table Aff1]{Kac}.
It is well known that for any module $V \in \scrC$, the element $q^{c}$ acts by the identity.
The quotient algebra $U_{q}^{\prime}(\hg)/\langle q^{c} - 1 \rangle$ is isomorphic to
the \emph{quantum loop algebra} $U_{q}(L\g)$ defined below \ber (see~\cite{Beck94, CP, Damiani98, Damiani12, Damiani15})\er.
Via this isomorphism $U_{q}^{\prime}(\hg)/\langle q^{c} - 1 \rangle \cong U_{q}(L\g)$,
we can identify the category $\scrC$ with the category of finite-dimensional
$U_{q}(L\g)$-modules of type $\bf{1}$.

\begin{definition}
The quantum loop algebra $U_{q}(L\g)$
is the $\kk$-algebra given by the set of generators
$$\{ k_{i}^{\pm 1} \mid i \in I\} \cup \{ x_{i,k}^{\pm} \mid i \in I, k \in \Z\} \cup \{ h_{i, l} \mid i \in I, l \in \Z \setminus\{0\}\}$$
satisfying the following relations:
\begin{itemize}
\item $k_{i} k_{i}^{-1} = 1 = k_{i}^{-1} k_{i}, k_{i} k_{j} = k_{j} k_{i}$ for $i,j \in I$,
\item $k_{i} x_{j, k}^{\pm} k_{i}^{-1} = q_{i}^{\pm c_{ij}} x_{j,k}^{\pm}$ for $i,j \in I$ and $k \in \Z$,
\item $[k_{i}, h_{j,l}] = [h_{i,l}, h_{j,m}] = 0$ for $i,j \in I$ and $l,m \in \Z \setminus\{0\}$,
\item $[x_{i}^{+}(z), x_{j}^{-}(w)]
= \displaystyle \frac{\delta_{ij}}{q_{i} - q_{i}^{-1}}
\left( \delta(z/w) \phi_{i}^{+}(w) - \delta(w/z) \phi_{i}^{-}(z) \right)$ for $i, j \in I$,
\item $(q^{\pm c_{ij}}_{i}z-w) x_{i}^{\pm}(z) x_{j}^{\pm}(w) = (z-q^{\pm c_{ij}}_{i}w) x_{j}^{\pm}(w) x_{i}^{\pm}(z)$ for $i,j \in I$,
\item $(q^{\pm c_{ij}}_{i}z-w) \phi_{i}^{\varepsilon}(z) x_{j}^{\pm}(w) = (z-q^{\pm c_{ij}}_{i}w) x_{j}^{\pm}(w) \phi_{i}^{\varepsilon}(z)$
for $i,j \in I$ and $\varepsilon \in \{ +, -\}$,
\item $\displaystyle \sum_{g \in \mathfrak{S}_{1-c_{ij}}} \sum_{k=0}^{1-c_{ij}}
(-1)^{k} \left[\begin{matrix} 1-c_{ij} \\ k\end{matrix} \right]_{i}
x_{i}^{\pm}(z_{g(1)}) \cdots x_{i}^{\pm}(z_{g(k)}) x_{j}^{\pm}(w) x_{i}^{\pm}(z_{g(k+1)}) \cdots x_{i}^{\pm}(z_{g(1-c_{ij})}) = 0
$ for $i, j \in I$ with $i \sim j$,
\end{itemize}
where $\delta(z), x_{i}^{\pm}(z), \phi_{i}^{\pm}(z)$ are formal power series defined by
$$
\delta(z) := \sum_{l=-\infty}^{\infty}z^{l},
\quad
x_{i}^{\pm}(z) := \sum_{l=-\infty}^{\infty}x_{i,k}^{\pm}z^{k},
\quad
\phi_{i}^{\pm}(z) := k_{i}^{\pm 1}
\exp \left( \pm (q_{i}-q_{i}^{-1})
\sum_{l=1}^{\infty}h_{i, \pm l} z^{\pm l} \right).
$$
\end{definition}

\begin{remark}
\label{remark: Beck}
Under the isomorphism $U_{q}^{\prime}(\hg) / \langle q^{c}-1\rangle \cong U_{q}(L\g)$,
the generators $e_{i}, f_{i}, K_{i}$ correspond to
$x^{+}_{i, 0}, x^{-}_{i,0}, k_{i}$ respectively for each $i \in I$.
Strictly speaking, the isomorphism depends on the choice
of a function $o \colon I \to \{ \pm 1\}$
such that $o(i) \neq o(j)$ if $i \sim j$.
In what follows, we choose and fix such a function $o$, although the results in this paper will not depend on this choice.
\end{remark}

Let $z$ be an indeterminate.
For any object $V \in \scrC$, we can consider its \emph{affinization} \ber $V_{z}$. \er
This is the $\kk[z^{\pm}]$-module $V_{z} := V \otimes \kk[z^{\pm1}]$ equipped with a structure of
left $U_{q}(L\g)$-module by
$$
k_{i}(v \otimes f(z)) = (k_{i}v) \otimes f(z), \qquad
x^{\pm}_{i, k} (v \otimes f(z)) = (x^{\pm}_{i,k}v)\otimes z^{k} f(z), \qquad
h_{i,l}(v \otimes f(z)) = (h_{i,l}v) \otimes z^{l}f(z),
$$
for any $v \in V$ and $f(z) \in \kk[z^{\pm 1}]$.
For a non-zero scalar $a \in \kk^{\times}$, we set
$V_{a} := V_{z}/(z-a)V_{z}$, which is an object of $\scrC$.
The assignment $V \mapsto V_{a}$ defines a monoidal \ber self\er-equivalence $T_{a}$ of $\scrC$.

\subsection{Simple modules and $q$-characters}

A complete classification of the simple objects of $\scrC$ up to isomorphism
was given by Chari-Pressley (see \cite{CP95} or \cite[Chapter 12]{CP}) in terms of the so-called \emph{Drinfeld polynomials}.
Let $\fkD^{+} := (1 + z\kk[z])^{I}$ denote the multiplicative monoid
consisting of $I$-tuples of polynomials with constant terms $1$.
This is a commutative monoid freely generated by
the elements $\pi_{i, a} := ((1-az)^{\delta_{ij}})_{j \in I}$ for $i \in I, a \in \kk^{\times}$.

\begin{theorem}[\cite{CP95}]
For each $P= (P_{i}(z))_{i \in I} \in \fkD^{+}$, there exists a simple object $L(P) \in \scrC$
with a unique line $\kk v \subset L(P)$ such that, in $L(P)\lf z, z^{- 1} \rf$, \ber we have\er
$$
x^{+}_{i}(z) v = 0, \qquad \phi^{\pm}_{i}(z)v = q_{i}^{\deg P_{i}(z)}\left[\frac{P_{i}(\ber q_{i}^{-1}\er z)}{P_{i}(\ber q_{i}\er z)} \right]^{\pm} v
\qquad \text{for each $i \in I$},
$$
where $[f(z)]^{\pm}$ denotes the formal Laurent expansion of the rational function $f(z) \in \kk(z)$
at $z^{\pm 1}=0$.
The correspondence $P(z) \mapsto L(P(z))$ gives a bijection between
$\fkD^{+}$ and
the set $\Irr \scrC$ of isomorphism classes of simple objects of $\scrC$.
\end{theorem}
We call the vector $v$ the \emph{$\ell$-highest weight vector} of $L(P)$,
which is unique up to \ber multiplication by \er $\kk^{\times}$.
By the above characterization, we have $L(P)_{a} \cong L(P_{a})$ with
$P_{a} := (P_{i}(az))_{i \in I}$.

For $(i,a) \in I \times \kk^{\times}$ and $l \in \Z_{\ge 0}$,
we set
$$
\pi^{(i)}_{l,a} := \prod_{k=0}^{l-1}\pi_{i, aq_{i}^{2k}} \in \fkD^{+}.
$$
The corresponding simple modules $L(\pi^{(i)}_{l, a})$ are called  \emph{Kirillov-Reshetikhin (KR) modules}.
In particular,
when $l=1$, the modules $L(\pi_{i,a})$ are called \emph{fundamental modules}.

Next we recall the notion of \emph{$q$-characters} introduced by
Frenkel-Reshetikhin~\cite{FR99}.
Let $\fkD$ denote the multiplicative group of
$I$-tuples of rational functions $\Psi = (\Psi_{i}(z))_{i \in I} \in \kk(z)^{I}$
satisfying $\Psi_{i}(1) = 1$ for all $i \in I$.
Note that $\fkD$ is naturally identified with the Grothendieck group of $\fkD^{+}$.
By \cite[Proposition 1]{FR99}, for any object $V \in \scrC$, we have a decomposition
$$
V=\bigoplus_{\Psi \in \fkD} V_{\Psi},
$$
where $V_{\Psi}$ is the subspace of $V$ on which the coefficient of $z^{k}$ in the series
$\phi_{i}^{\pm}(z) - q_{i}^{\deg \Psi_{i}(z)}\left[\Psi_{i}(q_{i}^{-1}z)/\Psi_{i}(q_i z) \right]^{\pm}$
acts nilpotently for every $k \in \Z_{\ge 0}$.
Here we define $\deg (f(z)/g(z)) := \deg f(z) - \deg g(z)$ for $f(z) \in \kk[z], g(z) \in \kk[z]\setminus\{0\}$.
Then the $q$-character $\chi_{q}(V)$ of $V$ is defined to be
$$
\chi_{q}(V) := \sum_{\Psi \in \fkD} \dim_{\kk} (V_{\Psi}) [\Psi],
$$
which is an element of the group ring $\Z[\fkD] = \bigoplus_{\Psi \in \fkD} \Z [\Psi]$.
It is known that the assignment
$V \mapsto \chi_{q}(V)$ defines an injective ring homomorphism
$\chi_{q} \colon K(\scrC) \hookrightarrow \Z[\fkD]$,
where $K(\scrC)$ is the Grothendieck ring of $\scrC$~\cite[Corollary 2]{FR99}.

\ber
Finally, we recall the behaviour of the duality functors 
$\scrD^{\pm 1}$. 
Denote by $\bfD$ the automorphism of the group $\fkD$ given by
$\bfD(\pi_{i, a}) := \pi_{i^*, aq^{h^{\vee}}}$ for 
all $i\in I$ and $a \in \kk^{\times}$. 
It naturally extends to a ring automorphism of 
$\Z[\fkD]$.

\begin{proposition}[cf.~{\cite[Proposition 5.1(b)]{CP96a}, \cite[(A.5)]{AK97}, \cite[Corollary 6.10]{FM01}}] 
\label{proposition: dual}
For each $P \in \fkD^+$ and $k \in \Z$, we have
$$
\scrD^{k}(L(P)) \cong L(\bfD^k(P))
$$
as $U_{q}^{\prime}(\hg)$-modules. Moreover, 
for any $V \in \scrC$, we have
$$
\chi_{q}(\scrD^k (V)) = \bfD^{k} (\chi_q(V)).
$$
\end{proposition}
\er
\subsection{The Hernandez-Leclerc category $\scrC^{0}$}

Let us consider the Laurent polynomial ring
$$\calY = \Z[ Y_{i,p}^{\pm 1} \mid (i,p) \in \hI \ ],$$
where the set $\hI$ is defined as in Subsection~\ref{subsection: folding}.
We denote by $\calM^{+}$ (resp.~$\calM$) the set of all monomials
(resp.~Laurent monomials) in the variables $Y_{i,p}$.
Note that $\calM^{+}$ is a commutative monoid freely generated by the set $\{Y_{i,p} \mid (i,p) \in \hI\}$,
and $\calM$ is the Grothendieck group of $\calM^{+}$.
In what follows,
we regard $\calM$ as a subgroup of $\fkD$
via the injective homomorphism $\calM \hookrightarrow \fkD$
given by $Y_{i,p} \mapsto \pi_{i, \qs^{p}}$.

\begin{definition}[\cite{HL10,HL16,KKO19}]
The category $\scrC^{0}$ is defined to be the Serre subcategory of $\scrC$
such that $\Irr \scrC^{0} = \{ L(m) \mid m \in \calM^{+}\}$.
\end{definition}

\beb
\begin{remark}
The category $\scrC^{0}$ was introduced in \cite[Section 3.7]{HL10} and \cite{HL16} for  untwisted affine types, and was extended to other affine types in \cite{KKO19} by using the \emph{denominator formulas} between fundamental modules. The descriptions for $\scrC^{0}$ given in \cite{HL10,HL16} and \cite{KKO19} are different, but coincide for untwisted affine types
since the denominator formulas are \emph{meaningful} only for the pair of fundamental modules $L(Y_{i,p})$ and $L(Y_{j,s})$ when $(i,p),(j,s) \in \hI$ (see Section~\ref{section: denominator} below). 
\end{remark}
\eb
The subcategory $\scrC^{0}$ captures an essential part of the monoidal category $\scrC$
in the following sense.

\begin{theorem}[\ber cf.~{\cite[Section 3.7]{HL10}, \cite[Lemma 3.8]{HL16}}\er] \label{thm: scrC0}  \ber The followings hold. \er
\begin{enumerate}
\item \label{monoidal subcategory}
The full subcategory $\scrC^{0} \subset \scrC$ is closed under the tensor product $\otimes$ and duals $\scrD^{\pm 1}$.
Moreover, 
the $q$-character homomorphism $\chi_{q}$ restricts to the injective ring homomorphism
$$
\chi_{q} \colon K(\scrC^{0}) \hookrightarrow \calY = \Z[\calM].
$$
\item \label{tensor decomposition}
Any simple object $V \in \Irr \scrC$ decomposes as
$$
V \cong (V_{1})_{a_{1}} \otimes \cdots \otimes (V_{m})_{a_{m}}
$$
for some $m \in \Z_{\ge 1}$ and $\{(V_{k}, a_{k}) \in \Irr \scrC^{0} \times \kk^{\times} \mid 1 \le k \le m \}$
such that $a_{k}/a_{l} \not \in \qs^{2\Z}$ for $1 \le k < l \le m$.
Moreover, in this case, the modules $\{ (V_{k})_{a_{k}} \mid 1 \le k \le m\}$ pairwise commute.
\end{enumerate}
\end{theorem}

\emph{In the sequel, we shall focus on the category $\scrC^{0}$ rather than $\scrC$.}
\subsection{The category $\scrC_{\calQ}$} \label{Subsection: CQ}
Let us fix a Q-datum $\calQ$ for $\g$.
Recall the bijection $\bphi_{\calQ} \colon \hI \to \hsfR^{+}$ 
and its restriction $\bphi_{\calQ, 0} \colon \hI_{\calQ} \to \sfR^+$ 
defined in Subsection~\ref{subsection: folding}.
Following~\cite{HL15, OS19b, OhS19}, we set
$$
V_{\calQ}(\alpha) := L(Y_{i,p}), \qquad \text{if} \ \ \bphi_{\calQ, 0}^{-1}(\alpha) = (i,p) \in \hI_{\calQ}
$$
for each $\alpha \in \sfR^{+}$.
Then the category $\scrC_{\calQ}$ is defined to be the smallest monoidal Serre subcategory of $\scrC^{0}$
containing the collection of fundamental modules
$\{ V_{\calQ}(\alpha) \mid \alpha \in \sfR^{+}\} = \{ L(Y_{i,p} ) \mid (i,p) \in \hI_{\calQ}\}$.
It turns out that we have $\Irr \scrC_{\calQ} = \{ L(m) \mid m \in \calM^{+}_{\calQ}\}$,
where $\calM^{+}_{\calQ}$ is the set of all the monomials in the variables $Y_{i,p}$ with $(i,p) \in \hI_{\calQ}$
(see~\cite[Lemma 5.8]{HL15} and \cite[Lemma 3.26]{HO19}).

The category $\scrC_{\calQ}$ can be seen as a ``heart'' of the category $\scrC^{0}$
due to the following property.

\begin{proposition}
\label{proposition: CQ}
For any $(\alpha, k) \in \hsfR^{+}$, we have
$$
\scrD^{k}(V_{\calQ}(\alpha)) \cong L(Y_{i,p}), \qquad \text{where $(i,p) = \bphi_{\calQ}^{-1}(\alpha, k)$}.
$$
In particular, the set $\{ \scrD^{k}( V_{\calQ}(\alpha)) \mid (\alpha, k) \in \hsfR^{+}\}$ forms a complete and irredundant collection
of fundamental modules in $\scrC^{0}$ up to isomorphisms.
\end{proposition}
\begin{proof}
It follows from Proposition~\ref{proposition: shift hDs}.
\end{proof}

\subsection{$\calQ$-weights and block decomposition}
\label{sec: Q-weights}
For a Laurent monomial $m \in \calM$, we write
\begin{equation} \label{eq: um}
m = \prod_{(i,p) \in \hI} Y_{i,p}^{u_{i,p}(m)}
\end{equation}
with $u_{i,p}(m) \in \Z$.
We define the map $h \colon \hsfR^+ \to \sfQ$ by 
$h(\alpha, k) := (-1)^k \alpha$.
\begin{definition} [cf.~\cite{KKOP20b}] \label{def: new-weights} 
Let $\calQ$ be a Q-datum for $\g$ and 
$m \in \calM$ be a Laurent monomial.
With the above notation,
 we define \emph{the $\calQ$-weight of $m$} by
$$
\wt_\calQ(m)  = \sum_{(i,p) \in \hI} 
u_{i,p}(m)h(\bphi_{\calQ}(i,p)) \in \sfQ.
$$
\end{definition}
Note that the $\calQ$-weights belong to the root lattice $\sfQ$ of the simply-laced Lie algebra $\sfg$ rather than that of $\g$.
The assignment $m \mapsto \wt_{\calQ}(m)$ defines the group homomorphism $\wt_\calQ \colon \calM \to \sfQ$
and equips the Laurent polynomial ring $\calY = \Z[\calM]$ with a $\sfQ$-grading. 

We shall see that the simple $q$-characters are homogeneous 
with respect to these $\sfQ$-gradings.
Let $\hI_{(+D)} := \{ (i,p) \in I \times \Z \mid (i,p-d_i) \in \hI \}$.
For each $(i,p) \in \hI_{(+D)}$, 
following~\cite{FR99},
we define 
the Laurent monomial $A_{i,p} \in \calM$  by
\begin{equation}
\label{eq: def A}
A_{i,p} := Y_{i, p-d_{i}} Y_{i, p+d_{i}} \left(\prod_{j: c_{ji}=-1} Y_{j,p}^{-1}\right)
\left(\prod_{j: c_{ji}=-2} Y_{j, p-1}^{-1} Y_{j, p+1}^{-1} \right)
\left( \prod_{j: c_{ji}=-3} Y_{j, p-2}^{-1} Y_{j, p}^{-1} Y_{j, p+2}^{-1} \right).
\end{equation}
By the $\g$-additive property in~\eqref{eq: twisted
additivity} and the property~\eqref{eq: bphi}, one can easily check
\begin{equation} \label{rmk: wt0}
\wt_\calQ(A_{i,p}) =0
\end{equation}
for any $(i,p) \in \hI_{(+D)}$ and a Q-datum $\calQ$ for $\g$.
Recall the following important result. 
\begin{theorem}[\cite{FM01}]\label{theorem: FM}
For each dominant monomial $m \in \calM^+$, the $q$-character of the 
corresponding simple module $L(m) \in \scrC^0$ is of the form
$$\chi_{q}(L(m)) = m \left(1 + \sum_{k} M_k\right)$$
where each $M_k$ is a monomial in the variables 
$A_{i,p}^{-1}$ with $(i,p) \in \hI_{(+D)}$.
\end{theorem}

By \eqref{rmk: wt0} and Theorem~\ref{theorem: FM},
we immediately obtain the following. 

\begin{proposition} \label{Prop:Qwt_simple}
For any dominant monomials $m \in \calM^+$ and a monomial $m'$ occurring in $\chi_{q}(L(m))$ and any ${\rm Q}$-datum $\calQ$ for $\g$, we have
$$\wt_\calQ(m)  =  \wt_\calQ(m').$$
In other words, the $q$-character $\chi_q(L(m))$ 
is a homogeneous element of $\calY$ 
with respect to the $\sfQ$-grading given by the $\calQ$-weight.
\end{proposition}

\begin{definition}
Let $\calQ$ be a Q-datum for $\g$.
For a simple module $V$ in $\scrC^0$, we set
$$  \wt_\calQ(V) \seteq   \wt_\calQ(m)   \quad \text{ for any $m \in \calM$ occurring in $\chi_q(V)$.}$$
This is well-defined thanks to Proposition~\ref{Prop:Qwt_simple}.
\end{definition}

Let $\calM_{A} \subset \calM$ denote the subgroup
generated by the elements $A_{i,p}$ with $(i,p) \in \hI_{(+D)}$
and consider the quotient group $\calM/\calM_{A}$,
which was initially introduced by~\cite{CM05}.

\begin{proposition}
\label{proposition: elliptic}
For each {\rm Q}-datum $\calQ$ for $\g$, the homomorphism
$\wt_{\calQ} \colon \calM \to \sfQ$ induces an isomorphism
of abelian groups 
\begin{equation} \label{eq: CMisom}
\calM/\calM_{A} \simeq \sfQ.
\end{equation}
\end{proposition}
\begin{proof}
By Remark~\ref{rmk: wt0}, the map $\wt_{\calQ}$ induces a surjective homomorphism
$\wt_{\calQ} \colon \calM/\calM_{A} \twoheadrightarrow \sfQ$.
Note that
the quotient group $\calM/\calM_{A}$ is generated by
the classes of the elements $\{Y_{\bar{\im},\uxi_{\im}} \mid \im \in \Delta_{0}\}$.
On the other hand, we know that their images
$\{\gamma^{\calQ}_{\im} = \wt_{\calQ}(Y_{\bar{\im}, \uxi_{\im}}) \mid \im \in \Delta_{0}\}$
form a free basis of $\sfQ$ by Lemma~\ref{lemma: gamma}.
Therefore $\wt_{\calQ}$ gives an isomorphism $\calM/\calM_{A} \simeq \sfQ$.
\end{proof}

Under the isomorphism~\eqref{eq: CMisom} $\wt_{\calQ} \colon \calM/\calM_{A} \simeq \sfQ$,
the element $\wt_{\calQ}(V)$ is identical to the \emph{elliptic character} of the simple module $V$
in the sense of~\cite{CM05}.
Therefore, we can re-express the main result of \cite{CM05} 
in terms of the $\calQ$-weight as follows. 

\begin{theorem}[{\cite[Theorem 8.3]{CM05}}] \label{Thm: CM}
Fix a Q-datum $\calQ$ for $\g$.
For each $\alpha \in \sfQ$, 
let $\scrC^0_{\alpha}$ denote the Serre subcategory of $\scrC^0$ generated by all the simple modules $V$ satisfying $\wt_\calQ(V) = \alpha$.
Then we have a block decomposition
$$
\scrC^0 = \bigoplus_{\alpha \in \sfQ} \scrC^0_{\alpha}.
$$
In particular, for any indecomposable module $V \in \scrC^0$,
its $q$-character $\chi_{q}(V) \in \calY$ is homogeneous with respect to $\wt_{\calQ}$
and hence $\wt_\calQ(V)$ is well-defined.

\end{theorem}

\begin{remark}
The block decomposition by the $\calQ$-weight in Theorem~\ref{Thm: CM} turns out to be the same as 
the block decomposition recently obtained in~\cite{KKOP20b}, which investigates all quantum affine algebras.
See Remark~\ref{remark: KKOP20b} below.
In this sense, our description connects the result of \cite{CM05}
with that of \cite{KKOP20b}.
\end{remark}

\begin{lemma} \label{lemma: dual wt}
Let $\calQ$ be a Q-datum for $\g$. For any indecomposable module $V \in \scrC^0$ and $k \in \Z$, we have
$$
\wt_{\calQ}(\scrD^k V) = (-1)^k \wt_\calQ(V). 
$$
\end{lemma}
\begin{proof}
It follows from Proposition~\ref{proposition: dual} and Proposition~\ref{proposition: shift hDs}.
\end{proof}

\subsection{The skew-symmetric pairing $\scrN$}
In this subsection, we recall a $\Z$-valued skew-symmetric pairing $\scrN$
on the group $\calM$ introduced by Hernandez~\cite{Hernandez04} 
in his construction of the quantum Grothendieck ring
(see Remark~\ref{remark: qt} below).
Here we relate it with the $\calQ$-weight defined in the last subsection.

\begin{definition}[\cite{Hernandez04}]
Recall the notation~\eqref{eq: um}.
We define the group bi-homomorphism 
$\scrN \colon \calM \times \calM \to \Z$ by setting
$$
\scrN(m; m') := 
\sum_{(i,p), (j,s) \in \hI} u_{i,p}(m) u_{j,s}(m') \scrN(i,p;j,s), 
$$
for any $m,m' \in \calM$, where
\begin{equation}
\label{eq: def scrN}
\scrN(i,p;j,s) := \tc_{ij}(p-s-d_{i})- \tc_{ij}(p-s+d_{i}) - \tc_{ij}(s-p-d_{i})+ \tc_{ij}(s-p+d_{i}).
\end{equation}
For simple modules $V, W$ in $\scrC^0$, 
we set 
$$\scrN(V, W) := \scrN(m;m') \qquad 
\text{if $V \cong L(m), W \cong L(m')$ with $m,m' \in \calM^+$.}
$$
\end{definition}

\begin{remark}
For any $(i,p), (j,s) \in \hI$, we have
\begin{equation} \label{eq: scrN skew}
    \scrN(i, p; j, s) = -\scrN(j,s;i,p)
\end{equation}
by Lemma~\ref{lemma: tc1}~(\ref{tc1_symm}).
Therefore the pairing $\scrN$ is skew-symmetric. 
Moreover, we have
\begin{equation}
\label{eq: scrN half}
\scrN(i,p;j,s) = \tc_{ij}(p-s-d_{i})- \tc_{ij}(p-s+d_{i}) \qquad 
\text{if \ber $p-s \ge \delta_{i,j}$ \er},
\end{equation}
by
\ber
Lemma~\ref{lemma: small vanish}.
\er
\end{remark}

\begin{remark} \label{remark: qt}
Let $t$ be an indeterminate with a formal square root $t^{1/2}$.
The skew-symmetric pairing $\scrN$ was introduced by Hernandez~\cite{Hernandez04}
to define the \emph{quantum Grothendieck ring} $K_t(\scrC^0)$,
which is a $t$-deformation of the Grothendieck ring $K(\scrC^0)$.
The ring 
$K_t(\scrC^0)$
is constructed as a $\Z[t^{\pm 1}]$-subalgebra of the quantum torus $(\calY_t, *)$
that is
$\calY_t := \Z[t^{\pm 1}][\calM] \cong \Z[t^{\pm 1/2}] \otimes_{\Z} \calY$ 
with the multiplication $*$ given by
$$
m * m' = t^{\scrN(m;m')}m' * m := t^{\scrN(m;m')/2} mm' 
$$
for $m,m' \in \calM$. 
It satisfies that $\evt(K_t(\scrC^0)) = \chi_q(K(\scrC^0))$,
where $\evt \colon \calY_{t} \to \calY$ is the evaluation map at $t=1$.

For each simple module $L(m) \in \scrC^0$ with $m \in \calM^+$,
Hernandez~\cite{Hernandez04} further constructed its \emph{$(q,t)$-character}
$[L(m)]_t$ as an element of $K_t(\scrC^0)$, which contains $m$ as the leading term, 
by means of an analog of Kazhdan-Lusztig algorithm.
It was conjectured that $\evt([V]_t) = \chi_q(V)$ holds for any simple module $V$ in $\scrC^0$
(\cite[Conjecture 7.3]{Hernandez04}).
Note that we have 
$$ [V]_t * [W]_t = t^{\scrN(V, W)} [W]_t * [V]_t$$
in $K_t(\scrC^0)$ if $[V]_t$ and $[W]_t$ commute up to a power of 
$t^{\pm 1/2}$.
\end{remark}

\begin{proposition} \label{proposition: gamma wt pairing same}
Let $(i,p), (j,s) \in \hI$.
If $\bphi_{\calQ}(i,p)=(\alpha, k)$ and $\bphi_{\calQ}(j, s)=(\beta, l)$, we have
$$  \scrN(i,p; j,s) = (-1)^{\delta(p\ge s) + k+l} \delta( (\alpha,k) \ne (\beta,l) ) (\alpha, \beta).$$
\end{proposition}

\begin{proof}
\ber Since the both sides of the desired equality are skew-symmetric,
we only have to consider the case $(i,p)\neq (j,s)$
and $p \ge s$. \er
Take $\im \in i$ and $\jm \in j$.
By Theorem~\ref{theorem: combinatorial formula} and (\ref{eq: scrN half}), we have
\begin{align*}
\scrN(i,p;j,s)
& =  -(\varpi_{\im},
\tau_{\calQ}^{(p-s+\uxi_{\jm}-\uxi_{\im})/2 }(\gamma^{\calQ}_{\jm})
- \tau_{\calQ}^{(p-s+\uxi_{\jm}-\uxi_{\im})/2 -d_{i}}(\gamma^{\calQ}_{\jm})) \allowdisplaybreaks\\
& =-  (\tau_{\calQ}^{(\uxi_{\im} - p)/2} (1-\tau_{\calQ}^{d_{i}})\varpi_{\im}, \tau_{\calQ}^{(\uxi_{\jm}-s)/2}(\gamma^{\calQ}_{\jm}))\allowdisplaybreaks\\
& =  - (\tau_{\calQ}^{(\uxi_{\im} - p)/2}(\gamma^{\calQ}_{\im}), \tau_{\calQ}^{(\uxi_{\jm}-s)/2}(\gamma^{\calQ}_{\jm}))\allowdisplaybreaks\\
& =  (-1)^{1+ k+l}(\alpha, \beta),
\end{align*}
where the last equality follows from (\ref{eq: bphi}).
This completes the proof.
\end{proof}

\begin{corollary} \label{corollary: scrN wt}
Let $m, m' \in \calM$ such that
$$
\min\{p \mid \exists i \in I, u_{i,p}(m) \neq 0\}
>
\max\{p \mid \exists i \in I, u_{i,p}(m') \neq 0\}.
$$
Then, for any Q-datum $\calQ$ for $\g$, we have
$$
\scrN(m;m') = -(\wt_\calQ(m), \wt_\calQ(m')).
$$
\end{corollary}

\section{R-matrices and related invariants}
\label{section: denominator}

In this section, we present a conjectural unified formulae of the denominators
of the normalized R-matrices between all the Kirillov-Reshetikhin modules.
We also compute the $\Lambda$-invariants introduced by \cite{KKOP20} in terms of the $\calQ$-weights
and the skew-symmetric pairing $\scrN$
discussed in Section~\ref{section: QAffine}.

\begin{convention} Throughout this section, we keep the following convention:
\begin{enumerate}
\item
For $f(z), g(z) \in \kk (\!( z )\!)$, we write
$f(z) \equiv g(z)$ if $f(z) / g(z) \in \kk[z^{\pm 1}]^{\times}$.
\item For $f(z) \in \kk(z)$ and $a \in \kk$,
we denote by $\zero_{z=a} (f(z))$ the order of zero of $f(z)$ at $z=a$.
\end{enumerate}
\end{convention}

\subsection{R-matrices}
In this subsection, we recall the notion of R-matrices between $U_{q}^{\prime}(\hg)$-modules
together with their denominators and universal coefficients
following~\cite[Section 8]{Kashiwara02} and \cite[Appendix A]{AK97}.
Choose a basis $\{ E_{\nu} \}_{\nu}$ of $U_{q}^{+}(\hg)$ and
a basis $\{ F_{\nu}\}_{\nu}$ of $U_{q}^{-}(\hg)$ 
\ber (with $\nu$ running over a certain index set) \er
which are dual to each other
with respect to a suitable pairing between $U_{q}^{+}(\hg)$ and $U_{q}^{-}(\hg)$.
For $U_{q}(L\g)$-modules $V$ and $W$ of type $\bf{1}$ ,
the \emph{universal R-matrix} defines a $U_{q}(L\g)$-linear homomorphism
$V \otimes W \to W \otimes V$
by
$$
R_{V, W}^{\univ}(v \otimes w) = q^{\ber \langle\lambda, \mu\rangle \er} \sum_{\nu} E_{\nu}w \otimes F_{\nu} v
$$
provided that the infinite sum has a meaning,
where $v \in V$, $w \in W$ and $\lambda = (\lambda_{i})_{i \in I}, \mu=(\mu_{i})_{i \in I} \in \Z^{I}$
are such that $K_{i}v = q_{i}^{\lambda_{i}}v, K_{i} w = q_{i}^{\mu_{i}} w$ for $i\in I$.
Here we set $\ber \langle\lambda, \mu\rangle \er := \sum_{i,j \in I} d_{i}\tc_{ij}\lambda_{i} \mu_{j}$.
Note that we have $\ber \langle\lambda, \mu\rangle = \langle\mu, \lambda\rangle\er$ by (\ref{eq: DtC symmetric}).

Recall the affinization $W_z$ of an object $W \in \scrC$
defined in Subsection~\ref{ssec: qaffine}.
For any objects $V, W \in \scrC$, it is known that $R^{\univ}_{V, W_{z}}$
converges in the $z$-adic topology and
induces a $U_{q}(L\g)\otimes \kk(\!(z)\!)$-linear isomorphism
\begin{equation*}
R^{\univ}_{V, W_{z}} \colon (V \otimes W_{z})\otimes_{\kk[z^{\pm 1}]} \kk(\!( z )\!)
\xrightarrow{\sim} (W_{z} \otimes V)\otimes_{\kk[z^{\pm 1}]} \kk(\!( z )\!).
\end{equation*}
Moreover, if $V, W$ are simple modules 
with $\ell$-highest weight vectors $v \in V$, $w \in W$,
there exists a unique element $a_{V, W}(z) \in \kk \lf z \rf^{\times}$
such that
$$
R^{\univ}_{V, W_{z}} (v \otimes w_{z}) = a_{V, W}(z) (w_{z} \otimes v),
$$
where $w_{z} := w \otimes 1 \in W_{z}$.
Then $R^{\norm}_{V, W_{z}} := a_{V, W}(z)^{-1} R^{\univ}_{V, W_{z}} |_{(V \otimes W_{z})\otimes_{\kk[z^{\pm 1}]} \kk(z)}$
induces a unique $U_{q}(L\g) \otimes \kk(z)$-linear isomorphism
$$
R^{\norm}_{V, W_{z}} \colon (V \otimes W_{z}) \otimes_{\kk[z^{\pm 1}]}\kk(z)
\xrightarrow{\sim} (W_{z} \otimes V) \otimes_{\kk[z^{\pm 1}]} \kk(z)
$$
satisfying
$
R^{\norm}_{V, W_{z}} (v \otimes w_{z}) = w_{z} \otimes v.
$
We call $a_{V, W}(z)$ 
(resp.~$R^{\norm}_{V, W_{z}}$)
the \emph{universal coefficient} 
(resp.~the \emph{normalized R-matrix})
between $V$ and $W$.

Let $d_{V, W}(z) \in \kk[z]$ 
be a monic polynomial of the smallest degree
such that the image of $d_{V, W}(z) R^{\norm}_{V, W_{z}}$ is contained in $W_{z} \otimes V$.
We call $d_{V,W}(z)$ the denominator of $R^{\norm}_{V, W_{z}}$.

The singularities of the normalized R-matrices
strongly reflect the structure of the tensor product modules.
For example, we have the following proposition.
Here we say that a simple module $V$ in $\scrC$ is \emph{real}
if its tensor square $V \otimes V$ is again simple.
For instance, every Kirillov-Reshetikhin module is
known to be real (see e.g.~\cite[proof of Theorem 4.11]{FH15}).
\begin{proposition}[{\cite[Section 3.2]{KKKO15}}]
\label{proposition: simplicity}
Let $V, W \in \Irr \scrC$ such that at least one of them is real.
Then the following three conditions are mutually equivalent:
\begin{itemize}
\item The tensor product $V \otimes W$ is a simple $U_{q}^{\prime}(\hg)$-module.
\item The objects $V$ and $W$ mutually commute.
\item We have $d_{V,W}(1) \cdot d_{W,V}(1) \neq 0$.
\end{itemize}
\end{proposition}

The orders of zeros of the denominators $d_{V,W}(z)$ also play an important role
especially in the theory of monoidal categorification of cluster algebras
and the construction of the generalized quantum affine Schur-Weyl duality functors.
See~\cite{KKOP20} and \cite{KKK18} for more details.

The following properties are well-known.
\begin{lemma}[cf.~{\cite[Appendix A]{AK97}}]
\label{lemma: property ad}
Let $V, W$ be two simple modules in $\scrC$ and $a,b \in \kk^{\times}$. We have
\begin{enumerate}
\item \label{shift ad}
$a_{V_{a}, W_{b}}(z) = a_{V, W}((b/a)z)$, and $
d_{V_{a}, W_{b}}(z) \equiv d_{V,W}((b/a) z)$,
\item \label{dual ad}
$a_{V, W}(z) = a_{{}^{*}V, {}^{*}W}(z) = a_{V^{*}, W^{*}}(z)$ and
$d_{V, W}(z) = d_{{}^{*}V, {}^{*}W}(z) = d_{V^{*}, W^{*}}(z)$.
\end{enumerate}
\end{lemma}

\subsection{A conjectural unified KR denominator formula}
\label{subsection: KR denominator}

In this subsection, we \ber give \er a conjectural unified formula
expressing the denominators of the normalized R-matrices
between all the KR modules.

For each $i \in I$ and $l \in \Z_{\ge 1}$, we write
$$
V^{(i)}_{l} := L\left(\pi^{(i)}_{l, q_{i}^{-l+1}}\right), \qquad
\text{where} \ \pi^{(i)}_{l, q_{i}^{-l+1}} = \pi_{i, q_{i}^{-l+1}} \pi_{i, q_{i}^{-l+3}} \cdots \pi_{i, q_{i}^{l-3}} \pi_{i, q_{i}^{l-1}}.
$$
Note that every KR module can be obtained as $(V^{(i)}_{l})_{a}$ for a suitable $a \in \kk^{\times}$.
When $l=1$, it gives a fundamental module $V^{(i)}_{1} = L(\pi_{i,1})$.
In view of~\ref{lemma: property ad}~(\ref{shift ad}), it is enough to consider the denominators
\begin{equation}
\label{eq: }
d_{i^{l}, j^{m}}(z) := d_{V^{(i)}_{l}, V^{(j)}_{m}}(z)
\end{equation}
for $i,j \in I$ and $l, m \in \Z_{\ge 1}$.
We also write $d_{i,j^{m}}(z) := d_{i^{1}, j^{m}}(z)$ and
$d_{i^{l}, j}(z) := d_{i^{l}, j^{1}}(z)$ for simplicity.

These denominators $d_{i, j}(z)$ have been computed
by \cite{DO94, AK97, KKK15, Oh15R, OhS19, Fujita19} and
 $d_{i^{l}, j^{m}}(z)$  have been computed in many cases by \cite{OhS19s}.

In Appendix~\ref{section: denominator formula}, we give a list of all the formulae of $d_{i^{l}, j^{m}}(z)$
which are currently known.

\begin{remark}
\label{remark: convention}
We have some remarks on the convention.
Our fundamental module $V^{(i)}_{1}$
is slightly different from the fundamental module $V(\varpi_{i})$
in the references e.g.,~\cite{AK97, Kashiwara02, KKOP20},
which possesses a global basis.
Indeed, it was shown by Nakajima~\cite[Section 3.1]{Nakajima04e} that we have
\begin{equation}
\label{eq: convention fund}
V(\varpi_i) \cong L(\pi_{i, a_{i}}) \cong (V^{(i)}_{1})_{a_{i}},
\qquad \text{with $a_{i} := -o(i)(-1)^{h}q^{-h^{\vee}}$}
\end{equation}
where $h$ is the Coxeter number of $\g$
and $o \colon I \to \{ \pm 1\}$ is as in Remark~\ref{remark: Beck}.

Similarly, our KR modules $V^{(i)}_{l}$ is slightly different from
the KR modules $V(i^{l})$ appearing in~\cite{OhS19s}.
Indeed, we have
\begin{equation}
\label{eq: convention KR}
V(i^{l}) = L\left(\pi^{(i)}_{l, (-q_{i})^{-l+1}a_{i}} \right) = (V^{(i)}_{l})_{(-1)^{l-1} a_{i}}.
\end{equation}

For $i, j \in I$, we have $a_{i}/a_{j} = o(i)/o(j) = (-1)^{d(i,j)}$, where $d(i,j)$ denotes
the distance between $i$ and $j$ in the Dynkin diagram of $\g$.
Thus, Lemma~\ref{lemma: property ad}~ together with (\ref{eq: convention KR}) implies
\begin{equation}
\label{eq: convention d}
d_{i^{l}, j^{m}}(z) \equiv d_{V(i^{l}), V(j^{m})}\left((-1)^{d(i,j) + l +m}z\right)
\end{equation}
for each $i,j \in I$ and $l, m \in \Z_{\ge 1}$.
One should notice that the same symbol $d_{i^{l}, j^{m}}(z)$ is used in~\cite{OhS19s}
for denoting $d_{V(i^{l}), V(j^{m})}(z)$, which is different from our convention.
The case-by-case denominator formulae listed in Appendix~\ref{section: denominator formula} have been rewritten
from the original formulae of $d_{V(i^{l}), V(j^{m})}(z)$
by using~(\ref{eq: convention d}).
They turn out to be obtained simply by forgetting signs appearing in the original ones.
\end{remark}

\begin{lemma}
\label{lemma: symmetry d}
For any $i,j \in I$, we have
$$d_{i,j}(z) =d_{j,i}(z) = d_{i^{*},j^{*}}(z).$$
\end{lemma}
\begin{proof}
By~\cite[(A.7)]{AK97}, we know $d_{V(\varpi_{i}), V(\varpi_{j})}(z) = d_{V(\varpi_{j}), V(\varpi_{i})}(z)$.
Combining with Lemma~\ref{lemma: property ad}~(\ref{shift ad}),
we get $d_{i,j}((a_{j}/a_{i})z) = d_{j,i}((a_{i}/a_{j})z)$.
Since $a_{j}/a_{i} = o(i)o(j) = a_{i}/a_{j}$, we obtain the equality $d_{i,j}(z) =d_{j,i}(z)$.
The other equality $d_{i,j}(z) = d_{i^{*},j^{*}}(z)$ follows from Lemma~\ref{lemma: property ad}~(\ref{dual ad}) and
Proposition~\ref{proposition: dual}.
\end{proof}

Note that the case-by-case formulae of $d_{i^{l}, j^{m}}(z)$ listed in Appendix~\ref{section: denominator formula}
are expressed in a symmetric way under the exchange $l d_{i} \leftrightarrow m d_{j}$.
In what follows,
we obtain more unified formulae expressed in terms of the integers $\tc_{ij}(u)$, which however break such a symmetry.
First, we observe the following. \ber Recall the notation $q_s := q^{1/r}$ and $r_i = r/d_i$. \er

\begin{proposition}
\label{proposition: generalization}
\beb For $i,j \in I$, we have \eb
\begin{equation}
\label{eq: generalization}
d_{i^{r_{i}}, j}(z) = \prod_{u=0}^{rh^{\vee}}(z-\qs^{u+r})^{\tc_{ij}(u)}.
\end{equation}
\end{proposition}

\begin{proof}
When $\g$ is simply-laced, this is \cite[Theorem 2.10]{Fujita19} (see~(\ref{eq: formulaADEfund})).
When $\g$ is not simply-laced, we can check the assertion directly by
comparing the known explicit formulae of $d_{i^{r_{i}}, j}(z)$ as in Appendix~\ref{section: denominator formula}
with the explicit values of $\tC(z)$ computed in Subsection~\ref{subsection: computation}.
\end{proof}

\begin{remark} \label{rmk: F4}
\beb
After the authors first uploaded this paper to arXiv, the second named author and
T.~Scrimshaw computed all the denominators $d_{i^{r_i},j}(z)$ of type $\mathsf{F}_4$, and proved that Proposition~\ref{proposition: generalization} holds also for $\mathsf{F}_4$. The result will be dealt with in the near future (see Appendix~\ref{section: denominator formula} below). 
\eb
\end{remark}

Now we propose the following more general formula.

\begin{conjecture}[A unified KR denominator formula]
\label{conjecture: unified}
Let $i, j \in I$ and $l, m \in \Z_{\ge 1}$ such that $l d_{i} \ge m d_{j}$. Then we have
\begin{align}
\label{eq: unified}
d_{i^{l}, j^{m}}(z) & = d_{j^{m}, i^{l}}(z) =
\prod_{k=0}^{m-1} \prod_{u=0}^{r h^{\vee}} \left(z - \qs^{u + l d_{i} + (2k - m +1 ) d_{j}}  \right)^{\tc_{ij}(u)}
\end{align}
with the following three exceptions:
\begin{itemize}
\item[($\mathtt{EX1}$)] $\g$ is of type $\mathsf{C}_{n}$, $d_{i} = d_{j} = 1$ and $l = m= 1$,
\item[($\mathtt{EX2}$)] $\g$ is of type $\mathsf{F}_{4}$, $d_{i}=d_{j}=1$ and $l = m =1$,
\item[($\mathtt{EX3}$)] $\g$ is of type $\mathsf{G}_{2}$, $d_{i}=d_{j}=1$ and $l = m \in \{ 1,2 \}$.
\end{itemize}
\end{conjecture}

\ber
\begin{remark}
Note that the three exceptions
($\mathtt{EX1}$), ($\mathtt{EX2}$), ($\mathtt{EX3}$)
occur only if $ld_i = m d_j < r$.
Even in these exceptional cases, 
the denominator $d_{i^{l}, j^{m}}(z) (= d_{j^{m}, i^{l}}(z))$ always divides the RHS of (\ref{eq: unified}). 
\end{remark}
\er

As an evidence for Conjecture~\ref{conjecture: unified}, 
we give a partial result here.

\begin{theorem}
\label{theorem: unified}
Conjecture~\ref{conjecture: unified} \ber holds \er 
when $\g$ is neither of type $\mathsf{E}_{6,7,8}$ nor of type $\mathsf{F}_{4}$.
\end{theorem}

\begin{proof}
Here we give a case-by-case proof based on Proposition~\ref{proposition: generalization} and
the results of~\cite{OhS19s}.
Since the equality $d_{i^{l}, j^{m}}(z) = d_{j^{m}, i^{l}}(z)$ is already understood
by~\cite[Lemma 2.2]{OhS19s} (see also (\ref{eq: convention d}) above),
we only have to prove that $d_{i^{l}, j^{m}}(z)$
is equal to the RHS of~(\ref{eq: unified}).
We shall provide a detailed proof only for type~$\mathsf{C}_{n}$,
since the other cases are quite similar or easier.

Now we suppose $\g$ is of type $\mathsf{C}_{n}$.
We use the labeling $I=\{1, \ldots, n \}$ as in~(\ref{eq: C_n}).
In the case $1 \le i,j < n$, we have $d_{i}=d_{j}=1$ and $r_{i}=2$.
The condition $l d_{i} \ge m d_{j}$ implies $l \ge m$.
When $(l, m) = (2,1)$, we have
\begin{equation}
\label{eq: formulaCi2j}
d_{i^{2}, j}(z) = \prod_{u=0}^{2h^{\vee}} (z-\qs^{u+2})^{\tc_{ij}(u)}
= \prod_{u=1}^{\min(i,j)} (z-\qs^{|i-j|+1+2u})(z-\qs^{2n+3-i-j+2u})
\end{equation}
by Proposition~\ref{proposition: generalization} and (\ref{eq: formulaCss}).
For general $l, m$ with $l \ge m$ and $l \ge 2$, we have
\begin{align*}
d_{i^{l}, j^{m}}(z) &= \prod_{k=0}^{m-1} \prod_{u=1}^{\min(i,j)} (z-\qs^{|i-j|+(l-m)+2(u+k)})(z-\qs^{2n+2-i-j+(l-m)+2(k+u)})
&& (\text{by (\ref{eq: formulaCss})})\allowdisplaybreaks\\
&= \prod_{k=0}^{m-1} \prod_{u=1}^{\min(i,j)} (z-\qs^{(|i-j|+1+2u) -1+ l +2k-m)})(z-\qs^{(2n+3-i-j+2u)-1+l+2k-m})
&& \allowdisplaybreaks\\
&= \prod_{k=0}^{m-1}  \prod_{u=0}^{2h^{\vee}} (z-\qs^{(u+2)-1+l +2k-m})^{\tc_{ij}(u)}
&& (\text{by (\ref{eq: formulaCi2j})}) \allowdisplaybreaks\\
&= \prod_{k=0}^{m-1}  \prod_{u=0}^{2h^{\vee}} (z-\qs^{u+l +(2k-m+1)})^{\tc_{ij}(u)},
\end{align*}
as desired.

In the case $i<j=n$, we have $d_{i}=1$ and $r_{i}=d_{j}=2$.
The condition $l d_{i} \ge m d_{j}$ implies $l \ge 2m$.
When $(l, m) = (2,1)$, we have
\begin{equation}
\label{eq: formulaCi2n}
d_{i^{2}, n}(z) = \prod_{u=0}^{2h^{\vee}} (z-\qs^{u+2})^{\tc_{in}(u)}
= \prod_{k=0}^{1} \prod_{u=1}^{i} (z-\qs^{n+1-i+2k+2u})
\end{equation}
by Proposition~\ref{proposition: generalization} and (\ref{eq: formulaCsl}).
For general $l, m$ with $l \ge 2m$, we have
\begin{align*}
d_{i^{l}, n^{m}}(z) &= \prod_{k=0}^{2m-1} \prod_{u=1}^{i} (z-\qs^{n+1-i+(l-2m)+2k+2u})
&& (\text{by (\ref{eq: formulaCsl})})\allowdisplaybreaks\\
&= \prod_{k=0}^{m-1}\prod_{k^{\prime}=0}^{1} \prod_{u=1}^{i} (z-\qs^{(n+1-i+2k^{\prime}+2u)+l+4k-2m})
&& (\text{replace $k$ with $2k+k^{\prime}$})\allowdisplaybreaks\\
&= \prod_{k=0}^{m-1}  \prod_{u=0}^{2h^{\vee}} (z-\qs^{(u+2)+l +4k-2m})^{\tc_{nn}(u)}
&& (\text{by (\ref{eq: formulaCi2n})}) \allowdisplaybreaks\\
&= \prod_{k=0}^{m-1}  \prod_{u=0}^{2h^{\vee}} (z-\qs^{u+l +2(2k-m+1)})^{\tc_{nn}(u)},
\end{align*}
as desired.

In the case $i=n>j$, we have $d_{i}=2$ and $r_{i}=d_{j}=1$.
The condition $l d_{i} \ge m d_{j}$ implies $2l \ge m$.
When $(l, m) = (1,1)$, we have
\begin{equation}
\label{eq: formulaCnj}
d_{n, j}(z) = \prod_{u=0}^{2h^{\vee}} (z-\qs^{u+2})^{\tc_{nj}(u)}
= \prod_{u=1}^{j} (z-\qs^{n+2-j+2u})
\end{equation}
by Proposition~\ref{proposition: generalization} and (\ref{eq: formulaCsl}).
For general $l, m$ with $2l \ge m$, we have
\begin{align*}
d_{n^{l}, j^{m}}(z) &= \prod_{k=0}^{m-1} \prod_{u=1}^{j} (z-\qs^{n+1-j+(2l-m)+2k+2u})
&& (\text{by (\ref{eq: formulaCsl})}) \allowdisplaybreaks\\
&= \prod_{k=0}^{m-1}\prod_{u=1}^{j} (z-\qs^{(n+2-j+2u)-1+2l+2k-m})
&& \allowdisplaybreaks\\
&= \prod_{k=0}^{m-1}  \prod_{u=0}^{2h^{\vee}} (z-\qs^{(u+2)-1+2l +2k-m})^{\tc_{nj}(u)}
&& (\text{by (\ref{eq: formulaCnj})}) \allowdisplaybreaks\\
&= \prod_{k=0}^{m-1}  \prod_{u=0}^{2h^{\vee}} (z-\qs^{u+2l +(2k-m+1)})^{\tc_{nj}(u)},
\end{align*}
as desired.

In the case $i=j=n$, we have $d_{i}=d_{j}=2$ and $r_{i}=1$.
The condition $l d_{i} \ge m d_{j}$ implies $l \ge m$.
When $(l, m) = (1,1)$, we have
\begin{equation}
\label{eq: formulaCnn}
d_{n, n}(z) = \prod_{u=0}^{2h^{\vee}} (z-\qs^{u+2})^{\tc_{nn}(u)}
= \prod_{u=1}^{n} (z-\qs^{2+2u})
\end{equation}
by Proposition~\ref{proposition: generalization} and (\ref{eq: formulaCll}).
For general $l, m$ with $l \ge m$, we have
\begin{align*}
d_{n^{l}, n^{m}}(z) &= \prod_{k=0}^{m-1} \prod_{u=1}^{n} (z-\qs^{(2+2u)+(2l-2m)+4k})
&& (\text{by (\ref{eq: formulaCll})}) \allowdisplaybreaks\\
&= \prod_{k=0}^{m-1}  \prod_{u=0}^{2h^{\vee}} (z-\qs^{(u+2)+2l +4k-2m})^{\tc_{nj}(u)}
&& (\text{by (\ref{eq: formulaCnn})})\allowdisplaybreaks \\
&= \prod_{k=0}^{m-1}  \prod_{u=0}^{2h^{\vee}} (z-\qs^{u+2l +2(2k-m+1)})^{\tc_{nj}(u)},
\end{align*}
as desired.
\end{proof}

\subsection{A universal coefficient formula}
\label{subsection: univ_coeff}

In this subsection, we briefly recall the formula 
computing the universal coefficients $a_{V, W}(z) \in \kk\lf z \rf^\times$
for all $V, W \in \Irr \scrC$
due to Frenkel-Reshetikhin~\cite{FR99}.
First we notice the following bi-multiplicativity 
of the universal coefficients.
\begin{lemma}[{\cite[Proposition 5]{FR99}}]
\label{lemma: bilinear a}
For $P_{1}, P_{2}, P_{3} \in \fkD^{+}$, we have
\begin{align*}
a_{L(P_{1} P_{2}), L(P_{3})}(z) &= a_{L(P_{1}), L(P_{3})}(z) \cdot a_{L(P_{2}), L(P_{3})}(z), \\
a_{L(P_{1}), L(P_{2}P_{3})}(z) &= a_{L(P_{1}), L(P_{2})}(z) \cdot a_{L(P_{1}), L(P_{3})}(z).
\end{align*}
\end{lemma}

In view of Lemma~\ref{lemma: bilinear a} 
and Lemma~\ref{lemma: property ad}~(\ref{shift ad}),
it suffices to consider the fundamental case
$$a_{i,j}(z) := a_{V^{(i)}_{1}, V^{(j)}_{1}}(z) = a_{L(\pi_{i,1}), L(\pi_{j,1})}(z)$$
for $i, j \in I$. 
The following simple formula was obtained by Frenkel-Reshetikhin 
(see the discussion in Section 4.3 of \cite{FR99}).
It is based on a factorization formula of the universal R-matrix
established by Khoroshkin-Tolstoy~\cite{KT94} and also by
Damiani~\cite{Damiani98}.

\begin{theorem}[{\cite[Section 4.3]{FR99}}]
\label{theorem: univ_coeff}
For $i, j \in I$, we have
\begin{equation} \label{eq: FR}
    a_{i,j}(z) = q_{i}^{\tc_{ij}} \prod_{u=0}^{2r h^{\vee}} \left(
\frac{[u-d_{i}]}{[u + d_{i}]}
\right)^{\tc_{ij}(u)}
= q_{i}^{\tc_{ij}} \prod_{u =0}^{\infty} \left( \frac{1-\qs^{u-d_{i}}z}{1-\qs^{u+d_{i}}z} \right)^{\tc_{ij}(u)},
\end{equation}
where, for each $m \in \Z$, we set
$$[m] := (\qs^{m}z; q^{2h^{\vee}})_{\infty} = \prod_{k=0}^{\infty}(1-\qs^{m+2krh^{\vee}}z) \quad \in \kk \lf z \rf^\times.$$
\end{theorem}
Note that the second equality in \eqref{eq: FR} follows
from the $2rh^\vee$-periodicity of $\tc_{ij}(u)$
(Corollary~\ref{corollary: tc2}~(\ref{tc2_+2rh})).

\subsection{\ber Computation of $\Lambda$-invariants\er}
\label{subsection: degree}

Let us recall the definitions of the degree functions introduced in~\cite{KKOP20}.
First we define the subgroup $\mathcal{G} \subset \kk(\!( z )\!)^{\times}$ by
$$
\mathcal{G} := \left\{ cz^{m} \prod_{a\in \kk^{\times}} \varphi(az)^{\mu_{a}} \; \middle| \;
\begin{array}{c}
c\in \kk^{\times}, m \in \Z, \\
\text{$\mu_{a} \in \Z$ vanishes except for finitely many $a$'s.}
\end{array} \right\},
$$
where we set  $\varphi(z) := (z ; q^{2h^{\vee}})_{\infty} = \prod_{k =0}^{\infty}(1-q^{2kh^{\vee}}z) \in \kk[\![z]\!]^{\times}$.
Note that $\kk(z)^{\times} \subset \mathcal{G}$.
Then we define the group homomorphisms $\Deg \colon \mathcal{G} \to \Z$ and $\Deg^{\infty} \colon \mathcal{G} \to \Z$
by
\begin{align*}
\Deg(f(z))
&:= \sum_{a \in \{q^{2kh^{\vee}}  \mid k \in \Z_{\le 0}\}} \mu_{a} - \sum_{a \in \{q^{2kh^{\vee}} \mid k \in \Z_{> 0}\}} \mu_{a}, \\
\Deg^{\infty}(f(z)) &:= \sum_{a \in \{q^{2kh^{\vee}} \mid k \in \Z \}} \mu_{a}
\end{align*}
for $f(z) = cz^{m} \prod_{a \in \kk^{\times}} \varphi(az)^{\mu_{a}} \in \mathcal{G}$.
For a rational function $f(z) \in \kk(z)$, it follows
that $\Deg(f(z)) = 2 \zero_{z=1}(f(z))$ and $\Deg^{\infty}(f(z)) = 0$ (see~\cite[Lemma 3.4]{KKOP20}).

By Lemma~\ref{lemma: bilinear a} and 
Theorem~\ref{theorem: univ_coeff}, we see that
the universal coefficient $a_{V,W}(z)$ belongs to $\mathcal{G}$
for any $V, W \in \Irr \scrC$. Thus it makes sense to consider $\Deg(a_{V,W}(z))$ and $\Deg^{\infty}(a_{V,W}(z))$.
Using these degree functions, the following $\Z$-valued invariants are introduced in \cite{KKOP20}. For $V, W \in \Irr \scrC$, we define
\begin{align}
\Lambda(V,W) & \seteq \Deg(d_{V,W}(z)/a_{V,W}(z)) = 2\zero_{z=1}(d_{V,W}(z)) - \Deg(a_{V,W}(z)), \label{eq: def Lambda}\\
\Lambda^{\infty}(V,W) & \seteq \Deg^{\infty}(d_{V,W}(z)/a_{V,W}(z)) = - \Deg^{\infty}(a_{V,W}(z)). \label{eq: def Lambda^infty}
\end{align}
In the reminder of this section,
we compute $\Deg(a_{V,W}(z))$ and $\Deg^{\infty}(a_{V,W}(z))$
in terms of the skew-symmetric pairing $\scrN$ 
and the $\calQ$-weight studied in the previous section. 

\begin{theorem}
\label{theorem: Deg}
For any $(i,p), (j,s) \in \widehat{I}$, we have
\begin{equation}
\label{eq: Deg}
\Deg (a_{L(Y_{i,p}), L(Y_{j,s})}(z))
= - \scrN(i, p; j, s).
\end{equation}
In particular, for any $V, W \in \Irr \scrC^0$, we have
\begin{equation} \label{eq: Deg general}
\Deg (a_{V, W}(z))
= - \scrN(V, W).
\end{equation}
\end{theorem}

We need a lemma.
Recall that we have defined
$[m] := \varphi(\qs^{m}z)$
for each $m \in \Z$.

\begin{lemma}
\label{Lem: sum}
For any $i,j \in I$ and $x \in \Z$, we have
$$
\sum_{u=0}^{2rh^{\vee}} \tc_{ij}(u) \cdot \Deg [u-x] = \tc_{ij}(|x|) = \tc_{ij}(x) + \tc_{ij}(-x).
$$
\end{lemma}
\begin{proof}
For $x \in \Z$, we denote by $\bar{x}$ the unique integer such that $0 \le \bar{x} < 2rh^{\vee}$
and $\bar{x} - x \in 2rh^{\vee}\Z$.
Under this notation, we have
$$
\sum_{u=0}^{2rh^{\vee}} \tc_{ij}(u) \cdot \Deg [u-x] = \tc_{ij}(\bar{x}) \cdot \Deg [\bar{x} -x].
$$
By the definition of $\Deg$, we have
$$
\Deg [\bar{x} -x] = \begin{cases}
1 & \text{if $x \ge 0$}, \\
-1 & \text{if $x < 0$}.
\end{cases}
$$
When $x \ge 0$, we have $c_{ij}(\bar{x}) = c_{ij}(x)$
by Corollary~\ref{corollary: tc2}~(\ref{tc2_+2rh}).
On the other hand, when $x < 0$, we have
$
\tc_{ij}(\bar{x})  = -\tc_{ij}(2rh^{\vee} - \bar{x}) = -\tc_{ij}(-x)
$
by Corollary~\ref{corollary: tc2}~(\ref{tc2_2rh-}) and (\ref{tc2_+2rh}).
As a result, we obtain
$\tc_{ij}(\bar{x}) \cdot \Deg [\bar{x} -x] = \tc_{ij}(|x|),$
which proves the first equality of the assertion.
The second equality follows from the fact that $\tc_{ij}(x)=0$ for $x \le 0$.
\end{proof}

\begin{proof}[Proof of Theorem~\ref{theorem: Deg}]
By Theorem~\ref{theorem: univ_coeff}, we have
$$
\Deg (a_{L(Y_{i,p}), L(Y_{j,s})}(z)) = \sum_{u=0}^{2rh^{\vee}} \tc_{ij}(u)
 \cdot(\Deg [u+s-p-d_{i}]-\Deg [u+s-p+d_{i}])
$$
Applying Lemma~\ref{Lem: sum} above, we obtain the 
desired equality \eqref{eq: Deg}.
The latter identities \eqref{eq: Deg general} and 
\eqref{eq: Lambda} follow
from \eqref{eq: Deg} and Lemma~\ref{lemma: bilinear a}.
\end{proof}

\begin{definition}
For an ordered pair $(V,W)$ of (real) simple modules in $\scrC^0$,  we say that
\begin{enumerate}[{\rm (i)}]
\item the pair is \emph{left pre-commutative} if  $\zero_{z=1}(d_{V,W}(z))=0$,
\item the pair is \emph{right pre-commutative} if $\zero_{z=1}(d_{W,V}(z))=0$,
\end{enumerate}
Note that a pair $(V,W)$ is commutative
if and only if it is left and right pre-commutative
thanks to Proposition~\ref{proposition: simplicity}.
\end{definition}

\begin{corollary} \label{cor: La Lac}
For any $V, W \in \Irr \scrC^0$, we have
\begin{equation}
\label{eq: Lambda}
\Lambda(V, W) 
=  2\zero_{z=1}(d_{V,W}(z)) + \scrN(V, W)
\end{equation}
In particular, for a left pre-commutative pair $(V,W)$ 
of simple modules in $\scrC^0$, we have
$$  \Lambda(V,W) = \scrN(V,W).$$
\end{corollary}

\begin{theorem}
\label{theorem: Deg^infty}
For any $(i,p), (j,s) \in \widehat{I}$ and a {\rm Q}-datum $\calQ=(\Delta, \sigma, \uxi)$ for $\g$,
we have 
\begin{equation} \label{eq: Deg^infty}
    -\Lambda^{\infty}(L(Y_{i,p}), L(Y_{j,s})) = \Deg^{\infty} (a_{L(Y_{i,p}), L(Y_{j,s})}(z)) = (-1)^{k+l}(\alpha, \beta),
\end{equation}
where $\bar{\phi}_{\calQ}(i,p)=(\alpha, k)$ and $\bar{\phi}_{\calQ}(j,s)=(\beta, l)$.
In particular, for any $V, W \in \Irr \scrC^0$, we have
\begin{equation} \label{eq: Deg^infty general}
    \Lambda^{\infty}(  V,W  ) = - (  \wt_\calQ(V)  , \wt_\calQ(W)).
\end{equation}
\end{theorem}
\begin{proof}
By Theorem~\ref{theorem: univ_coeff}, we have
\begin{align*}
\Deg^{\infty} (a_{L(Y_{i,p}), L(Y_{j,s})}(z)) &= \sum_{u=0}^{2rh^{\vee}} \tc_{ij}(u)
 \cdot(\Deg^{\infty} [u+s-p-d_{i}]-\Deg^{\infty} [u+s-p+d_{i}]) \\
&= \tc_{ij}(\overline{-s+p+d_{i}}) - \tc_{ij}(\overline{-s+p-d_{i}}),
\end{align*}
where $\bar{x}$ denotes for each $x \in \Z$ the unique integer such that
$0 \le \bar{x} < 2rh^{\vee}$ and $\bar{x}-x \in 2rh^{\vee}\Z$
as before.
Let us choose
$\im, \jm \in \Delta_{0}$ such that $\bar{\im}=i, \bar{\jm}=j$.
By Theorem~\ref{theorem: combinatorial formula}, we have
$$
\tc_{ij}(\bar{x}) = (\varpi_{\im}, \tau_{\calQ}^{(x+\uxi_{\imath}-\uxi_{\jmath}-d_{i})/2}(\gamma_{\jm}^{\calQ}))
$$
for each $x \in \Z$.
By the same computation as in the proof of Proposition~\ref{proposition: gamma wt pairing same}, we obtain
$$
\tc_{ij}(\overline{-s+p+d_{i}}) - \tc_{ij}(\overline{-s+p-d_{i}})
= (\tau_{\calQ}^{(\uxi_{\im}-p)/2}(\gamma_{\im}^{\calQ}),
        \tau_{\calQ}^{(\uxi_{\jm}-s)/2}(\gamma_{\jm}^{\calQ})) = (-1)^{k+l} (\alpha, \beta),
$$
which proves \eqref{eq: Deg^infty}.
The equality \eqref{eq: Deg^infty general}
follows from \eqref{eq: Deg^infty} and Lemma~\ref{lemma: bilinear a}.
\end{proof}



\begin{remark} \label{remark: KKOP20b}
Theorem~\ref{theorem: Deg^infty} gives a new unified proof of~\cite[(5.2)]{KKOP20b} for all untwisted affine types,
which was a crucial step in the proof of~\cite[Theorem 3.6]{KKOP20b}.
In particular, it leads to a block decomposition of the category 
$\scrC^0$ indexed by the simply-laced root lattice $\sfQ$
by \cite[Section 4]{KKOP20b},
which is comparable with Theorem~\ref{Thm: CM} above.
\end{remark}

\begin{corollary}[cf.~{\cite[(3.6)]{KKOP20b}}]
For simple modules $V$ and $W$ in $\scrC^0$, we have
$$  \scrN(\scrD^{\ber 2k \er}V,W) = 
\scrN(V,\scrD^{\ber -2k \er}W) =   \Lambda^\infty(V,W) \quad \text{ for }k \gg 0.$$
\end{corollary}

\begin{proof}
Let $V \cong L(m)$ and $W \cong L(m')$ 
with $m,m' \in \calM_+$.
The first equality follows from Proposition~\ref{proposition: dual}
and the fact $\bfD^2(Y_{i,p}) = Y_{i, p+2rh^\vee}$.
Choose $k \in \Z$ large enough so that we have
$$
\min\{p \mid \exists i \in I, u_{i,p}(\bfD^{2k}m) \neq 0\}
>
\max\{p \mid \exists i \in I, u_{i,p}(m') \neq 0\}.
$$
Then the second equality follows from 
Lemma~\ref{lemma: dual wt},
Corollary~\ref{corollary: scrN wt}
and \eqref{eq: Deg^infty general}.
\end{proof}

\appendix

\section{Denominator formulae}
\label{section: denominator formula}

In this appendix, we give a list of all the formulae of $d_{i^{l}, j^{m}}(z)$ which are currently known.
These are quoted from~\cite{Oh15R, OhS19, OhS19s, Fujita19}.
See Remark~\ref{remark: convention} for our convention and \beb
Remark~\ref{rmk: F4} for type $\mathsf{F}_4$. \eb
For non-simply-laced $\g$, we use the labeling $I=\{1, \ldots, n\}$ as in~(\ref{eq: diagram foldings}).

\subsection*{(Type $\mathsf{ADE}$)}
For any $i,j \in I$, we have
\begin{equation}
\label{eq: formulaADEfund}
d_{i,j}(z) = \prod_{u=0}^{h^{\vee}} (z-q^{u+1})^{\tc_{ij}(u)}.
\end{equation}
If $\g$ is of type $\mathsf{AD}$ and $l,m \in \Z_{\ge 1}$, we have
\begin{equation}
\label{eq: formulaADBll}
d_{i^{l}, j^{m}}(z) \equiv \prod_{t=0}^{\min(l, m) -1} d_{i,j}(q^{-|l-m|-2t}z).
\end{equation}

\subsection*{(Type $\mathsf{B}_{n}$)}
For $1 \le i,j <n$ and $l,m \in \Z_{\ge 1}$, we have
\begin{equation}
d_{i,j}(z) = \prod_{u=0}^{\min(i,j)} (z-q^{|i-j|+2u})(z-q^{2n-i-j-1+2u})
\end{equation}
and $d_{i^{l}, j^{m}}(z)$ are given by the same formula as~(\ref{eq: formulaADBll}).
For $1 \le i < n$ and $l, m \in \Z_{\ge 1}$, we have
\begin{align}
\label{eq: formulaBls}
d_{i^{l}, n^{m}}(z) &= d_{n^{m}, i^{l}}(z) = \prod_{t=0}^{\min(2l, m)-1} \prod_{u=1}^{i} ( z - \qs^{2n -2i -2 + |2l - m| +4u + 2t} ), \allowdisplaybreaks \\
\label{eq: formulaBss}
d_{n^{l}, n^{m}}(z) &\equiv \prod_{t=0}^{\min(l, m) -1} d_{n,n}(\qs^{-|l-m| -2t}z).
\end{align}

\subsection*{(Type $\mathsf{C}_{n}$)}
For $1 \le i, j < n$ and $l, m\in \Z_{\ge 1}$ with $\max(l, m) > 1$, we have
\begin{align}
\label{eq: formulaCss}
d_{i^{l}, j^{m}} (z) &= \prod_{t=0}^{\min(l, m) -1} \prod_{u=1}^{\min(i,j)} (z-\qs^{|i-j|+|l-m|+2(t+u)}) (z-\qs^{2n+2-i-j+|l-m|+2(t+u)}), \allowdisplaybreaks\\
\label{eq: formulaCsl}
d_{i^{l}, n^{m}}(z) &= d_{n^{m}, i^{l}}(z) = \prod_{t=0}^{\min(l, 2m)-1} \prod_{u=1}^{i}(z-\qs^{n+1-i+|2m-l|+2t+2u}), \allowdisplaybreaks\\
\label{eq: formulaCll}
d_{n^{l}, n^{m}}(z) &= \prod_{t=0}^{\min(l, m)-1} \prod_{u=1}^{n} (z-\qs^{2+|2l-2m|+2u+4t}),
\end{align}
while, for $1 \le i,j \le n$, we have
\begin{equation}
\label{eq: formulaCfund}
d_{i,j}(z) = \prod_{u=1}^{\min(i, j, n-i, n-j)}(z-\qs^{|i-j|+2u}) \prod_{u=1}^{\min(i,j)}(z-\qs^{2n+2-i-j+2u}).
\end{equation}

\subsection*{(Type $\mathsf{F}_{4}$)} We have \beb
\begin{align}
\label{eq: formulaF11} d_{1,1}(z) &= (z-\qs^{4})(z-\qs^{10})(z-\qs^{12})(z-\qs^{18}), \allowdisplaybreaks \\
\label{eq: formulaF12} d_{1,2}(z) &= (z-\qs^{6})(z-\qs^{8})(z-\qs^{10})(z-\qs^{12})(z-\qs^{14})(z-\qs^{16}), \allowdisplaybreaks\\
\label{eq: formulaF13} d_{1,3}(z) &= (z-\qs^{7})(z-\qs^{9})(z-\qs^{13})(z-\qs^{15}), \allowdisplaybreaks\\
\label{eq: formulaF14} d_{1,4}(z) &= (z-\qs^{8})(z-\qs^{14}), \allowdisplaybreaks \\
\label{eq: formulaF22} d_{2,2}(z)
&= (z-\qs^{4})(z-\qs^  {6})(z-\qs^{8})^{2}(z-\qs^{10})^{2}(z-\qs^{12})^{2}(z-\qs^{14})^{2}(z-\qs^{16})(z-\qs^{18}), \allowdisplaybreaks\\
\label{eq: formulaF23} d_{2,3}(z) &= (z-\qs^{5})(z-\qs^{7})(z-\qs^{9})(z-\qs^{11})^{2}(z-\qs^{13})(z-\qs)^{15}(z-\qs^{17}), \allowdisplaybreaks\\
\label{eq: formulaF24} d_{2,4}(z) &= (z-\qs^{6})(z-\qs^{10})(z-\qs^{12})(z-\qs^{16}), \allowdisplaybreaks\\
\label{eq: formulaF33} d_{3,3}(z) &= (z-\qs^{2})(z-\qs^{6})(z-\qs^{8})(z-\qs^{10})(z-\qs^{12})^{2}(z-\qs^{16})(z-\qs^{18}),\allowdisplaybreaks \\
\label{eq: formulaF34} d_{3,4}(z) &= (z-\qs^{3})(z-\qs^{7})(z-\qs^{11})(z-\qs^{13})(z-\qs^{17}),\allowdisplaybreaks \\
\label{eq: formulaF44} d_{4,4}(z) &= (z-\qs^{2})(z-\qs^{8})(z-\qs^{12})(z-\qs^{18}).
, \allowdisplaybreaks\\
\label{eq: formulaF3s1}
d_{3^2,1}(z) & = (z-\qs^{6})(z-\qs^{8})^2(z-\qs^{10})(z-\qs^{12})(z-\qs^{14})^2(z-\qs^{16})
\allowdisplaybreaks\\
\label{eq: formulaF3s2}
d_{3^2,2}(z) &= (z-\qs^{4})(z-\qs^{6})^2 (z-\qs^{8})^{2} (z-q_s^{10})^{3}
\ber (z-\qs^{12})^3 \er (z-\qs^{14})^2(z-\qs^{16})^2(z-\qs^{18}),  \allowdisplaybreaks\\
\label{eq: formulaF3s3}
d_{3^2,3}(z)& = (z-\qs^{3})(z-\qs^{5})(z-\qs^{7})(z-\qs^{9})^{2}(z-\qs^{11})^2(z-\qs^{13})^2(z-\qs^{15})(z-\qs^{17})(z-\qs^{19}), \allowdisplaybreaks\\
\label{eq: formulaF3s4}
d_{3^2,4}(z) & = (z-\qs^{4}) (z-\qs^{8})(z-\qs^{10})(z-\qs^{12})(z-\qs^{14})(z-\qs^{18}), 
\allowdisplaybreaks\\
\label{eq: formulaF4s1}
d_{4^2,1}(z) &= (z-\qs^7)(z-\qs^9)(z-\qs^{13})(z-\qs^{15}), 
\allowdisplaybreaks\\
\label{eq: formulaF4s2}
d_{4^2,2}(z)&=  (z-\qs^{5})(z-\qs^{7})(z-\qs^{9})(z-\qs^{11})^2(z-\qs^{13})(z-\qs^{15})(z-\qs^{17}), \allowdisplaybreaks\\
\label{eq: formulaF4s3}
d_{4^2,3}(z)&= (z-\qs^{4})(z-\qs^{8})(z-\qs^{10})(z-\qs^{12})(z-\qs^{14})(z-\qs^{18}), 
\allowdisplaybreaks\\
\label{eq: formulaF4s4}
d_{4^2,4}(z)& = (z-\qs^3)(z-\qs^9)(z-\qs^{13})(z-\qs^{19}).
\end{align}
\eb

\subsection*{(Type $\mathsf{G}_{2}$)}
We have
\begin{align}
d_{1,1}(z) &= (z-\qs^{6})(z-\qs^{8})(z-\qs^{10})(z-\qs^{12}), \qquad   d_{1,2}(z) = (z-\qs^{7})(z-\qs^{11}),
\end{align}
and for $l,m \in \Z_{\ge 1}$, we have
\begin{align}
\label{eq: formulaGll}
d_{1^{l}, 1^{m}}(z) &\equiv \prod_{t=1}^{\min(l, m) -1} d_{1,1}(q^{-|l-m|-2t}z), \\
\label{eq: formulaGls}
d_{1^{l}, 2^{m}}(z) &= d_{2^{m}, 1^{l}}(z) \equiv \prod_{t=0}^{\min(3l, m)-1} d_{1,2}(\qs^{-|3l - m| + 2 - 2t}z).
\end{align}
For $l, m \in \Z_{\ge 1}$ with $(l, m) \neq (1,1), (2,2)$, we have
\begin{align}
\label{eq: formulaGss}
d_{2^{l}, 2^{m}}(z) &= \prod_{t=0}^{\min(l, m)-1} \prod_{u=1}^{2}(z-\qs^{-2+|l-m|+4u+2t})(z-\qs^{4+|l-m|+4u+2t}),
\end{align}
while
\begin{align}
\label{eq: formulaG2121}
d_{2,2}(z) &= (z-\qs^{2})(z-\qs^{8})(z-\qs^{12}), \\
\label{eq: formulaG2222}
d_{2^{2}, 2^{2}}(z) &= (z-\qs^{2})(z-\qs^{4})(z-\qs^{8})^{2}(z-\qs^{10})(z-\qs^{12})(z-\qs^{14}).
\end{align}


\end{document}